\documentclass[a4paper,11pt]{article}
\usepackage[utf8x]{inputenc}

\usepackage{a4wide}
\usepackage{amssymb}
\usepackage{amsmath}
\usepackage{amsthm}
\usepackage[mathscr]{euscript}
\usepackage[normalem]{ulem}
\usepackage[colorlinks=true]{hyperref}
\usepackage[usenames,dvipsnames]{xcolor}
\usepackage{appendix}
\usepackage{tikz}
\usetikzlibrary{arrows,angles}
\usetikzlibrary{matrix}
\usetikzlibrary{decorations}
\usetikzlibrary{arrows,calc,shapes,decorations.pathreplacing}
\usepackage{tikz-cd}
\usepackage{todonotes}

\usepackage{mathabx} 
\usepackage[style=numeric, giveninits=true, sorting=nyt, maxbibnames=99]{biblatex}
\DeclareNameAlias{default}{family-given}
\numberwithin{equation}{section}

\theoremstyle{definition}

\newtheorem{Definition}{Definition}[section]

\newtheorem{Remark}[Definition]{Remark}
\newtheorem{Example}[Definition]{Example}

\theoremstyle{plain}

\newtheorem{Theorem}[Definition]{Theorem}
\newtheorem{Proposition}[Definition]{Proposition}
\newtheorem{Lemma}[Definition]{Lemma}
\newtheorem{Corollary}[Definition]{Corollary}

\newcommand{\R}{\mathbb R}
\newcommand{\N}{\mathbb N}

\newcommand{\eps}{\varepsilon}

\usepackage[xcolor]{changebar}
\cbcolor{blue}
\newcommand{\Ric}{\mathrm{Ric}}

\newcommand{\m}{\mathfrak m}

\usepackage[inline]{enumitem}
\newcommand{\enumlabelformat}{\roman}
\newcommand{\enumlabelfont}[1]{#1}
\newlength{\thelabelsep}
\setlist{labelsep=\thelabelsep}
\setlist[enumerate,1]{font=\enumlabelfont,label=(\enumlabelformat*),leftmargin=2.10em}
\setlist[itemize]{leftmargin=2.10em,label=$-$}

\newcounter{inlineenum}
\renewcommand{\theinlineenum}{\enumlabelformat{inlineenum}}

\makeatletter
\let\save@mathaccent\mathaccent
\newcommand*\if@single[3]{%
  \setbox0\hbox{${\mathaccent"03102{#1}}^H$}%
  \setbox2\hbox{${\mathaccent"03102{\kern0pt#1}}^H$}%
  \ifdim\ht0=\ht2 #3\else #2\fi
  }
\newcommand*\rel@kern[1]{\kern#1\dimexpr\macc@kerna}
\newcommand*\wide@bar[2]{\if@single{#1}{\wide@bar@{#1}{#2}{1}}{\wide@bar@{#1}{#2}{2}}}
\newcommand*\wide@bar@[3]{%
  \begingroup
  \def\mathaccent##1##2{%
    \let\mathaccent\save@mathaccent
    \if#32 \let\macc@nucleus\first@char \fi
    \setbox\z@\hbox{$\macc@style{\macc@nucleus}_{}$}%
    \setbox\tw@\hbox{$\macc@style{\macc@nucleus}{}_{}$}%
    \dimen@\wd\tw@
    \advance\dimen@-\wd\z@
    \divide\dimen@ 3
    \@tempdima\wd\tw@
    \advance\@tempdima-\scriptspace
    \divide\@tempdima 10
    \advance\dimen@-\@tempdima
    \ifdim\dimen@>\z@ \dimen@0pt\fi
    \rel@kern{0.10}\kern-\dimen@
    \if#31
      \overline{\rel@kern{-0.10}\kern\dimen@\macc@nucleus\rel@kern{0.4}\kern\dimen@}%
      \advance\dimen@0.4\dimexpr\macc@kerna
      \let\final@kern#2%
      \ifdim\dimen@<\z@ \let\final@kern1\fi
      \if\final@kern1 \kern-\dimen@\fi
    \else
      \overline{\rel@kern{-0.10}\kern\dimen@#1}%
    \fi
  }%
  \macc@depth\@ne
  \let\math@bgroup\@empty \let\math@egroup\macc@set@skewchar
  \mathsurround\z@ \frozen@everymath{\mathgroup\macc@group\relax}%
  \macc@set@skewchar\relax
  \let\mathaccentV\macc@nested@a
  \if#31
    \macc@nested@a\relax111{#1}%
  \else
    \def\gobble@till@marker##1\endmarker{}%
    \futurelet\first@char\gobble@till@marker#1\endmarker
    \ifcat\noexpand\first@char A\else
      \def\first@char{}%
    \fi
    \macc@nested@a\relax111{\first@char}%
  \fi
  \endgroup
}
\makeatother

\newcommand*\diff{\mathop{}\!\mathrm{d}}

\newcommand{\loc}{\mathrm{loc}}
\DeclareMathOperator{\supp}{supp}

\title{The Hawking and Penrose singularity theorems\\ for low regularity Finsler spacetimes}

\author{Darius Er\"os\thanks{Department of Mathematics, University of Vienna, Oskar-Morgenstern-Platz 1, 1090 Vienna, Austria (\textsf{darius.eroes@univie.ac.at})}\\
Ettore Minguzzi\thanks{Dipartimento di Matematica, Università degli Studi di Pisa, Via B. Pontecorvo 5, I-56127
Pisa, Italy (\textsf{ettore.minguzzi@unipi.it})}\\
Argam Ohanyan\thanks{Department of Mathematics, University of Toronto, 45 St. George Street, M5S 2E5 Toronto, Ontario, Canada (\textsf{argam.ohanyan@utoronto.ca})}\\
Shin-ichi Ohta\thanks{
Department of Mathematics, University of Osaka, Osaka 560-0043, Japan
(\textsf{s.ohta@math.sci.osaka-u.ac.jp});
RIKEN Center for Advanced Intelligence Project (AIP),
1-4-1 Nihonbashi, Tokyo 103-0027, Japan}
}
\begin{document}

\date{\today}


\maketitle

\begin{abstract}
We prove the Hawking and Penrose singularity theorems for Lorentz--Finsler structures of mixed horizontal/vertical regularity $C^{(1,2)}$ in the $C^1$ weighted case and $C^{(1,3)}$ in the unweighted case.
At this regularity, geodesics need not be uniquely determined by their initial conditions, and the (weighted) Ricci curvature must be interpreted distributionally.
Our approach relies on a two-stage approximation procedure: first, a homogeneity-preserving convolution, which is of independent interest in the broader semi-Riemann--Finsler setting, and second, a causality-adapted approximation à la Chru{\'s}ciel--Grant.
In the special case of $C^1$ Lorentzian metrics, our results extend the $C^1$ singularity theorems of Graf to the $C^1$ weighted case.

\vspace{1em}

\medskip


\end{abstract}
\tableofcontents
\section{Introduction}

The study of nonsmooth spacetimes has garnered considerable interest in recent years. 
Central questions about the nature of singularities in general relativity and extendibility of spacetimes in the context of the landmark singularity theorems of Hawking and Penrose necessitate the study of a wider range of geometries under reduced smoothness assumptions. 
Currently, the class of $C^1$ Lorentzian metrics is the one which is best understood in terms of singularity \cite{Graf-SingularityC1, KOSS, SchinnerlSteinbauer} and rigidity \cite{KOV, quintetnonsmooth} theorems, as well as recently emergent synthetic descriptions of energy conditions based on optimal transport \cite{BraunCalisti}.

In the smooth case, Lorentz--Finsler geometry provides a physically meaningful extension of Lorentzian geometry which incorporates an additional directional dependence of the metric tensor and its derived curvature quantities. 
Finsler spacetimes have been considered both as a 
geometric description of gravity beyond general relativity (see, e.g., \cite{BeJaSa20, KoStSt09, Pfeife19}) and in the classical limit of modified dispersion relations
encompassing Lorentz violation in quantum gravity and in the Standard Model Extension (see, e.g., \cite{GiLiSi06, Vacaru11, Kostel11, KoRuTs12, Russel15, Collad17}). 
Remarkably, they provide a counter-example to the Shiff conjecture \cite{LaPeHa12}, but they are compatible with a slight modification of Ehlers--Pirani--Schild axioms for general relativity (see \cite{LaePer18, BeJaSa20}). 

Previous work on smooth Finsler spacetimes covers singularity \cite{MinguzziRaychaudhuri, AJ16, LuMinguzziOhtaSingularity}, rigidity \cite{LuMinguzziOhtaSplitting, CaponioOhanyanOhta}, and comparison \cite{LuMinguzziOhtaComparison} theorems, as well as optimal transport \cite{BraunOhta}.
However, so far, there is no comparison-geometric treatment of \emph{low regularity} Lorentz--Finsler structures. Their causality theory is well captured and understood within the framework of cone structures due to the second author \cite{Minguzziconestructure}, but the effects of curvature bounds on their global geometry have not yet been investigated. 
The aim of this article is to initiate this study with the goal of establishing the Hawking and Penrose singularity theorems for this class of spacetimes.
A major difference from the purely Lorentzian setting is the additional directional dependence of the curvature quantities. 
Moreover, in the low regularity formulation, one can incorporate additional nonsmoothness in the vertical (tangent) directions. 

Many of our results for nonsmooth Finsler structures, their curvatures, and approximations, are signature-independent; see Remark \ref{Remark: general semi Riemann}. 
Hence, they can also be used in the study of Riemann--Finsler structures as well as more general semi-Riemann--Finsler structures of low regularity. 
Moreover, we expect that relaxing the regularity assumptions can contribute to the construction of interesting classes of both Riemann-- and Lorentz--Finsler manifolds.

We now outline our program and present our main results.

\begin{Definition}[$C^{(k,l)}$ Lorentz--Finsler structures]
Let $M$ be a smooth manifold and $k,l \in \N \cup \{0,+\infty\}$ with $k + l \geq 2$. A \emph{Lorentz--Finsler structure of regularity $C^{(k,l)}$} is a map $L: TM \to \R$ such that the following hold:
\begin{enumerate}
    \item For each fiberwise linear chart $(x^i,v^i)$ on $TM \setminus \mathbf 0$, $\partial_x^\beta \partial_v^\sigma L$ exists and is continuous for all multi-indices $\beta,\sigma$ satisfying $|\beta| \leq k$ and $|\sigma| \leq k + l - |\beta|$;
    \item $L(cv) = c^2 L(v)$ for all $c > 0$ and $v \in TM$;
    \item For all $v \in TM \setminus \mathbf 0$, the vertical Hessian $g_{ij}(v):=\partial_{v^i} \partial_{v^j} L$ has Lorentzian signature $(-,+,\dots,+)$.
\end{enumerate}
\end{Definition}

General low regularity semi-Riemann--Finsler structures (in particular, positive definite Riemann--Finsler ones) are defined in complete analogy, changing the signature in (iii) accordingly. 
The purpose of condition (i) above is to provide a coordinate-invariant way of phrasing the condition ``jointly $C^k$ in $x$ and $C^l$ in $v$" for a map on the tangent bundle $TM$; see the discussion in Subsection \ref{subsection: mixed regularity of maps}. 
The majority of our results will assume $k = 1$ and $l = 2$; in particular, it covers the case of $C^1$ Lorentzian metrics $g$, which naturally induce Lorentz--Finsler structures
\begin{equation*}
    L(v):=\frac{1}{2} g(v,v)
\end{equation*}
of class $C^{(1,\infty)}$.

In the setting most relevant to us, namely $L \in C^{(1,2)}$ or $L \in C^{(1,3)}$, many basic properties of smooth Lorentz--Finsler geometry fail to hold.
In particular, geodesics are, in general, no longer uniquely determined by their initial conditions, and curvature quantities can only be defined in a distributional sense, similar to the $C^1$ Lorentzian case \cite{Graf-SingularityC1, KOSS, SchinnerlSteinbauer, quintetnonsmooth}.

To study the geometry and causality induced by a nonsmooth Lorentz--Finsler structure $L$, we therefore argue by approximation. 
Adapting from the $C^1$ Lorentzian case \cite{Graf-SingularityC1, KOSS}, we construct two types of smooth Lorentz--Finsler structures: the convolution-based approximation $L_\varepsilon$, as well as the causality-adapted approximations $\widecheck L_\varepsilon$ and $\widehat L_\varepsilon$.
We obtain $L_\varepsilon$ through a \emph{homogeneity-preserving} convolution in local charts (see Subsection \ref{ssc:L_e}) which works irrespective of the signature of $L$ (see Remark \ref{Remark: general semi Riemann}) and is therefore of independent interest in the more general framework of low regularity semi-Riemann--Finsler geometries.
At this stage, however, we lose control of any causality relation between $L_\eps$ and $L$.
To remedy this deficiency, we locally align the future cones of $L_\varepsilon$ and $L$ (see Proposition \ref{prop:ChruscielGrantApprox}) and further construct two families of (inner and outer) approximations $\widecheck L_\eps$ and $\widehat L_\eps$ in the spirit of Chru\'sciel--Grant (cf.\ \cite{ChruscielGrant}) that satisfy 
\[
\widecheck L_\varepsilon \prec L \prec \widehat L_\varepsilon,
\]
where $\prec$ indicates the nesting of the future cones (see Definition \ref{df:order}).
These have the same analytic approximation properties as $L_\varepsilon$ (in fact, $\widecheck L_\varepsilon - L_\varepsilon \to 0$ in $C^\infty_{\mathrm{loc}}(TM \setminus \mathbf{0})$, same for $\widehat L_\varepsilon$), but their causality is closely tied to that of $L$, which is crucial for the proof of singularity theorems.

The energy conditions for $L$, which are key assumptions in the singularity theorems, are formulated in a distributional sense on $TM$. 
This is in sharp contrast to the purely Lorentzian case, in which distributional curvature can still be evaluated along smooth vector fields (in particular, in timelike directions) and which can therefore be pulled back to a distribution on $M$. 
Just as in the Lorentzian case, however, our distributional energy conditions on $TM$ imply surrogate energy conditions for the approximations $\widecheck L_\varepsilon$, which can be treated in a standard way. 
In the case of the strong energy condition, with $\Omega^T \subseteq TM$ denoting the set of future timelike vectors, this takes the following form (see Subsections \ref{ssc:curvature}, \ref{ssc:convergence}):

\begin{Definition}[Distributional strong energy condition]
We say that a $C^{(1,2)}$ Finsler spacetime $(M,L)$ satisfies the \emph{distributional strong energy condition} if $\Ric \geq 0$ on $\Omega^T$ in the distributional sense, i.e., $\langle \Ric, \mu\rangle \geq 0$ for any nonnegative volume density $\mu$ on $TM$ compactly supported in $\Omega^T$.
Similarly, we say that $\Ric \geq K$ \emph{timelike distributionally} for $K \in \R$ if the distribution $\Ric + 2 K L$ is nonnegative.
\end{Definition}

\begin{Proposition}[Approximation of strong energy condition]
    Let $(M,L)$ be a $C^{(1,3)}$ Finsler spacetime satisfying the distributional strong energy condition, and let $\mathcal{K} \subseteq \Omega^T$ be compact.
    Then for any $\delta >0$, there exists $\eps_0 =\eps_0(\mathcal{K},\delta) >0$ such that for all $\eps \in (0,\eps_0)$ and $v \in \mathcal{K}$ with $2 \widecheck{L}_\eps(v)=-1$, we have
    \begin{equation*}
        \Ric[\widecheck{L}_\eps](v) \ge - \delta.
    \end{equation*}
\end{Proposition}
Interestingly, favorable cancellations of certain terms allow us to prove an analogous approximation result in the weighted case for the less regular $C^{(1,2)}$ class (see Remark \ref{rmk:WeightedRicciSimplified}, Proposition \ref{Prop: approx weighted energy conditions}), which is the reason for the discrepancy in regularity in the unweighted and weighted singularity theorems below.

Let us now present our main results, which are the two versions of the Hawking singularity theorem as well as the Penrose singularity theorem, for $L \in C^{(1,3)}$ in the unweighted case and for $L \in C^{(1,2)}$ in the case where the Finsler spacetime is equipped with a $C^1$ reference measure. 
We refer to Section \ref{section: singularity theorems} for notions that appear below and which we have not discussed so far.

\begin{Theorem}[Unweighted $C^{(1,3)}$ Hawking singularity theorem,
version 1]
Let $(M,L)$ be a globally hyperbolic $C^{(1,3)}$
Finsler spacetime satisfying
$\operatorname{Ric}\geq K$ timelike distributionally
for some $K\in\mathbb R$.
Suppose that $\Sigma\subseteq M$ is a $C^2$ spacelike
Cauchy hypersurface that is $\beta$-contracting,
where $\beta>0$ and additionally
$\beta>\sqrt{(n-1)|K|}$ if $K<0$. Then
\begin{equation*}
    \sup_M d_\Sigma\leq T_{K,\beta},
\end{equation*}
where 
\begin{equation*}
    T_{K,\beta}:=
    \begin{cases}
        \dfrac{n-1}{\beta},
        & K=0,\\[1.2ex]
        \sqrt{\dfrac{n-1}{K}}\,
        \operatorname{arctan}\!
        \left(\dfrac{\sqrt{(n-1)K}}{\beta}\right),
        & K>0,\\[1.6ex]
        \sqrt{\dfrac{n-1}{|K|}}\,
        \operatorname{arctanh}\!
        \left(\dfrac{\sqrt{(n-1)|K|}}{\beta}\right),
        & K<0.
    \end{cases}
\end{equation*}
In particular, $(M,L)$ is future timelike geodesically
incomplete.
\end{Theorem}

\begin{Theorem}[Unweighted $C^{(1,3)}$ Hawking singularity theorem, version 2]
Let $(M,L)$ be a $C^{(1,3)}$ Finsler spacetime satisfying the distributional strong energy condition, $\Sigma \subseteq M$ a $C^2$ compact spacelike hypersurface that is contracting. 
Then $(M,L)$ is future timelike geodesically incomplete.
\end{Theorem}

\begin{Theorem}[Unweighted $C^{(1,3)}$ Penrose singularity theorem]
Let $(M,L)$ be a globally hyperbolic $C^{(1,3)}$ Finsler spacetime which satisfies the distributional null energy condition and contains a $C^2$ trapped surface $S$ and a noncompact Cauchy hypersurface. 
Then there exists a future null geodesic $\eta$ emanating orthogonally from $S$ that is future incomplete.
\end{Theorem}

\begin{Theorem}[$C^1$ weighted $C^{(1,2)}$ Hawking singularity theorem, version 1]
Let $(M,L)$ be a globally hyperbolic $C^{(1,2)}$ Finsler spacetime, equipped with a $C^1$ measure $\mathfrak m$, satisfying 
\[
\Ric^N_\m \geq K\exp\biggl( \frac{4(\kappa-1)}{n-1}\Psi \biggr)
\]
timelike distributionally for some $K\in\R$, $N \in (-\infty,1] \cup [n,\infty]$ and $\kappa$ timelike admissible for $N$.
Suppose that $\Sigma\subseteq M$ is a $C^2$ spacelike Cauchy hypersurface with $\theta_\kappa <-\beta$, where $\beta>0$ and, additionally, $\beta>\sqrt{|K|/c_{N,\kappa}}$ when $K<0$.
Then we have
\[
\sup_\eta \lim_{t \uparrow T} \varphi_\eta(t) \leq T^{(N,\kappa)}_{K,\beta},
\]
where $\eta:[0,T) \to M$ runs over all future unit speed timelike geodesics maximizing the distance from $\Sigma$ and we set 
\[
    T^{(N,\kappa)}_{K,\beta}
    :=
    \begin{cases}
        \dfrac{1}{c_{N,\kappa}\beta},
        & K=0,\\[1.2ex]
        \dfrac{1}{\sqrt{c_{N,\kappa}K}}\,
        \arctan\!\left(
            \dfrac{1}{\beta}\sqrt{\dfrac{K}{c_{N,\kappa}}}
        \right),
        & K>0,\\[1.6ex]
        \sqrt{\dfrac{1}{c_{N,\kappa}|K|}}\,
        \operatorname{arctanh}\!\left(
            \dfrac{1}{\beta}\sqrt{\dfrac{|K|}{c_{N,\kappa}}}
        \right),
        & K<0.
    \end{cases}
\]
In particular, every such $\eta$ is future $\kappa$-incomplete and, when $N \in [n,\infty)$ and $\kappa=1$, we have $\sup_M d_\Sigma \le T^{(N,1)}_{K,\beta}$ and $(M,L)$ is future timelike geodesically incomplete.
\end{Theorem}

\begin{Theorem}[$C^1$ weighted $C^{(1,2)}$ Hawking singularity theorem, version 2]
Let $(M,L,\mathfrak m)$ be a $C^{(1,2)}$ Finsler spacetime equipped with a $C^1$ measure $\mathfrak m$ satisfying $\Ric^N_{\mathfrak m}\geq0$ timelike distributionally for some $N\in(-\infty,1]\cup[n,\infty]$.
Suppose that $\Sigma\subseteq M$ is a compact $C^2$ spacelike hypersurface such that $\theta_\kappa<0$ for some $\kappa$ which is timelike admissible for $N$.

Then there exists a future inextendible unit speed timelike geodesic $\eta:[0,T)\to M$, starting orthogonally to $\Sigma$, that is future $\kappa$-incomplete.
In particular, when $N\in[n,\infty)$ and $\kappa=1$, this geodesic is future incomplete in the usual sense.
\end{Theorem}

\begin{Theorem}[$C^1$ weighted $C^{(1,2)}$ Penrose singularity theorem]
Let $(M,L,\mathfrak m)$ be a globally hyperbolic $C^{(1,2)}$ Finsler spacetime of dimension $n\geq3$, equipped with a $C^1$ measure  $\mathfrak m$, satisfying the distributional null $N$-energy condition for some $N\in(-\infty,2]\cup[n,\infty]$.
Suppose that $M$ contains a noncompact Cauchy hypersurface and a compact orientable $C^2$ spacelike submanifold $S$ of codimension $2$ satisfying $\theta_\kappa^\pm<0$ for some $\kappa$ which is null admissible for $N$. 

Then there exists a future inextendible null geodesic $\eta:[0,T)\to M$, starting orthogonally to $S$, which is future $\kappa$-incomplete.
In particular, if $N\in[n,\infty)$ and $\kappa=1$, $\eta$ is future incomplete in the usual sense.
\end{Theorem}

For all of the singularity theorems stated above, time reversal gives past distance bounds or past incompleteness; see Corollary \ref{cr:past-sing} and Remark \ref{remark: past versions of singularity theorems} for details.

Let us remark that, in the case of $C^1$ Lorentzian metrics, our unweighted singularity theorems recover the ones established by Graf in \cite{Graf-SingularityC1}, while the weighted theorems provide genuine extensions of those even for Lorentzian manifolds. 
A common feature of both approaches is that branching of geodesics is not excluded, while the synthetic singularity theorems using optimal transport (see \cite{CavallettiMondino, CavallettiManiniMondinosynthetic}) require nonbranching.
As for low regularity Finsler structures, it is unclear if curvature bounds imply nonbranching (cf.\ \cite{OhtaBranching, MagnaboscoCDbranching}), hence our theorems would not be subsumed by their synthetic counterparts \cite{CavallettiMondino, CavallettiManiniMondinosynthetic}.

\subsection{Notation and conventions}

We set $\N_0 :=\N \cup \{0\}$. 
All manifolds are assumed to be smooth, connected, paracompact, Hausdorff, and without boundary. 
We denote by $\tilde g$ a fixed complete Riemannian background metric on $M$. 
We denote by $\mathbf 0 \subseteq TM$ the zero section and by $\pi:TM \to M$ the canonical projection. 
Throughout, we will make use of the Einstein summation convention.
\subsection{Outlook}

Our work opens the possibility of further investigation of low regularity Finsler structures from various angles. 
For Lorentz--Finsler structures, the Hawking--Penrose singularity theorem \cite{KOSS} and the splitting theorem \cite{quintetnonsmooth} in $C^{(1,3)}$ regularity (or in $C^{(1,2)}$ regularity with $C^1$ weights) are natural next goals. 
It would further be interesting to investigate the links between our distributional energy conditions and corresponding bounds expressed via optimal transport \cite{BraunCalisti, braunsalamocandal}. 
Such bounds are likewise relevant in the setting of low regularity Riemann--Finsler structures; cf.\ \cite{OhtaMCP, OhtaCD} for the analogous results in the smooth case.

Furthermore, as already alluded to, relaxing the smoothness assumptions in the definition of semi-Riemann--Finsler structures allows for more flexibility in constructing examples that do not come from semi-Riemannian metrics, which is of great independent interest.
In particular, our results show that such low regularity Lorentz--Finsler structures are still able to make physically meaningful predictions about the formation of singularities.

\section{The geometry of low 
regularity Finsler spacetimes}

In this section, we describe our framework for Finsler spacetimes of low regularity.
Here, in contrast to the purely Lorentzian case, one allows for additional nonsmoothness in (vertical) tangent directions.
We refer to \cite{MinguzziRaychaudhuri,LuMinguzziOhtaSingularity} for the basics of Lorentz--Finsler geometry and Finsler spacetimes, and to \cite{BCS,Shlec,Ohta2021} for the positive definite setting of Finsler manifolds.

\subsection{Mixed regularity of functions on the tangent bundle}
\label{subsection: mixed regularity of maps}

Let $M$ be an $n$-dimensional smooth manifold and denote by $TM$ its tangent bundle.
If $(U,(x^i))$ is a chart on $M$, then a chart on the open set $TU \subseteq TM$ is obtained by mapping $v \in TU$ to $(x^1(\pi(v)),\ldots,x^n(\pi(v)), v^1,\ldots,v^n)$, where $\pi:TM \to M$ is the canonical projection and
\begin{align*}
    v = v^i \frac{\partial}{\partial x^i}\bigg|_{\pi(v)}.
\end{align*}
If $(V,(y^j))$ is another chart on $M$, then the transition map for the bundle chart on $TU \cap TV = T(U \cap V) \subseteq TM$ is given by $(x^i,v^i) \mapsto (y^j(x),v^i A^{j}_i(x)) = (y^j,w^j)$, where $A^j_i(x) := \frac{\partial y^j}{\partial x^i}(x)$.
The derivatives in the $(x^i,v^i)$-coordinates can be expressed in the new coordinates as follows:
\begin{align*}
    &\frac{\partial}{\partial x^i} 
    = \frac{\partial y^j}{\partial x^i}\frac{\partial}{\partial y^j} 
    + \frac{\partial w^j}{\partial x^i} \frac{\partial}{\partial w^j} 
    = A^j_i \frac{\partial}{\partial y^j} 
    + v^k \frac{\partial A^j_k}{\partial x^i} \frac{\partial}{\partial w^j},\\
    &\frac{\partial}{\partial v^i} 
    = \frac{\partial w^j}{\partial v^i} \frac{\partial}{\partial w^j} 
    = A^j_i \frac{\partial}{\partial w^j}.
\end{align*}
Note that, in the (horizontal) $x$-directions, mixed derivatives appear after a coordinate change, while in the (vertical) $v$-directions, no such derivatives appear. 
Given that both the transition map and its derivatives $A^j_i$ are smooth, the following definition is independent of the specific choice of a chart.

\begin{Definition}[$C^{(k,l),\alpha}$-regular maps]
    Let $M, N$ be smooth manifolds and let $k,l \in \mathbb{N}_0 \cup \{+\infty\}$, $\alpha \in [0,1]$, and $\mathcal{U} \subseteq TM$ be an open set.
    A map $f: \mathcal{U} \to N$ is said to be $C^{(k,l),\alpha}$-\emph{regular} if, in any bundle chart $(x^i,v^i)$ on $\mathcal{U}$, for its coordinate representation $f(x,v)$ the partial derivatives $\partial_x^{\beta} \partial_v^{\sigma} f$ exist and are
    locally $\alpha$-H\"older continuous 
    for all multi-indices $\beta,\sigma$ with $|\beta|\leq k$, $|\sigma| \leq k + l - |\beta|$.
    We write $C^{(k,l),\alpha}(\mathcal{U})$ for the space of $C^{(k,l),\alpha}$-regular functions on $\mathcal{U} \subseteq TM$. 
\end{Definition}

In the above, $\partial_x^\beta \partial_v^\sigma$ is a placeholder for any derivative of the form 
\begin{equation*}
    \partial_{x^{s_1}}^{\beta_1}\cdots \partial _{x^{s_n}}^{\beta_n} \partial_{v^{r_1}}^{\sigma_1}\cdots \partial_{v^{r_n}}^{\sigma_n} f, 
\end{equation*}
where $\beta,\sigma \in \N_0^n$ are multi-indices with $|\beta| = \sum_i \beta_i \leq k$, $|\sigma| = \sum_i \sigma_i \leq k+l-|\beta|$, and $(s_1,\ldots,s_n)$, $(r_1,\ldots,r_n)$ are arbitrary permutations of $(1,\ldots,n)$.

When $\alpha = 0$ (i.e., $\partial_x^\beta \partial_v^\sigma f$ is continuous), we omit it in the notation; namely $C^{(k,l)}:=C^{(k,l),0}$. 
Note that the usual regularity classes are obtained for $l = 0$: $C^k(TM) = C^{(k,0)}(TM)$ and $C^{k,\alpha}(TM) = C^{(k,0),\alpha}(TM)$.
By Schwarz's theorem, as long as one takes up to $k+l$ admissible derivatives in $x$ and $v$, one can order the various partial derivatives in any way without changing the result.

We remark that if $f:\mathcal{U} \to N$ is $C^{(k',l'),\alpha'}$-regular, then it is also $C^{(k,l),\alpha}$-regular provided $k \leq k'$, $k+l \leq k'+l'$ and $\alpha \leq \alpha'$. 
Moreover, note that $C^{k,\alpha} \subseteq C^{(k',l'),\alpha}$ as long as $k' + l' \leq k$.

\begin{Remark}[Generalization to vector bundles]
The fact that no mixed derivatives appear in the expression for $\partial / \partial v^i$ is a consequence of the vector bundle structure of $TM$.
Thus, more generally, it makes sense to consider maps from an open set $\mathcal{U} \subseteq \mathcal{V}$ of an arbitrary smooth vector bundle $\mathcal{V}$ over $M$ to any other smooth manifold $N$ that are of class $C^{(k,l),\alpha}$.
\end{Remark}

\subsection{Low regularity Lorentz--Finsler structures and Finsler spacetimes}

We now introduce low regularity Lorentz--Finsler manifolds and Finsler spacetimes, building on the notion of $C^{(k,l),\alpha}$-regularity.

\begin{Definition}[$C^{(k,l),\alpha}$ Lorentz--Finsler structures]
\label{Definition: Lorentz Finsler structure}
Let $M$ be a smooth manifold, $k,l \in \mathbb{N}_0 \cup \{+\infty\}$ with $k+l \geq 2$, and $\alpha \in [0,1]$. A function $L: TM \to \R$ is called a $C^{(k,l),\alpha}$ \emph{Lorentz--Finsler structure} on $M$ if
\begin{enumerate}
    \item[(1)] $L \in C^{(k,l),\alpha}(TM \setminus \mathbf0)$, where $\mathbf{0} \subseteq TM$ denotes the zero section;
    \item[(2)] $L(cv) = c^2 L(v)$ for all $v \in TM$ and $c>0$;
    \item[(3)] For any $v \in TM \setminus \mathbf 0$, the symmetric $(n \times n)$-matrix
    \begin{equation*}
        g_{ij}(v) := \frac{\partial^2 L}{\partial v^i \partial v^j}(v)
    \end{equation*}
    is non-degenerate with signature $(-,+,\ldots,+)$.
\end{enumerate}
We call the pair $(M,L)$ a $C^{(k,l),\alpha}$ \emph{Lorentz--Finsler manifold}.
\end{Definition}

Observe that any $C^{(k,l),\alpha}$ Lorentz--Finsler structure is continuous at the zero section $\mathbf{0} \subseteq TM$ (with $L=0$ on $\mathbf{0}$) as a consequence of the positive $2$-homogeneity asserted in condition (2). 
We also remark that $L(-v)$ may not coincide with $L(v)$ since $c>0$ in (2).
Moreover, by condition (3), for each $v \in T_xM \setminus \mathbf{0}$, the matrix $g_{ij}(v)$ (components of the \emph{fundamental tensor}) defines a Lorentzian metric $g_v$ on $T_xM$ via
\begin{equation}\label{eq:g_v}
g_v(w_1,w_2) :=g_{ij}(v)w_1^i w_2^j
\end{equation}
for $w_1,w_2 \in T_xM$.

\begin{Example}[Low regularity Lorentzian metrics]
Let $g$ be a $C^{k,\alpha}$ Lorentzian metric on $M$ for $k \in \N_0 \cup \{+\infty\}$ and $\alpha \in [0,1]$. Then it naturally induces a $C^{(k,\infty), \alpha}$ Lorentz--Finsler structure on $M$ via
\begin{equation*}
    L(v):= \frac{1}{2}g(v,v).
\end{equation*}
As is easily seen,
\begin{equation*}
    \frac{\partial^2 L}{\partial v^i \partial v^j}(v) = g\biggl( \frac{\partial}{\partial x^i}, \frac{\partial}{\partial x^j} \biggr)
\end{equation*}
for all $v \in TM \setminus \mathbf{0}$, so in this case, the induced Lorentzian metric $g_v$ as in \eqref{eq:g_v} is equal to $g$, and thus independent of the direction $v$.
\end{Example}

\begin{Example}[Low regularity Randers metrics]\label{ex:Randers}
An example of a low regularity \emph{Riemann--Finsler structure} (whose definition is exactly the same as that of a Lorentz--Finsler structure $L$, except that $g_{ij}(v)$ is required to be positive definite for all $v \in TM \setminus \mathbf{0}$) is given by \emph{Randers structures} of the form
\begin{equation*}
    L(v):=\frac{1}{2}\bigl(\sqrt{g(v,v)} + \beta(v)\bigr)^2,
\end{equation*}
where $g$ is a $C^{k,\alpha}$ Riemannian metric and $\beta$ is a $C^{k,\alpha}$ $1$-form such that $|\beta|_g < 1$. 
The resulting Randers metric $L$ is then of regularity $C^{(k,\infty),\alpha}$. 
Lorentzian Randers structures defined on all of $TM$ are difficult to describe; we refer to the discussion in \cite{Minguzzi-LightconesFinsler}.
\end{Example}

\begin{Example}[An explicit warped product-type example]
An explicit Lorentz--Finsler structure of regularity $C^{(1,3)}$, the class for which we establish our unweighted singularity theorems, can be constructed as follows. 
Let $M = \R \times \R^n$, $a(t):=1 + |t|^{3/2}$ and define $L: TM = (\R \times \R^n) \times (\R \times \R^n) \to \R$ by
\begin{equation*}
    2L\bigl((t,x),(\tau,v)\bigr):= - 5\tau^2 + a(t)^2 \sum_{i= 1}^n (v^i)^2 + \biggl( |\tau|^{9/2} + a(t)^{9/2} \sum_{i=1}^n |v^i|^{9/2} \biggr)^{4/9}.
\end{equation*}
An explicit computation shows that $L$ is of regularity $C^{(1,3)}$, but not of regularity $C^{(1,4)}$, and that its vertical Hessian has Lorentzian signature at every $(\tau,v) \neq 0$.
\end{Example}

Similarly to the Lorentzian case, we call $v \in TM$ \emph{timelike} if $L(v) < 0$, \emph{null} if $L(v) = 0$ and $v \not\in \mathbf{0}$, and \emph{spacelike} if $L(v) > 0$ or $v \in \mathbf{0}$.
Timelike and null vectors are collectively referred to as \emph{causal}.
We write $\Omega_x^{'T}$ for the set of timelike vectors in $T_xM$, and $\Omega^{'T}:=\bigcup_{x \in M} \Omega_x^{'T}$. 
Similarly, we write $\Omega_x^{'C}$ and $\Omega^{'C}$ for the respective sets of causal vectors.
Moreover, we set
\[ F(v):=\sqrt{-2L(v)} \]
if $v$ is causal. By Euler's homogeneous function theorem, $F(v) = \sqrt{-g_v(v,v)}$. 
It is a consequence of \cite[Lem.\ 1]{Minguzzi-LightconesFinsler} that $\Omega^{'T}_x = (\Omega^{'C}_x)^\circ$, $\partial \Omega^{'T}_x = \partial \Omega^{'C}_x$ and $\overline{\Omega^{'T}_x} = \Omega^{'C}_x \cup \{0\}$.

\begin{Definition}[$C^{(k,l),\alpha}$ Finsler spacetimes]
Let $(M,L)$ be a $C^{(k,l),\alpha}$ Lorentz--Finsler manifold. 
If there is a continuous timelike vector field $X$ on $M$, we call $(M,L)$ a $C^{(k,l),\alpha}$ \emph{Finsler spacetime} that is \emph{time-oriented} by $X$.
\end{Definition}

As in the smooth case, whenever there is a continuous timelike vector field $X$, there also exists a smooth timelike vector field $\widetilde{X}$ such that $X(x)$ and $\widetilde{X}(x)$ are in the same connected component of $\Omega_x^{'T}$ for every $x \in M$. 
This follows from a simple partition of unity argument by patching together smooth timelike vector fields obtained via local convolution in coordinates.

Given a $C^{(k,l), \alpha}$ Finsler spacetime $(M,L)$ time-oriented by $X$, we say that a causal vector $v \in \Omega_x^{'C}$ is \emph{future}(-\emph{directed}) if it lies in the same connected component of $\Omega_x^{'C}$ as $X(x)$.
We write $\Omega_x^T \subseteq \Omega_x^{'T}$ for the set of all future timelike vectors at $x$, similarly $\Omega_x^C \subseteq \Omega_x^{'C}$, as well as
\begin{equation*}
    \Omega^T := \bigcup_{ x \in M} \Omega_x^T, \qquad \Omega^C:=\bigcup_{x \in M} \Omega_x^C.
\end{equation*}

For the future timelike cone $\Omega^T_x$, consider its polar cone
\[
(\Omega^T_x)^*:=\bigl\{ \omega \in T_x^*M : \omega(v) < 0
\text{ for all $v \in \Omega^C_x$} \bigr\},
\]
which is an open convex cone in $T_x^*M$. 
For $\omega \in \overline{(\Omega_x^T)^*}$, define 
\begin{equation*}
    L^*(\omega) := -\frac{1}{2} \inf_{v \in \Omega_x^T \cap L^{-1}(-1/2)} \omega(v)^2.
\end{equation*}
Setting $F^*(\omega):=\sqrt{-2L^*(\omega)}$ for $\omega \in \overline{(\Omega_x^T)^*}$, for $v \in \Omega_x^C$, we have the \emph{reverse Cauchy--Schwarz inequality}:
\begin{equation*}
    \sqrt{4 L(v) L^*(\omega)} = F(v) F^*(\omega) \leq -\omega(v).
\end{equation*}

By the concavity property of $L$ in timelike directions, we obtain the \emph{Legendre transform} $\mathcal{L}^*: (\Omega_x^T)^* \to \Omega_x^T$, which sends $\omega$ to the unique element $v$ such that $L(v) = L^*(\omega) = \omega(v)/2$ and is a $C^{(k,l-1)}$-diffeomorphism (assuming $l \geq 1$).
A coordinate expression for $\mathcal{L}^*$ is
\begin{equation*}
    \mathcal{L}^*(\omega) = \frac{\partial L^*}{\partial \omega_i}(\omega) \frac{\partial}{\partial x^i},
\end{equation*}
where $\omega=\omega_i \diff x^i$, and the inverse map is given in the same way:
\[
(\mathcal{L}^*)^{-1}(v) =\frac{\partial L}{\partial v^i}(v) \diff x^i.
\]
Thus, in particular, $L^* \in C^{(k,l)}$ by the inverse function theorem.
Setting
\begin{equation*}
    g^*_{ij}(\omega):=\frac{\partial^2 L^*}{\partial \omega_i \partial \omega_j}(\omega)
\end{equation*}
for $\omega \in (\Omega_x^T)^*$,
we note that the matrix $(g^*_{ij}(\omega))$ also has the Lorentzian signature and satisfies (if $k+l \geq 3$)
\begin{equation*}
    \frac{\partial g^*_{ij}}{\partial \omega_k}(\omega) \omega_k = \frac{\partial g^*_{ik}}{\partial \omega_j}(\omega) \omega_k = 0
\end{equation*}
due to Euler's homogeneous function theorem, and that $(g^*_{ij}(\omega))$ is the inverse matrix of $(g_{ij}(\mathcal{L}^*(\omega)))$.

A piecewise $C^1$ curve $\eta:I \subseteq \R \to M$ in a $C^{(k,l), \alpha}$ Finsler spacetime $(M,L)$ is called \emph{timelike} if for all $t \in I$ for which $\dot \eta(t)$ exists, $\dot \eta(t)$ is a timelike vector. If $\eta$ is merely absolutely continuous, we say that it is \emph{causal} if for a.e.\ $t \in I$, $\dot \eta(t)$ is a causal vector. Similarly, we define what it means for timelike and causal curves to be \emph{future}(-\emph{directed}).
We define the \emph{chronological} (or \emph{timelike}) relation $x \ll y$ by the existence of a future timelike curve from $x$ to $y$, and the \emph{causal} relation $x \leq y$ by either $x= y$ or the existence of a future causal curve from $x$ to $y$. 
Then we employ the usual notation:
\begin{align*}
I^+(x) &:=\{ y \in M : x \ll y \}, \qquad
I^-(x) :=\{ y \in M : y \ll x \}, \\
J^+(x) &:=\{ y \in M : x \le y \}, \qquad
J^-(x) :=\{ y \in M : y \le x \}.
\end{align*}
Moreover, for $A \subseteq M$, we set $I^+(A):=\bigcup_{x \in A} I^+(x)$.

We define the \emph{Lorentz--Finsler arclength} and \emph{Lorentz--Finsler distance} (or \emph{time separation}) by
\begin{equation*}
    \mathbf L(\eta):=\int_0^1 F\bigl( \dot{\eta}(t) \bigr) \diff t, \qquad d(x,y):=\sup_{\eta} \mathbf{L}(\eta),
\end{equation*}
where the supremum runs over future causal curves $\eta:[0,1] \to M$ from $x$ to $y$. 
Note that, by convention, $d(x,y) = 0$ if there is no such curve. 
Given any subset $A \subseteq M$, we define $d_A(y):=\sup_{x \in A} d(x,y)$.
For any $y \in J^+(A)=\bigcup_{x \in A}J^+(x)$, we say that a future causal curve $\eta:[0,1] \to M$ with $\eta(0) \in A$ and $\eta(1) = y$ \emph{maximizes} the distance from $A$ to $y$ if $\mathbf L(\eta) = d_A(y)$.

The following lemma will be useful in several places.

\begin{Lemma}[Proper cone structures]
\label{Lemma: lowregFinsler cone structure}
Let $(M,L)$ be a $C^{(1,3)}$ Finsler spacetime. 
Then the future causal cone distribution $\Omega^C  \subseteq TM$ is a locally Lipschitz proper cone structure in the sense of \textnormal{\cite{Minguzziconestructure}}.
\end{Lemma}

\begin{proof}
    This was already shown in \cite[Thm.\ 2.52]{Minguzziconestructure}, noting that the $C^2$ assumption in its statement is not needed in the proof.
\end{proof}

\subsection{Covariant derivative and geodesics}

Next, we introduce the covariant derivative under slightly better regularity.

\begin{Definition}[Chern connection]
Let $(M,L)$ be a $C^{(1,3),\alpha}$ Finsler spacetime.
Define
\begin{equation}\label{def:FormalChristoffels}
    \gamma^i_{jk}(v):= \frac{1}{2} g^{il}(v) \left( \frac{\partial g_{lk}}{\partial x^j} + \frac{\partial g_{jl}}{\partial x^k} - \frac{\partial g_{jk}}{\partial x^l}\right)(v)
\end{equation}
for $v \in TM \setminus \mathbf{0}$,
where $(g^{ij}(v))$ denotes the inverse matrix of $(g_{ij}(v))$.
Further, set for $v \in TM \setminus \mathbf{0}$
\begin{equation*}
    G^i(v):= \gamma^i_{jk}(v) v^j v^k, \qquad
    N^i_j(v):=\frac{1}{2} \frac{\partial G^i}{\partial v^j}(v).
\end{equation*}
By convention, we put $G^i(0) := 0$ and $N^i_j(0):=0$. 
The coefficients of the \emph{Chern connection} of $(M,L)$ are now defined for $v \in TM \setminus \mathbf{0}$ via
\begin{equation*}
    \Gamma^i_{jk}(v):=\gamma^i_{jk}(v) - \frac{1}{2} g^{il}(v) \left(\frac{\partial g_{lk}}{\partial v^m} N^m_j + \frac{\partial g_{jl}}{\partial v^m} N^m_k - \frac{\partial g_{jk}}{\partial v^m} N^m_l \right)(v).
\end{equation*}
\end{Definition}

\begin{Remark}[On regularities]\label{rm:regular}
Suppose that $(M,L)$ is a $C^{(1,2)}$ Finsler spacetime.
Invoking Euler's homogeneous function theorem, we see that the spray coefficients $G^{i}$ satisfy
\begin{equation}\label{def:SprayCoeffImproved}
    G^i(v)= g^{il}(v) \left( \frac{\partial^2 L}{\partial x^j \partial v^l}(v)v^j- \frac{\partial L}{\partial x^l}(v)\right).
\end{equation}
This allows us to define the nonlinear connection coefficients $N^{i}_j$ even in this regularity, as $\partial_x^{\beta} \partial_v^{\sigma} L$ is well-defined and continuous for all $|\beta| \leq 1, |\sigma| \leq 3-|\beta|$.
For the convenience of the reader, we summarize the regularities of the various objects associated with $L \in C^{(1,2)}$:
\begin{align*}
    g_{ij},\, g^{ij} &\in C^{(1,0)}(TM \setminus \mathbf{0}),\\
    G^{i} &\in C^{(0,1)}(TM \setminus \mathbf{0}),\\
    \gamma^i_{jk},\,N^i_j,\, \Gamma^i_{jk} &\in C^{(0,0)}(TM \setminus \mathbf{0}).
\end{align*}
\end{Remark}

\begin{Definition}[Covariant derivative]
Let $(M,L)$ be a $C^{(k,l),\alpha}$ Finsler spacetime with $k \geq 1$ and $l \geq 2$.
Let $Y = Y^i \partial_{x^i},V = V^i\partial_{x^i}$ and $W = W^i \partial_{x^i} \neq 0$ be vector fields on $M$ of regularity $C^{m_1,\alpha_1}, C^{m_2,\alpha_2}$ and $C^{m_3,\alpha_3}$, respectively, where $m_1 \geq 1$.
Then the \emph{covariant derivative} of $Y$ in the direction $V$ with respect to \emph{reference direction} $W$ is defined as
\begin{equation*}
    D^W_V Y := \left( V^j \frac{\partial Y^i}{\partial x^j} + \Gamma^i_{jk}(W) Y^j V^k\right) \frac{\partial}{\partial x^i}.
\end{equation*}
Evidently, it is a vector field of regularity $C^{\min(k-1,l-2, m_1-1, m_2,m_3),\min(\alpha,\alpha_1,\alpha_2,\alpha_3)}$.
\end{Definition}

Let $(M,L)$ be a $C^{(1,2)}$ Finsler spacetime and recall from Remark~\ref{rm:regular} that the coefficients $\gamma^i_{jk}$ are continuous in this case.

\begin{Definition}[Geodesics]
A $C^2$-curve $\eta:I \to M$ with an interval $I \subseteq \R$ is called a \emph{pregeodesic} if
\begin{equation*}
\Ddot{\eta}^i + \gamma^i_{jk}(\dot{\eta})\dot \eta^j \dot \eta^k = \lambda \dot{\eta}^i
\end{equation*}
on $I$ for some continuous function $\lambda:I \to \R$. It is called a \emph{geodesic} if $\lambda \equiv 0$. 
We may equivalently characterize (pre)geodesics via the Chern connection, i.e., $D^{\dot{\eta}}_{\dot{\eta}}\dot{\eta} =\lambda \dot{\eta}$ or $D^{\dot{\eta}}_{\dot{\eta}}\dot{\eta} = 0$.
\end{Definition}

Note that by Peano's theorem, for any initial datum $v \in TM$ there exists a (not necessarily unique) geodesic $\eta:I \to M$ defined on its maximal domain $I$ with $0 \in I$ and $\dot{\eta}(0) = v$. 
We say that $(M,L)$ is (\emph{future/past} \emph{timelike/null/causal}) \emph{geodesically incomplete} if there exists a solution of the geodesic equation of the respective quality whose maximal domain is bounded above (for future incompleteness) or below (for past incompleteness).

Along any geodesic $\eta$, $L(\dot{\eta})$ is constant since
\[
\frac{\diff}{\diff t}\bigl[ L(\dot{\eta}) \bigr]
= \frac{1}{2} \frac{\diff}{\diff t}\bigl[ g_{\dot{\eta}}(\dot{\eta},\dot{\eta}) \bigr]
=g_{\dot{\eta}}(D^{\dot{\eta}}_{\dot{\eta}}\dot{\eta},\dot{\eta})
=0
\]
(see \cite[Eq.\ (3.1)]{LuMinguzziOhtaSingularity} for the second equality).
As a consequence, we obtain the following.

\begin{Lemma}[Fixed causal character]
Geodesics in a $C^{(1,2)}$ Finsler spacetime have fixed causal character, i.e., they are either timelike, null, or spacelike, on the entirety of their domain.
\end{Lemma}

The following is a consequence of the basic ODE theory.

\begin{Proposition}[Existence time estimate for geodesics]
\label{Prop: Existence time for Leps geodesics}
Let $(M,L)$ be a $C^{(1,2)}$ Finsler spacetime and $(L_\eps)_{\eps >0}$ be smooth Lorentz--Finsler structures converging in $C^{(1,2)}_{\loc}$ to $L$ as $\eps \to 0$.
Let $\mathcal K \subseteq TM$ be compact and suppose that all $L$-geodesics with initial data in $\mathcal{K}$ are defined on $[0,\infty)$.
Then for any $T>0$ there exists $\eps_0=\eps_0(\mathcal{K},T)>0$ such that for any $\eps \in (0,\eps_0)$, all $L_\eps$-geodesics with initial data in $\mathcal{K}$ are defined at least on $[0,T]$.
\end{Proposition}

Here, by the \emph{$C^{(k,l),\alpha}_{\loc}$-convergence}, we mean that for any pair of multi-indices $\beta,\sigma$ with $|\beta| \leq k$ and $|\sigma| \leq k +l -|\beta|$,
    \begin{equation}\label{eq:kla-conv}
        \partial_x^\beta \partial_v^\sigma L_\eps \to \partial_x^\beta \partial_v^\sigma L
    \end{equation}
locally in the $C^{\alpha}$-norm on $TM \setminus \mathbf 0$.

\begin{proof}
    Suppose that the claim is false, so that there is some $T>0$ as well as a sequence $\eps_i \to 0$ and $v_i \in \mathcal K$ such that there are $L_{\eps_i}$-geodesics $\eta_i$ with $\dot{\eta}_i(0) = v_i$ that are defined at most on $[0,T)$.
    By compactness, we may assume $v_i \to v \in \mathcal K$.
    Since the coefficients $\gamma^j_{kl}[L_{\eps_i}]$ corresponding to $L_{\eps_i}$ converge locally uniformly to the coefficients $\gamma^j_{kl}[L]$ of $L$, $\eta_i$ converges in $C^2_{\loc}$ to a solution $\eta$ of the $L$-geodesic equation with $\dot{\eta}(0) = v$.
    Then, an estimate of Hartman (see \cite[Prop.\ 2.5]{KOSS}) asserts that $\eta$ is defined at most on $[0,T)$, contradicting our assumption.
\end{proof}

An analog of the above result also holds if $L$-geodesics in $\mathcal{K}$ are defined on $[0,T)$; in this case for any $\widetilde{T} \in (0,T)$ there exists $\eps_0 >0$ such that for any $\eps \in (0,\eps_0)$, all $L_\eps$-geodesics with initial data in $\mathcal K$ are defined at least on $[0,\widetilde{T}]$.

We have the following compactness result as a consequence of these existence time estimates.

\begin{Corollary}[Compactness for geodesics]\label{cr:Fcpt}
Let $(M,L)$ be a $C^{(1,2)}$ Finsler spacetime and $(L_\eps)_{\eps>0}$ be smooth Lorentz--Finsler structures converging in $C^{(1,2)}_{\loc}$ to $L$ as $\eps \to 0$.
Let $\mathcal K \subseteq TM$ be compact and suppose that all $L$-geodesics with initial data in $\mathcal{K}$ are defined on $[0,T)$, with $T \leq +\infty$.
Then for any $\widetilde{T} \in [0,T)$ there exists $\eps_0=\varepsilon_0(\mathcal{K},\widetilde{T}) > 0$ such that the set
\begin{equation*}
    \mathcal{F}_{\mathcal{K}, \widetilde{T}, \varepsilon_0}
    := \bigl\{ \dot{\eta}(t) : \eta \text{ is an } L_{\varepsilon}\text{-geodesic},\, \dot{\eta}(0) \in \mathcal K,\, \varepsilon \in [0,\varepsilon_0],\, t \in [0,\widetilde{T}] \bigr\}
    \subseteq TM
\end{equation*}
is compact, where we set $L_0 := L$.
\end{Corollary}

\subsection{Distributional curvature and energy conditions}
\label{ssc:curvature}

In the smooth case, the components of the curvature tensor $R : T_xM \rightarrow T^*_xM \otimes T_xM$ are defined as
\begin{equation}\label{eq:DistributionalCurvature}
    R^i_j(v) := \frac{\partial G^i}{\partial x^j}(v) - \frac{\partial N^i_j}{\partial x^k}(v)v^k + \frac{\partial N^i_j}{\partial v^k}(v)G^k(v) - N^i_k(v)N^k_j(v).
\end{equation}
In our setting, we must first give meaning to this expression, as it formally involves derivatives of (merely) continuous functions; we assume $L \in C^{(1,2)}$, so that $G^i \in C^{(0,1)}$ and $N^i_j \in C^{(0,0)}$, as we saw in Remark \ref{rm:regular}. 

As laid out in \cite[Sec.\ 3.2]{Graf-SingularityC1}, we can argue locally and consider \eqref{eq:DistributionalCurvature} on some open subset $U \times V$ of $\R^n \times \R^n$.
We then recast \eqref{eq:DistributionalCurvature} by interpreting the expressions that appear as (first-order) distributions. 
To do so, it suffices to define the action of each term on the dense subspace $C^1_c(U) \otimes C^1_c(V)$.
For the first term in \eqref{eq:DistributionalCurvature}, we simply set
\begin{equation*}
    \biggl\langle \frac{\partial G^i}{\partial x^j}, \varphi \otimes \psi \biggr\rangle = - \biggl\langle G^{i}, \biggr(\frac{\partial\varphi}{\partial x^j} \biggr) \otimes \psi \biggr\rangle = - \int_{U \times V} G^i(x,v)\frac{\partial\varphi}{\partial x^j}(x) \psi(v) \diff x \diff v
\end{equation*}
for $\varphi \otimes \psi \in C^1_c(U) \otimes C^1_c(V)$.
We can similarly interpret the second term, while for the third, the fact that $G^k \in C^{(0,1)}$ ensures $G^k(x, \, \cdot \,) \cdot \psi \in C^1_c(V)$ for each $x \in U$ and $\psi \in C^1_c(V)$, so that
\begin{equation*}
    \biggl\langle \frac{\partial N^i_j}{\partial v^k}G^k, \varphi\otimes \psi \biggr\rangle =  - \int_{U \times V} N^{i}_j(x,v) \varphi(x) \frac{\partial}{\partial v^k}\bigl[ G^k(x,v)\psi(v) \bigr] \diff x \diff v
\end{equation*}
is well-defined.
The last term in \eqref{eq:DistributionalCurvature} merely involves the multiplication of two continuous functions, which we can directly interpret as a regular distribution on $U \times V$.

The distributional \emph{Ricci curvature} (or \emph{Ricci scalar}) $\Ric$ of a $C^{(1,2)}$ Finsler spacetime $(M,L)$ is then defined locally as the distribution $R^{i}_{\, i}$. 
Since at most the first order derivatives of continuous functions appear in the local expression for $R^i_j$, one sees that $\Ric \in \mathcal{D}^{'(1)}(TM \setminus \mathbf 0)$, i.e., the Ricci curvature is a distribution of first order (see \cite{KOSS} for details regarding tensor distributions on smooth manifolds).

\begin{Definition}[Distributional strong energy condition]
\label{Definition: distributional strong energy}
We say that a $C^{(1,2)}$ Finsler spacetime $(M,L)$ satisfies the \emph{distributional strong energy condition} if $\Ric \geq 0$ on $\Omega^T$ in the distributional sense, i.e., $\langle \Ric, \mu\rangle \geq 0$ for any nonnegative volume density $\mu$ on $TM$ compactly supported in $\Omega^T$.
Similarly, we say that $\Ric \geq K$ \emph{timelike distributionally} for $K \in \R$ if the distribution $\Ric + 2 K L$ is nonnegative.
\end{Definition}

\begin{Remark}[Compatibility with standard strong energy condition in sufficient regularity]
    It is easily seen that this condition coincides with the usual strong energy condition in the case where $\Ric : TM \setminus \mathbf{0} \rightarrow \R$ is continuous.
    Indeed, if $\Ric(v) \geq 0$ for all $v \in \Omega^T$, then clearly $\langle \Ric, \mu \rangle \geq 0$ for any $\mu$ compactly supported in $\Omega^T$.
    On the other hand, if $\Ric(v)<0$ for some $v \in \Omega^T$, by continuity, this inequality continues to hold in a neighborhood $\mathcal{U} \subseteq TM$ of $v$ and by the openness of $\Omega^T$, we can assume that $\mathcal{U} \subseteq \Omega^T$.
    Hence, $\langle \Ric, \mu \rangle <0$ for any nonnegative, nontrivial $\mu$ compactly supported in $\mathcal{U}$. 
    The same argument shows that the general bound $\Ric \geq K$ agrees with its continuous counterpart if $\Ric$ is continuous.
\end{Remark}

The null energy condition is more subtle since the set of null vectors is invisible from the distributional viewpoint.

\begin{Definition}[Distributional null energy condition]
\label{Definition: null energy condition}
    We say that a $C^{(1,2)}$ Finsler spacetime $(M,L)$ satisfies the \textit{distributional null energy condition} if for any compact set $\mathcal{K} \subseteq TM \setminus \mathbf{0}$, we have
    \begin{equation*}
        \liminf_{\delta \to 0} \inf_{\mu \in M_\delta(\mathcal{K})} \langle \Ric, \mu\rangle \geq 0,
    \end{equation*}
    where $M_\delta(\mathcal{K})$ is the set of densities $\mu$ on $TM$ such that $\mu \geq 0$, $\supp \mu \subseteq \{|L| \leq \delta \} \cap \Omega^C \cap \mathcal{K}$, and $\int_{TM} \mu = 1$. 
    When $M_\delta(\mathcal{K}) = \varnothing$, the infimum is set to $+ \infty$.
\end{Definition}

Spelled out, this definition requires that for every $\eps >0$ and any compact set $\mathcal{K} \subseteq TM \setminus \mathbf{0}$, we can find some $\delta_0>0$ such that
    \begin{equation*}
            \langle \Ric, \mu \rangle \geq - \eps \quad \textnormal{for any } \mu \in M_\delta(\mathcal{K}),\, \delta \in (0,\delta_0].
    \end{equation*}
Note (in comparison with \cite{KOSS}) that we need to restrict $M_\delta(\mathcal{K})$ to future causal vectors due to the irreversibility of $L$ (i.e., $L(-v) \neq L(v)$, and thus $\Ric(-v) \neq \Ric(v)$).

\begin{Remark}[Compatibility with standard null energy condition in sufficient regularity]
    Again, this is consistent with the usual definition in the smooth case. 
    Suppose that the Ricci curvature $\Ric : TM \setminus \mathbf{0} \rightarrow \R$ is continuous and satisfies the pointwise null energy condition $\Ric(v) \geq 0$ for all future null vectors $v \in TM \setminus \mathbf{0}$. 
    By continuity, for any $\eps >0$, there exists an open neighborhood $\mathcal{U} \subseteq TM \setminus \mathbf{0}$ containing $\{ L = 0\} \cap \Omega^C$ such that $\Ric(v) > - \eps$ for all $v \in \mathcal{U}$. 
    Now, for any compact set $\mathcal{K} \subseteq TM \setminus \mathbf{0}$, there exists $\delta_0 > 0$ such that 
    \begin{equation*}
        \{|L| \leq \delta\} \cap \Omega^C \cap \mathcal{K} \subseteq \mathcal{U} \quad \textnormal{for all } \delta \in (0,\delta_0].
    \end{equation*}
    In fact, otherwise, the family $\mathcal{F}_\delta:=(\{|L| \leq \delta\} \cap \Omega^C \cap \mathcal{K}) \setminus \mathcal{U}$ constitutes a family of closed, nested, and nonempty sets, and hence, by the compactness of $\mathcal{K}$, we would have $\bigcap_{\delta >0} \mathcal{F}_\delta \neq \varnothing$. 
    This, however, contradicts $\bigcap_{\delta > 0} \{|L| \leq \delta\} \cap \Omega^C = \{L=0\} \cap \Omega^C \subseteq \mathcal{U}$.
    Therefore $\Ric > - \eps$ on the support of any $\mu \in M_\delta(\mathcal{K})$, thereby
    \begin{equation*}
        \langle \Ric, \mu \rangle > - \eps \int_{\{|L| \leq \delta\} \cap \Omega^C \cap \mathcal{K}} \mu = -\eps \quad \textnormal{for any } \mu \in M_\delta(\mathcal{K}),\, \delta \in (0,\delta_0].
    \end{equation*}
    Next, suppose that there is some future null vector $v_0 \in TM \setminus \mathbf{0}$ such that $\Ric(v_0) < 0$. 
    Then there exists some $\eps >0$ and a compact neighborhood $\mathcal{K} \subseteq TM \setminus \mathbf{0}$ of $v_0$ such that $\Ric(v) \leq - \eps$ for all $v \in \mathcal{K}$. 
    Hence, $\langle \Ric, \mu \rangle \leq -\eps$ for any $\delta >0$ and any $\mu \in M_\delta(\mathcal{K})$.
\end{Remark}

\subsection{Weighted Ricci curvature}
\label{Subsection: weights}

On a $C^{(1,2)}$ Finsler spacetime $(M,L)$, we can study $C^1$ measures and associated weighted Ricci curvature. 
Let $\mathfrak m$ be a $C^1$ measure on $M$ in the sense that it is locally represented as $\mathfrak m = \sigma(x) \diff x^1 \cdots \diff x^n$ with a positive $C^1$ function $\sigma$. 
Consider the distortion
\begin{equation*}
    \Psi(x,v):= - \log \sigma(x) + \frac{1}{2} \log \bigl| \det \bigl(g_{ij}(x,v) \bigr) \bigr|
\end{equation*}
for $v \in TM \setminus \mathbf 0$.
Then we have $\Psi \in C^{(1,0)}$ and $\mathfrak m =e^{-\Psi(V)} \sqrt{|\det (g_{ij}(V))|} \diff x^1 \cdots \diff x^n$ for any future timelike vector field $V$. 
Using the \emph{geodesic spray vector field} 
\begin{equation*}
    S:= v^i \frac{\partial}{\partial x^i} - G^i(x,v) \frac{\partial}{\partial v^i}
\end{equation*}
on $TM$ (giving geodesics as the projection of its integral curves; see \cite{Shspray}), we define the \emph{weighted Ricci curvature} as the distribution
\begin{equation*} 
    \Ric_{\mathfrak m}^N := \Ric + S(S\Psi) - \frac{(S\Psi)^2}{N - n}
\end{equation*}
on $TM \setminus \mathbf 0$, for $N \in \mathbb R \cup \{+\infty\} \setminus \{n\}$, where the last term vanishes definitionally if $N = +\infty$.
Note that $S\Psi \in C^{(0,0)}$. 
As in the case of $\Ric$, since $G^i \in C^{(0,1)}$, we see that $S(S\Psi)$ is a well-defined first order distribution on $TM \setminus \mathbf 0$. 
We further define $\Ric^{\infty}_{\mathfrak m} := \Ric + S(S\Psi)$, as well as $\Ric^n_{\mathfrak m}:=\Ric$ provided $S\Psi \equiv 0$.

\begin{Remark}[Weighted Ricci curvature in higher regularity]
If $L \in C^{(2,2)}$ and $\mathfrak m$ is a $C^2$-measure, then the above definition of weighted Ricci curvature agrees with the following usual one (see \cite{LuMinguzziOhtaSingularity,Ohta2021}).
Fix $v \in \Omega_x^T$, let $\eta:(-\eps,\eps) \to M$ be the geodesic with $\dot{\eta}(0)=v$, and put $\psi_\eta(t):=\Psi(\eta(t),\dot{\eta}(t))$.
Then, since $\Psi \in C^{(2,0)}$ and $\eta \in C^3$, we can define
\[
\Ric^N_{\mathfrak m}(v) :=\Ric(v) +\psi''_\eta(0) -\frac{\psi'_\eta(0)^2}{N-n}.
\]
\end{Remark}

\begin{Definition}[Weighted energy conditions]
\label{Definition: Weighted energy conditions}
    Let $(M,L,\mathfrak m)$ be a $C^{(1,2)}$ Finsler spacetime $(M,L)$ equipped with a $C^1$ measure $\mathfrak m$. 
    Recall $n=\dim M$.
\begin{enumerate}
\item[(1)] 
    Let $N \in (-\infty,1] \cup [n,\infty]$. We say that $(M,L,\mathfrak m)$ satisfies $\Ric^N_{\mathfrak{m}} \geq K$ \emph{timelike distributionally} for $K \in \R$ (or the \emph{distributional strong $N$-energy condition} when $K=0$) if $\langle \Ric^N_{\mathfrak m} + 2 KL, \mu \rangle \geq 0$ for any nonnegative volume density $\mu$ compactly supported in $\Omega^T$. 
\item[(2)]
    Let $N \in (-\infty,2] \cup [n,\infty]$ with $n \ge 3$. 
    We say that $(M,L,\mathfrak m)$ satisfies the \emph{distributional null $N$-energy condition} if for any compact set $\mathcal K \subseteq TM \setminus \mathbf 0$ we have
    \begin{equation*}
        \liminf_{\delta \to 0} \inf_{\mu \in M_\delta(\mathcal K)} \langle \Ric^N_{\mathfrak m},\mu \rangle \geq 0,
    \end{equation*}
    where $M_\delta(\mathcal{K})$ is as in Definition \ref{Definition: null energy condition}.
\end{enumerate}
\end{Definition}

Let us note that, in general, there may be no measure $\m$ on $M$ such that $S\Psi \equiv 0$ (see \cite{OhtaScurvature}, \cite[\S 10.3.2]{Ohta2021} in the positive definite case), so an unweighted Ricci bound can, in general, not be obtained as a special case of a weighted bound in the Finsler setting.

\begin{Remark}[From the distributional strong energy condition to the null energy condition]
Let $(M,L)$ be a $C^{(1,2)}$ Finsler spacetime. If $L$ satisfies $\Ric \geq K$ timelike distributionally in the sense of Definition \ref{Definition: distributional strong energy} for some $K \in \R$, then it satisfies the distributional null energy condition. 
Similarly, in the weighted case, if $(M,L,\m)$ is a $C^1$ weighted $C^{(1,2)}$ Finsler spacetime, $N \in (-\infty,1] \cup [n,\infty]$, $K \in \R$ and $\Ric^N_\m \geq K$ timelike distributionally, then $(M,L, \m)$ satisfies the distributional $N$-null energy condition. 
This was unclear even in the $C^1$ Lorentzian case \cite{Graf-SingularityC1}, however, we can establish it here due to our definition of (weighted) null energy condition, which is implied by \cite[Def.\ 5.1]{Graf-SingularityC1} for $C^1$ Lorentzian metrics and is sufficient for the purposes of establishing the Penrose singularity theorems (see Theorems \ref{theorem: unweighted penrose} and \ref{theorem: weighted penrose}).

Let us sketch the argument in the unweighted case with $K = 0$, the (un)weighted case with general $K$ is analogous. 
Suppose that $(M,L)$ is $C^{(1,2)}$ and satisfies the distributional strong energy condition. 
Fix a compact set $\mathcal K \subseteq TM \setminus \mathbf 0$ and let $\mu \in M_\delta(\mathcal K)$, i.e., $\mu$ is supported in $\{|L| \leq \delta\} \cap \Omega^C \cap \mathcal K$. 
Let $X$ be the smooth time orientation vector field of $(M,L)$ and $s > 0$. Define the map $F_s:TM \to TM$ via $F_s(v):= v + sX(\pi(v))$.
Then the push-forward $\mu_s :=(F_s)_* \mu$ is a nonnegative volume density with compact support such that $\mathrm{supp}(\mu_s) = F_s(\mathrm{supp}(\mu)) \subseteq \Omega^T$. 
Since $\mu_s \to \mu$ in the sense of compactly supported volume densities, we conclude that
\begin{equation*}
    \langle \Ric,\mu\rangle = \lim_{s \to 0} \langle \Ric,\mu_s\rangle \geq 0.
\end{equation*}
Taking the infimum over $\mu \in M_\delta(\mathcal K)$ and then sending $\delta \downarrow 0$, we conclude that the distributional null energy condition is satisfied.
\end{Remark}

\begin{Remark}[Rewriting of weighted Ricci curvature]
\label{rmk:WeightedRicciSimplified}
    For the convergence results of the next section, it will be beneficial to rewrite the expression for $\Ric_{\m}^N$; cf.\ \cite[Sec.\ 7]{ShenSCurvature}, \cite[Lem.\ 10.6]{Ohta2021}.
    To do so, decompose $\Psi$ into
    \begin{equation*}
        \tau(x,v) := \frac{1}{2} \log \bigl|\det \bigl(g_{ij}(x,v)\bigr)\bigr|, \quad
        u(x):= \log \sigma(x),
    \end{equation*}
    and recall from \eqref{def:SprayCoeffImproved} that
    \begin{equation*}
        G^i = g^{il} A_l, \quad \textnormal{where} \, A_l := v^j \frac{\partial^2 L}{\partial x^j \partial v^l} - \frac{\partial L}{\partial x^l}.
    \end{equation*}
    Using the symmetry of $g^{il}$ and Jacobi's formula $2 \frac{\partial \tau}{\partial x^j} = g^{il}\frac{\partial g_{il}}{\partial x^j}$, we get
    \begin{align*}
        g^{il}\frac{\partial A_l}{\partial v^i} 
        = g^{il}\left(\frac{\partial^2L}{\partial x^i\partial v^l} +v^j\frac{\partial g_{il}}{\partial x^j} -\frac{\partial^2L}{\partial v^i \partial x^l} \right)
        = g^{il}v^j\frac{\partial g_{il}}{\partial x^j} = 2v^j\frac{\partial\tau}{\partial x^j}.
    \end{align*}
    Moreover, differentiating the inverse
    matrix $\partial_{v^i} g^{il} = - g^{ij} (\partial_{v^i}g_{jk})g^{kl}$ and using the symmetry of the vertical derivatives of
    $L$, we obtain
    \begin{equation*}
        \frac{\partial g^{il}}{\partial v^i}A_l = -g^{ij}\frac{\partial g_{jk}}{\partial v^i}G^k = -g^{ij}\frac{\partial g_{ij}}{\partial v^k}G^k = -2G^k\frac{\partial\tau}{\partial v^k}.
    \end{equation*}
    Hence,
    \begin{equation*}
        \nu :=N^i_i = \frac{1}{2} \frac{\partial G^i}{\partial v^i}= v^j\frac{\partial\tau}{\partial x^j}
        -G^j\frac{\partial\tau}{\partial v^j} = S\tau.
    \end{equation*}
    Using this, we can rewrite the expression for $\Ric_{\m}^N$ as follows:
    Since $\Psi = \tau - u$, we have
    \begin{equation*}
        S\Psi = \nu - Su = \nu - v^i \frac{\partial u}{\partial x^i}, \qquad S^2\Psi = S\nu -  v^iv^j \frac{\partial^2 u}{\partial x^i \partial x^j} + G^i \frac{\partial u}{\partial x^i}.
    \end{equation*}
    Taking the trace of $R^i_j$ in \eqref{eq:DistributionalCurvature}, we find
    \begin{equation*}
        \Ric = \frac{\partial G^i}{\partial x^i} - S \nu - N^i_k N^k_i.
    \end{equation*}
    Thus,
    \begin{equation}
    \label{eq:WeightedRicciSimplified}
        \Ric_{\m}^N =\frac{\partial G^i}{\partial x^i} - N^i_j N^j_i- v^iv^j\frac{\partial^2 u}{\partial x^i\partial x^j} +G^i\frac{\partial u}{\partial x^i} -\frac{1}{N-n}\left(N^i_i-v^i\frac{\partial u}{\partial x^i}\right)^2,
    \end{equation}
    where the last term is understood as $0$ when $N=\infty$.
    Notice that, in this expression, the term $S\nu$ in the definition of $\Ric_{\m}^N$ has canceled with the term $S^2\tau = S \nu$ coming from the distortion, and hence, no derivatives of $N^i_i$ appear in \eqref{eq:WeightedRicciSimplified}.
    \end{Remark}

\section{Smooth approximations}
\label{subsection: Regularization, Legendre transform}

This section is devoted to the key technical tool for our discussion: smooth approximations of low regularity Lorentz--Finsler structures.
We construct them in two steps.

\subsection{Construction of regularizations}\label{ssc:L_e}

Let $(M,L)$ be a $C^{(k,l),\alpha}$ Finsler spacetime.
We first introduce a regularization scheme (following \cite{Graf-SingularityC1,KOSS}) that produces a one-parameter family of $C^\infty$ Lorentz--Finsler structures approximating $L$.

To this end, we pick an atlas $\{(U_a,  \psi_a)\}_{a \in A}$ of $M$ with relatively compact $U_a$ and choose a subordinate partition of unity $\{\zeta_a\}_{a \in A}$ on $M$.
Further, let $\{{\chi}_{a}\}_{a \in A}$ be a locally finite family of smooth, compactly supported functions on $M$ such that $\chi_a \equiv 1$ on a neighborhood of $\supp \zeta_a \subseteq U_a$, $0 \leq \chi_a \leq 1$, and that $\supp \chi_a \subseteq U_a$.
Finally, fix a Riemannian background metric $\tilde g$ on $M$.

Locally, in a chart $(U_a, \psi_a)$, set $\ell_a:=({\psi}_{a})_* (( \zeta_{a} \circ \pi) L)\colon T[\psi_a(U_a)] \to \R$, and define for (small) $\eps>0$ and $(x,v) \in T[\psi_a(U_a)]=\psi_a(U_a) \times \R^n$,
\begin{equation}\label{eq:HeaL}
\begin{split}
    \mathcal{H}_{\eps, a} L(x,v):= &\int_{\R^n \times \R^n} \ell_a \bigl( x-\eps y, v- \eps r_a(x,v)w \bigr) \rho(y,w) \diff y \diff w \\
    = \; &r_a(x,v)^2 (\ell_a \star \rho_\eps)\bigl( x,v/r_a(x,v)\bigr)
\end{split}
\end{equation}
(the latter expression is understood as $0$ at $v=0$),
where $r_a(x,v):=\| (\diff \psi_a^{-1})_x(v)\|_{\tilde g}$, $\rho$ is a convolution kernel on $\R^{2n}$ supported in $B_1(0)^2$, and $\rho_\eps := \eps^{-2n} \rho(\cdot / \eps)$
($\ell_a(x-\eps y,\cdot)$ reads as $0$ when $x-\eps y \not\in \psi_a(U_a)$).
Then, $\mathcal{H}_{\eps, a} L$ is positively $2$-homogeneous in $v$ and smooth away from the zero section.
Globally, we now define
\begin{equation}
\label{eq:convo}
    \mathcal{H}_\eps L:= \sum_{a \in A} (\chi_a \circ \pi)(\psi_a)^{*} (\mathcal{H}_{\eps,a} L).
\end{equation}
Since each $\mathcal{H}_{\eps,a} L$ is positively $2$-homogeneous and smooth away from the zero section, so is $\mathcal{H}_\eps L$.

We will also use the operator $\mathcal{H}_\eps$ to regularize $\Ric[L]$ (the Ricci curvature of $L$) by locally setting
\begin{equation*}
    \mathcal{H}_{\eps,a}(\Ric[L]) := r_a(x,v)^2 (\Ric[L]_a \star \rho_\eps)\bigl( x,v/r_a(x,v) \bigr),
\end{equation*}
where $\Ric[L]_a := (\psi_a)_*((\zeta_a \circ \pi) \Ric[L])$, and summing up these local pieces as in \eqref{eq:convo}.
To define this regularization globally, some care is required regarding the supports of the convolution kernels (in low regularity, $\Ric[L]$ is only defined as a distribution on the slit tangent bundle).
Since our approximation results concern only compact subsets of $TM \setminus \mathbf{0}$, this issue can be omitted, however.

We summarize the most important properties of $\mathcal{H}_\eps L$ in the following proposition.

\begin{Proposition}[Properties of the regularization]\label{prop:PropertiesApproximation}
Let $(M,L)$ be a $C^{(k,l),\alpha}$ Finsler spacetime.
\begin{enumerate}
    \item
    \label{PropApproximation1}
    We have $\mathcal{H}_\eps L \to L$ in $C^{(k,l)}_{\loc}$ as $\eps \rightarrow 0$, and when $\alpha >0$, in $C^{(k,l), \alpha'}_{\loc}$ \textnormal{(}in the sense of \eqref{eq:kla-conv}\textnormal{)} for every $\alpha' < \alpha$.
    Moreover, provided $\alpha>0$, for each compact set $\mathcal K \subseteq TM \setminus \mathbf 0$, there exist $\varepsilon_\mathcal K, c_\mathcal K > 0$ such that
    \begin{equation*}
        \|\partial_x^\beta \partial_v^\sigma \mathcal{H}_\eps L - \partial_x^\beta \partial_v^\sigma L\|_{\infty, \mathcal{K}} \leq c_\mathcal K \eps^{\alpha} \quad \text{for all } \eps \in (0, \eps_\mathcal{K}]
    \end{equation*}
    and $|\beta| \leq k, \, |\sigma | \leq k+l-|\beta|$.
    In particular, if $L \in C^{(k,l)}$ with $k \geq 1$, then
    \begin{equation*}
        \|\partial_x^\beta \partial_v^\sigma \mathcal{H}_\eps L - \partial_x^\beta \partial_v^\sigma L\|_{\infty, \mathcal{K}} \leq c_\mathcal K \eps \quad \text{for all } \eps \in (0, \eps_\mathcal{K}]
    \end{equation*}
    and $|\beta| \leq k-1$, $|\sigma | \leq k+l-1-|\beta|$.
    \item \label{PropApproximation2}
    There exists a smooth function $(\varepsilon,v) \mapsto L_{\varepsilon}(v)$, $(0,1) \times (TM \setminus \mathbf 0) \to \R$, such that
    \begin{enumerate}
    \item each $v \mapsto L_\varepsilon(v)$ is a smooth Lorentz--Finsler structure, time-orientable by the same vector field as $L$;
    \item for each compact set $\mathcal K \subseteq TM \setminus \mathbf 0$, there exists $\varepsilon_{\mathcal K}>0$ such that $L_\varepsilon|_{\mathcal K} = \mathcal{H}_\eps L|_{\mathcal K}$ for all $\eps \in (0, \eps_{\mathcal K}]$.
    \end{enumerate}
\end{enumerate}
\end{Proposition}

\begin{proof}
\ref{PropApproximation1}
For $\ell_a \star \rho_\eps \to \ell_a$, these assertions follow from properties of the convolution, noting that $\| f - f\star\rho_\eps\|_{\infty} \leq \| f\|_{C^{0,\alpha}} \eps^\alpha$ provided that $f\colon \R^n \times \R^n \rightarrow \R$ is $\alpha$-H\"older continuous; see \cite[Proposition 3.5]{Graf-SingularityC1} for details. 
Now, to pass from $\ell_a \star \rho_\eps$ to $\mathcal{H}_{\eps, a}L$, we make use of the homogeneity of $\ell_a$ to write
\begin{align*}
    &\partial_x^\beta \partial_v^\sigma \mathcal{H}_{\eps, a}L(x,v) - \partial_x^\beta \partial_v^\sigma \ell_a(x,v) \\
    &= \partial_x^\beta \partial_v^\sigma \biggl[ r_a(x,v)^2 \biggl((\ell_a \star \rho_\eps)\biggl(x,\frac{v}{r_a(x,v)}\biggr) -  \ell_a\biggl(x,\frac{v}{r_a(x,v)}\biggr)\biggr) \biggr],
\end{align*}
and note that $r_a(x,v)^2$ and $v/r_a(x,v)$ are smooth on $\mathcal{K} \subseteq TM \setminus \mathbf{0}$. 
Hence, by the product and chain rules, the above estimates pass to $\mathcal{H}_{\eps}L$.

\ref{PropApproximation2}
Let $X$ be the timelike vector field determining the time orientation of $L$. 
For each compact set $ B \subseteq M$, the $\tilde g$-unit tangent bundle $\mathbb S_B$ over $B$ is compact. 
From the local uniform convergence of $\mathcal H_\varepsilon L$ and its second $v$-derivatives to those of $L$ established in \ref{PropApproximation1}, it follows that there exists $\varepsilon_B > 0$ such that, for every $\varepsilon \in (0,\varepsilon_B]$, the symmetric matrix $\partial_{v^i}\partial_{v^j}\mathcal H_\varepsilon L$ has Lorentzian signature on $\mathbb{S}_B$, as well as $\mathcal H_\varepsilon L(X(x)) < 0$ for every $x \in B$. 
Since $\mathcal H_\varepsilon L$ is positively $2$-homogeneous, its vertical Hessian is positively $0$-homogeneous. 
Hence, the signature assertion holds on all of $\pi^{-1}(B)\setminus \mathbf 0$.

Using the parameter construction in the proof of the \emph{globalization lemma} \cite[Lem.\ 4.3]{HoermannKunzingerSteinbauer-WaveEq}, with a compact exhaustion of $M$, we can therefore choose a smooth function $\mu : (0,1)\times M \rightarrow (0,1)$ such that, for every $(\varepsilon,x)$, the function $\mathcal{H}_{\mu(\varepsilon,x)}L$ has these properties on $T_xM\setminus\{0\}$, and such that $\mu(\varepsilon,x)=\varepsilon$ on a neighborhood of each fixed compact subset of $M$ for sufficiently small $\varepsilon >0$.
Set
\[
    L_\varepsilon(v)
    := \mathcal{H}_{\mu(\varepsilon,\pi(v))}L(v).
\]
This assignment is smooth away from the zero section and positively $2$-homogeneous. 
Moreover, since $\mu$ depends only on the base point,
\[
    \partial_{v^i}\partial_{v^j}L_\varepsilon(v)
    =            \partial_{v^i}\partial_{v^j}\mathcal{H}_{\mu(\eps,x)}L(v),
    \quad x=\pi(v).
\]
Thus, $L_\varepsilon$ has Lorentzian signature and is time-oriented by $X$, proving (a). 
For a compact set $\mathcal K \subseteq TM\setminus \mathbf 0$, eventually $\mu(\varepsilon,\cdot)=\varepsilon$ on a neighborhood of $\pi(\mathcal K)$, thus $L_\varepsilon=\mathcal H_\varepsilon L$ on a neighborhood of $\mathcal K$, proving (b).
\end{proof}

\begin{Remark}[On low regularity Finsler structures of general signature]
\label{Remark: general semi Riemann}
A majority of the regularization results we have presented so far apply to any signature.
One may define a \emph{semi-Riemann--Finsler structure} of regularity $C^{(k,l),\alpha}$ for $k + l \geq 2$ and $\alpha \in [0,1]$ in the same way as a Lorentz--Finsler structure in Definition \ref{Definition: Lorentz Finsler structure}, with the only difference being that the signature of $\partial^2 L/\partial v^i \partial v^j$ can be of general semi-Riemannian type $(-,\dots,-,+,\dots,+)$ with $s_-$ minuses and $s_+$ pluses, $s_- + s_+ = n$.
The case $s_- = 1$ corresponds to the Lorentz--Finsler setting, $s_- = 0$ to the Riemann--Finsler setting.

Given a $C^{(k,l),\alpha}$ semi-Riemann--Finsler structure $L$, one can construct the approximations $\mathcal{H}_{\eps}L$ and $L_\varepsilon$, which have the properties described in Proposition \ref{prop:PropertiesApproximation}, upon replacing all instances of ``Lorentz'' with ``semi-Riemann''.
Concerning the (weighted) Ricci curvature, the convergence is as good as in the Lorentzian case, namely
\begin{equation*}
    \Ric[L_\eps] -\mathcal{H}_\eps \Ric[L] \to 0
\end{equation*}
locally uniformly on $TM\setminus \mathbf 0$ if $L \in C^{(1,3)}$, which is proven along the same lines used in the proof of Proposition \ref{Prop: ApproxtimelikeRiccibound} below.
\end{Remark}

\subsection{Inner and outer approximations}\label{ssc:L_e^}

While the smooth one-parameter family of regularizations $(L_\eps)_{\eps>0}$ as in Proposition~\ref{prop:PropertiesApproximation}\ref{PropApproximation2} approximates $L$ well analytically, it does not incorporate any considerations concerning causality. To remedy this deficiency, we now introduce two further approximation schemes $\widecheck{L}_\varepsilon$ and $\widehat{L}_\varepsilon$ (in the spirit of \cite{ChruscielGrant,Graf-SingularityC1}), whose causality is strongly tied to that of $L$. 

\begin{Definition}[The ``wider lightcone" relation $\prec$]
\label{df:order}
We write $L_1 \prec L_2$, for two (not necessarily smooth) Lorentz--Finsler structures with time-orientation, if any future causal vector for $L_1$ is future timelike for $L_2$, i.e., if the future light cone of $L_2$ is ``wider" than that of $L_1$.    
\end{Definition}

\begin{Proposition}[Chru{\'s}ciel--Grant approximations]\label{prop:ChruscielGrantApprox}
Let $(M,L)$ be a $C^{(1,2)}$ Finsler spacetime.
Then there exist smooth Lorentz--Finsler structures $\widecheck{L}_\varepsilon$ and $\widehat{L}_\varepsilon$ on $M$, for $\eps>0$, time-oriented by the same vector field as $L$, such that
\begin{enumerate}
    \item \label{CG1}
    $\widecheck{L}_\varepsilon \prec L \prec \widehat{L}_\varepsilon$;
    \item \label{CG2}
    $\widecheck{L}_\varepsilon - L_\varepsilon \to 0$ and $\widehat{L}_\varepsilon - L_\varepsilon \to 0$ in $C^\infty_{\loc}(TM \setminus \mathbf 0)$ as $\eps \to 0$;
    \item \label{CG3}
    for any compact set $\mathcal{K} \subseteq TM \setminus \mathbf{0}$, there exist $c_{\mathcal{K}} >0$ and $\eps_0(\mathcal{K})>0$ such that, for all $\eps \in (0,\eps_0(\mathcal{K}))$, the fundamental tensors of $\widecheck{L}_\varepsilon$ and $L_\varepsilon$ satisfy
    \begin{equation}\label{eq:EstimatesInverseMetric}
        \| \check{g}_\eps- g_\eps \|_{\infty,\mathcal{K}} \leq c_{\mathcal{K}} \eps, \qquad \| \check{g}_\eps^{-1} - g_\eps^{-1} \|_{\infty,\mathcal{K}} \leq c_{\mathcal{K}} \eps,
    \end{equation}
    and we have
    \begin{equation}\label{eq:EstimatesCheckSpray}
        \| \widecheck{G}_{\eps}^{i} - G_\eps^{i} \|_{\infty, \mathcal{K}} \leq c_\mathcal{K} \eps, \qquad \| (\widecheck{N}_{\eps})^{i}_j - (N_\eps)^{i}_j \|_{\infty, \mathcal{K}} \leq c_\mathcal{K} \eps
    \end{equation}
    for the spray coefficients and the nonlinear connection coefficients.
\end{enumerate}
\end{Proposition}

\begin{proof}
We will only construct $\widecheck{L}_\eps$; the construction of $\widehat{L}_\eps$ is analogous.
Let $X$ be a smooth timelike vector field determining the time orientation of $(M,L)$. We cover $M$ by a locally finite family of charts $\{(U_a, \psi_a)\}_{a \in A}$ with relatively compact domains, and let $\{\xi_a\}_{a \in A}$ denote a respective partition of unity.

Now, fix a chart $(U_a,\psi_a)$ and assume that $X^n \neq 0$ in $U_a$. 
We set
\begin{equation*}
    \nu_a^{i} :=\diff x^{i} - \frac{X^{i}}{X^n} \diff x^{n}
\end{equation*}
for $1 \leq i \leq n-1$.
Then, $\{\nu^1_a, \dots, \nu^{n-1}_a, \diff x^n\}$ is a smooth coframe on $U_a$ such that $\nu^i_{a}(X) = 0$ for all $i$, and $\nu_a(w):=\sum_{i=1}^{n-1} \nu_a^i (w)^2$, $w \in TU_a$, is $2$-homogeneous with $\nu_{a}(w)=0$ precisely when $w$ is proportional to $X$.
By the relative compactness of $U_a$, the function $\nu_a$ attains the minimum value $\kappa_a$ on the compact set 
\begin{equation*}
    \mathcal{V}_a:=\{w \in T\overline{U}\!_a  \cap (\Omega^C \setminus \Omega^T): |w|_{\tilde g} = 1\},
\end{equation*}
where $\tilde g$ is again an arbitrary Riemannian background metric.
This constant $\kappa_a$ is positive, as $\nu_a(w) = 0$ for some $w \in \Omega_x^C \setminus \Omega_x^T$ implies that $w = c X(x)$ for some $c >0$ (both are future-directed), but then $L(w) = c^2 L(X(x)) < 0$, a contradiction.

Moreover, by Proposition \ref{prop:PropertiesApproximation}, there exist constants $c_a,\eps_a>0$ such that $|L_\eps(w) - L(w)| \leq c_a \eps$ for any $w \in \mathcal{V}_a$ and $\eps \in (0,\eps_a]$.
Choosing a constant $C_a > c_a/\kappa_a$, we now set
\begin{equation*}
    \widetilde{L}_\eps := L_\eps + \eps \sum_{a \in A} C_a\xi_a(\pi) \nu_{a}.
\end{equation*}
We claim that $\widetilde{L}_\eps \prec L$ on $T \overline{U}\!_a$ for $\eps \in (0, \eps_a]$. 
To this end, let $w \in T \overline{U}\!_a \cap (\Omega^C \setminus \Omega ^T)$ and suppose, without loss of generality, that $|w|_{\tilde g} = 1$ (i.e., $w \in \mathcal{V}_a$).
Since
\begin{equation*}
    L_\eps(w) = \sum_{a \in A} \xi_a \bigl( \pi(w) \bigr) \bigl( L_\eps(w) - L(w)\bigr) \\
    \geq - \eps \sum_{a \in A} \xi_a \bigl( \pi(w) \bigr)c_a
\end{equation*}
for $\eps \in (0, \eps_a]$, we have
\begin{equation*}
    \widetilde{L}_\eps(w) = L_\eps(w) + \eps \sum_{a \in A} C_a\xi_a \bigl(\pi(w) \bigr) \nu_a(w) \geq \eps \sum_{a \in A} \xi_a \bigl(\pi(w)\bigr) \bigl(C_a\nu_a(w) - c_a \bigr) > 0,
\end{equation*}
implying that, indeed, $\widetilde{L}_\eps \prec L$ on $T \overline{U}\!_a$ for $\eps \in (0, \eps_a]$.

Furthermore, setting $Q:=\sum_{a \in A} C_a\xi_a(\pi) \nu_{a}$, for any compact set $\mathcal{K} \subseteq TM \setminus \mathbf{0}$ and any $l \in \N$, there exist $c_{\mathcal{K},l},\eps_{\mathcal{K},l}>0$ such that 
\begin{equation} \label{eq:CGProofLCheckLEps}
    \| \widetilde{L}_\eps - L_\eps\|_{C^l(\mathcal{K})} =  \eps \| Q\|_{C^l(\mathcal{K})} \leq c_{\mathcal{K},l} \eps
\end{equation}
for all $\eps \in (0,\eps_{\mathcal{K},l}]$.
Hence, $\widetilde{L}_\eps - L_\eps \rightarrow 0$ in $C^{\infty}_{\loc}(TM \setminus \mathbf{0})$.
Then we also immediately see that $(\tilde g_\varepsilon)_{ij} = (g_\eps)_{ij} + \eps q_{ij}$, where $q_{ij}:=\partial_{v^i}\partial_{v^j} Q$, satisfies
\begin{equation} \label{eq:CGProofMetrics}
    \| \tilde g_\eps - g_\eps \|_{\infty, \mathcal{K}} \leq c_{\mathcal{K}} \eps
\end{equation}
on any compact set $\mathcal{K} \subseteq TM \setminus \mathbf{0}$, for some $c_{\mathcal K},\eps_{\mathcal K}>0$, and for all $\eps \in (0,\eps_{\mathcal K}]$.
This shows that $\tilde{g}_\eps$ has the same signature as $g_\eps$ on $\mathcal{K}$ for sufficiently small $\eps>0$.

We now use the same base-dependent globalization as in the proof of Proposition \ref{prop:PropertiesApproximation}\ref{PropApproximation2}.
For each compact $B \subseteq M$, apply the preceding estimates on the $\tilde g$-unit tangent bundle $\mathbb{S}_B$. 
By homogeneity, for all sufficiently small $\varepsilon>0$, the vertical Hessian of $\widetilde L_\varepsilon$ has Lorentzian signature on $\pi^{-1}(B)\setminus \mathbf 0$, and $\widetilde L_\varepsilon \prec L$ over $B$. 
Moreover, $X$ is $\widetilde{L}_\eps$-timelike since $Q(X)=0$. The parameter construction in \cite[Lem.\ 4.3]{HoermannKunzingerSteinbauer-WaveEq}
therefore gives a smooth function $\mu:(0,1) \times M \to (0,1)$ such that
\[
    \widecheck L_\varepsilon(v)
    := \widetilde L_{\mu(\varepsilon,\pi(v))}(v)
\]
is a smooth Lorentz--Finsler structure time-oriented by $X$, with $\widecheck L_\varepsilon \prec L$ everywhere. 
As in Proposition \ref{prop:PropertiesApproximation}\ref{PropApproximation2}, $\widecheck L_\varepsilon$ is positively $2$-homogeneous, and the vertical Hessian is obtained by evaluating that of $\widetilde L_{\mu(\eps,x)}$ with $x=\pi(v)$.
Furthermore, for every compact $\mathcal K \subseteq TM\setminus \mathbf 0$, we have $\widecheck L_\varepsilon=\widetilde L_\varepsilon$ on a
neighborhood of $\mathcal K$ for sufficiently small $\varepsilon >0$.
This completes the construction in \ref{CG1}, and the convergence claim in \ref{CG2} now directly follows from \eqref{eq:CGProofLCheckLEps}.

For the estimates in \ref{CG3}, notice that \eqref{eq:CGProofMetrics} immediately gives the first estimate in \eqref{eq:EstimatesInverseMetric}.
Moreover, since $g_\eps \rightarrow g$ locally uniformly and $g$ is non-degenerate, both $g_\eps$ and $\check g_\eps$ are uniformly non-degenerate on $\mathcal{K}$ for all $\eps \in (0,\eps_{\mathcal{K}}]$. 
In particular, their inverses are uniformly bounded on $\mathcal{K}$. 
Using
\begin{equation} \label{eq:matrixInversion}
    \check g_\eps^{-1}-g_\eps^{-1} = \check g_\eps^{-1}(g_\eps-\check g_\eps)g_\eps^{-1},
\end{equation}
we thus obtain
\begin{equation} \label{eq:CGProofInverse}
    \|\check g_\eps^{-1}-g_\eps^{-1}\|_{\infty,\mathcal{K}} \leq c_{\mathcal{K}}\eps.
\end{equation}
For the estimate of the spray coefficients, by \eqref{def:SprayCoeffImproved}, this amounts to analyzing the difference
\begin{equation}
\label{eq:G_eps}
    \widecheck G_\eps^i -G_\eps^i =( \check g_\eps^{il}-g_\eps^{il})
    (\widecheck{A}_\eps)_l + g_\eps^{il}\bigl( (\widecheck{A}_\eps)_l - (A_\eps)_l \bigr),
\end{equation} 
where $(A_\eps)_l := v^j \partial_{x^j} \partial_{v^l} L_\eps - \partial_{x^l} L_\eps$ and correspondingly for $(\widecheck A_\eps)_l$.
The first term is $O(\eps)$ thanks to \eqref{eq:CGProofInverse} and $\widecheck{L}_\eps \rightarrow L$ in $C^{(1,2)}_\loc(TM \setminus \mathbf{0})$, while the second term is $O(\eps)$ by \eqref{eq:CGProofLCheckLEps} and the local uniform boundedness of $(g_\eps)^{-1}$. 

Finally, to prove the same estimate for the nonlinear connection coefficients, we first claim that
\begin{equation*}
    \| \partial_v (\check g_\eps^{-1}-g_\eps^{-1}) \|_{\infty,\mathcal K}\leq c_{\mathcal K}\eps.
\end{equation*}
For this, we use the identity
\begin{equation*}
    \check g_\eps^{-1}-g_\eps^{-1}
    = -\eps \check g_\eps^{-1} q g_\eps^{-1}, \qquad q=(q_{ij}), 
\end{equation*}
following from $\widecheck{L}_\eps -L_\eps = \eps Q$, to get
\begin{equation*}
    \partial_{v}
    (\check g_\eps^{-1}-g_\eps^{-1}) 
    = -\eps (\partial_{v}\check g_\eps^{-1}) qg_\eps^{-1} 
    -\eps \check g_\eps^{-1}(\partial_{v}q)g_\eps^{-1}
    -\eps \check g_\eps^{-1}q(\partial_{v}g_\eps^{-1}).
\end{equation*}
Now, all the terms appearing in this expression are uniformly bounded on $\mathcal K$; the inverse metrics are uniformly bounded by the argument above, $q$ and $\partial_v q$ are smooth, and $\partial_v g_\eps^{-1}$, $\partial_v \check g_\eps^{-1}$ are uniformly bounded by the $C^{(1,2)}_{\mathrm{loc}}(TM \setminus \mathbf{0})$-convergence of $L_\eps$ and $\widecheck{L}_\eps$ to $L$.

Moreover, from \eqref{eq:CGProofLCheckLEps}, we get that $(\widecheck{A}_\eps - A_\eps)$ and $\partial_v (\widecheck{A}_\eps - A_\eps)$ are $O(\eps)$ on $\mathcal{K}$, while both $\widecheck{A}_\eps$ and $\partial_v \widecheck{A}_\eps$ are uniformly bounded on $\mathcal{K}$ as $\widecheck{L}_\eps \rightarrow L$ in $C^{(1,2)}_\loc(TM \setminus \mathbf{0})$. 
Differentiating the expression \eqref{eq:G_eps} for
$\widecheck G_\eps-G_\eps$, we obtain
\begin{equation}\label{eq:ProofCGN}
\begin{split}
    \partial_{v^j} \bigl(\widecheck G_\eps^i -G_\eps^i \bigr)
    &= \partial_{v^j} (\check g_\eps^{il}-g_\eps^{il})
    (\widecheck{A}_\eps)_l + (\check g_\eps^{il}-g_\eps^{il})
    \partial_{v^j} (\widecheck{A}_\eps)_l \\
    &\quad + \partial_{v^j} g_\eps^{il} \bigl( (\widecheck{A}_\eps)_l - (A_\eps)_l \bigr) + g_\eps^{il} \partial_{v^j} \bigl((\widecheck{A}_\eps)_l - (A_\eps)_l\bigr).
\end{split}
\end{equation}
The above arguments imply that each of these terms is $O(\eps)$ on $\mathcal{K}$, and hence, as desired,
\begin{equation*}
    \|(\widecheck N_\eps)^i_j-(N_\eps)^i_j\|_{\infty,\mathcal{K}}
    \leq c_{\mathcal{K}}\eps. \qedhere
\end{equation*}
\end{proof}

\subsection{Convergence of Ricci curvature}
\label{ssc:convergence}

We next analyze the behavior of the Ricci curvature for approximating Lorentz--Finsler structures $L_\eps$ and $\widecheck{L}_\eps$, denoted by $\Ric[L_\eps]$ and $\Ric[\widecheck{L}_\eps]$, as $\eps \to 0$.
Recall that $L_\eps$ was obtained from $\mathcal{H}_\eps L$ in Proposition \ref{prop:PropertiesApproximation}, and $\widecheck{L}_\eps$ from $\widetilde{L}_\eps$ in the proof of Proposition \ref{prop:ChruscielGrantApprox}, both via the globalization lemma.
Thus, on each compact set $\mathcal{K} \subseteq TM \setminus \mathbf{0}$, we have $L_\eps =\mathcal{H}_\eps L$ and $\widecheck{L}_\eps =\widetilde{L}_\eps$ for sufficiently small $\eps>0$.

\begin{Lemma}[Uniform convergence of $\Ric$]\label{lem:RicciCheckConvergence}
    Let $(M,L)$ be a $C^{(1,2)}$ Finsler spacetime. Then, on any compact set $\mathcal{K} \subseteq TM \setminus \mathbf{0}$, we have $\Ric [\widecheck{L}_\eps] - \Ric[L_\eps] \rightarrow 0$ uniformly as $\eps \to 0$.
\end{Lemma}

\begin{proof}
In view of \eqref{eq:DistributionalCurvature}, we shall estimate the difference
\begin{align*}
    \Ric[\widecheck L_\eps] - \Ric[L_\eps]
    &= \partial_{x^i}(\widecheck G_\eps^i-G_\eps^i)
    - \partial_{x^k}\bigl((\widecheck N_\eps)^i_{i}-( N_\eps)^i_{i}\bigr)v^k \\
    &\quad + \partial_{v^k}\bigl((\widecheck N_\eps)^i_{i}-(N_\eps)^i_{i}\bigr)\widecheck G_\eps^k 
    + \partial_{v^k}(N_\eps)^i_{i}\bigl(\widecheck G_\eps^k-G_\eps^k\bigr)\\
    &\quad - \bigl((\widecheck N_\eps)^i_{k}(\widecheck N_\eps)^k_{i}
    - (N_\eps)^i_{k}(N_\eps)^k_{i}\bigr).
\end{align*}
Our assumptions on the regularity of $L$ imply that $G^{i} \in C^{(0,1)}$ and $N^{i}_k \in C^{(0,0)}$; recall Remark \ref{rm:regular}. 
We first note that the last term 
\begin{equation*}
    (\widecheck N_\eps)^i_{k}(\widecheck N_\eps)^k_{i} 
    - (N_\eps)^i_{k}(N_\eps)^k_{i}
    = \bigl((\widecheck N_\eps)^i_{k}-(N_\eps)^i_{k}\bigr)(\widecheck N_\eps)^k_{i}
    + (N_\eps)^i_{k}\bigl((\widecheck N_\eps)^k_{i}-(N_\eps)^k_{i}\bigr)
\end{equation*}
tends to $0$ uniformly. 
This is because $N_\eps, \widecheck N_\eps$ are locally uniformly bounded and 
\begin{equation*}
    \|(\widecheck N_\eps)^i_k -(N_\eps)^i_k\|_{\infty,\mathcal{K}} \leq c_{\mathcal{K}}\eps
\end{equation*}
by the estimate \eqref{eq:EstimatesCheckSpray}. 

Next, we establish the convergence of the first three terms by showing
\begin{equation}
\label{eq:GNN}
    \|\partial_x(\widecheck G_\eps-G_\eps)\|_{\infty,\mathcal{K}}
    +
    \|\partial_x(\widecheck N_\eps-N_\eps)\|_{\infty,\mathcal{K}}
    +
    \|\partial_v(\widecheck N_\eps-N_\eps)\|_{\infty,\mathcal{K}}
    \rightarrow 0.
\end{equation}
For the first term $\partial_{x^i} (\widecheck{G}^{i}_\eps - G^{i}_\eps)$, we can essentially follow the same argument we gave for the proof of the estimate on the nonlinear connection coefficients in \eqref{eq:ProofCGN} with the important exception that now $\partial_{x^i} (\widecheck A_\eps)_l$, where 
\begin{equation*}
    (\widecheck A_\eps)_l= v^j \partial_{x^j} \partial_{v^l} \widecheck L_\eps - \partial_{x^l} \widecheck L_\eps
\end{equation*}
as in the above proof, is \emph{not} locally bounded since $\partial_{x^j} \partial_{v^l} L \in C^{(0,1)}$. 
However, we will show that
\begin{equation}\label{eq:eAto0}
    \eps \|\partial_{x^i} (\widecheck A_\eps)_l\|_{\infty, \mathcal{K}} \rightarrow 0 \quad \textnormal{as} \quad \eps \rightarrow 0,
\end{equation}
which is sufficient to conclude the convergence $\|\partial_{x^i}(\widecheck{G}_\eps^{i} - G_\eps^{i})\|_{\infty,\mathcal{K}} \to 0$ by virtue of \eqref{eq:EstimatesInverseMetric}. 
We reduce the situation to a standard mollifier estimate on $\R^{2n}$ as follows. 
First, as $\widecheck L_\eps = L_\eps + \eps Q$ on $\mathcal{K}$ (for small $\eps$), with smooth $Q$ as in the proof of Proposition \ref{prop:ChruscielGrantApprox}, it is sufficient to prove the claim for $L_\eps$.
Furthermore, in local coordinates, we can disregard cut-off and transition functions appearing in the convolution formula \eqref{eq:HeaL}, as they do not spoil the convergence (the terms arising from differentiating the smooth cut-off and transition functions contain at most first derivatives of the regularized coefficients and are therefore locally uniformly bounded). 
Hence, it suffices to study
\begin{equation}\label{eq: ProofRicReduction}
    \partial_{x^i}(f \star \rho_\eps),
\end{equation}
where $f$ is a continuous function on $\R^{2n}$ (representing either $\partial_{x^j} \partial_{v^l} L$ or $\partial_{x^l} L$). 
To prove this, using the integration by parts and $c \star \partial_{x^i}\rho_\eps=0$ for any constant $c$, we write
\begin{align*}
    \partial_{x^i}(f \star \rho_\eps)(z) 
    &=\int_{\R^{2n}} \partial_{x^i}f(z-y)\rho_\eps(y)\diff y
    =\int_{\R^{2n}} f(z-y)\partial_{x^i}\rho_\eps(y)\diff y \\ 
    &=\int_{\R^{2n}} \bigl(f(z-y)-f(z)\bigr)\partial_{x^i}\rho_\eps(y)\diff y.
\end{align*}
Letting $\omega_K$ denote the modulus of continuity of $f$ on a compact set $K \subseteq \R^{2n}$ and recalling $\rho_\eps(y)=\eps^{-2n}\rho(y/\eps)$, we have
\begin{alignat*}{2}
    |\eps\partial_{x^i} (f \star \rho_\eps)(z)|
    &\leq \eps \int_{\R^{2n}}\omega_K(|y|)|\partial_{x^i}\rho_\eps(y)|\diff y 
    =\eps^{-2n} \int_{\R^{2n}} \omega_K(|y|) |\partial_{x^i}\rho(y/\eps)| \diff y \\
    &= \int_{\R^{2n}}\omega_K(\eps|y|)|\partial_{x^i}\rho(y)| \diff y.
\end{alignat*}
Now, the right-hand side of this expression tends to $0$ by the dominated convergence theorem. 
Hence,
\begin{equation}
\label{eq:ConvergenceC0derivative}
    \eps\|\partial_{x^i}(f \star \rho_\eps)\|_{\infty,K}\rightarrow 0,
\end{equation}
giving the claim \eqref{eq:eAto0}. 

For the latter two terms in \eqref{eq:GNN}, we can apply a similar argument. 
Indeed, to control $\partial_{x^k} ((\widecheck N_\eps)^i_i - (N_\eps)^i_i)$, the potentially problematic terms are (recall \eqref{eq:ProofCGN})
\begin{equation*}
    \partial_{v^i}(\check g_\eps^{il} 
    - g_\eps^{il}) \partial_{x^k} (\widecheck A_\eps)_l, \qquad 
    (\check g_\eps^{il} - g_\eps^{il})\partial_{x^k}\partial_{v^i} (\widecheck A_\eps)_l, \qquad 
    \partial_{x^k}\partial_{v^i} \check g_\eps^{il}\bigl((\widecheck A_\eps)_l - (A_\eps)_l\bigr). 
\end{equation*}
All of these are controlled exactly as in the estimate for $\partial_{x^i}(\widecheck G^i_\eps-G^i_\eps)$.
Indeed, we have already seen that the respective factors involving the differences of $\widecheck{L}_\eps$- and $L_\eps$-quantities are $O(\eps)$, while the other factors can be reduced to a problem of the form \eqref{eq: ProofRicReduction} again. 
For $\partial_{v^k} ((\widecheck N_\eps)^i_i - (N_\eps)^i_i)$, the only two potentially problematic terms are
\begin{equation*}
    (\check g_\eps^{il}
    -g_\eps^{il})\partial_{v^k}\partial_{v^i} (\widecheck A_\eps)_l, \qquad \partial_{v^k} \partial_{v^i} g_\eps^{il}\bigl((\widecheck A_\eps)_l - (A_\eps)_l\bigr)
\end{equation*}
which can again be handled as above. 

Repeating the same argument one last time to get $\eps\|\partial_{v^k} (N_\eps)^i_i\|_{\infty,\mathcal{K}}\rightarrow 0$, we have thus proven that
$\Ric[\widecheck L_\eps] - \Ric[L_\eps] \rightarrow 0$
uniformly on $\mathcal{K}$, as claimed.
\end{proof}

To compare the Ricci curvature of $\widecheck{L}_\eps$ with that of $L$, we will need to improve the regularity from $C^{(1,2)}$ to $C^{(1,3)}$ (or $C^{(1,2),\alpha}$ with $\alpha>1/2$; see Remark \ref{rm:12alpha} below).
Note that, even under $L \in C^{(1,3)}(TM \setminus \mathbf{0})$, the Ricci curvature is defined only in the distributional sense (in horizontal directions) since it involves terms like $\partial_{x^i} G^i$ but $G^i \in C^{(0,2)}$.

\begin{Proposition}[Approximation of strong energy condition]
\label{Prop: ApproxtimelikeRiccibound}
    Let $(M,L)$ be a $C^{(1,3)}$ Finsler spacetime satisfying the distributional strong energy condition, and let $\mathcal{K} \subseteq \Omega^T$ be compact.
    Then for any $\delta >0$, there exists $\eps_0 =\eps_0(\mathcal{K},\delta) >0$ such that for all $\eps \in (0,\eps_0)$ and $v \in \mathcal{K}$ with $2 \widecheck{L}_\eps(v)=-1$, we have
    \begin{equation*}
        \Ric[\widecheck{L}_\eps](v) \ge - \delta.
    \end{equation*}
\end{Proposition}

\begin{proof}
We show that $\Ric[\widecheck L_\eps] - \mathcal{H}_\eps \Ric[L]$ tends to $0$ locally uniformly.
Then, as $\Ric[L] \geq 0$ on $\Omega^T$ in the distributional sense, we have $\mathcal{H}_\eps \Ric[L]\geq 0$ on $\mathcal{K} \subseteq \Omega^T$ for all sufficiently small $\eps >0$, giving the claim. 
To prove this convergence, we split the difference into two pieces:
\begin{equation*}
    \Ric[\widecheck L_\eps] - \mathcal{H}_\eps \Ric[L]= (\Ric[\widecheck L_\eps] - \Ric[L_\eps]) + (\Ric[L_\eps]-\mathcal{H}_\eps \Ric[L]).
\end{equation*}
By Lemma \ref{lem:RicciCheckConvergence}, we know that $\Ric[\widecheck{L}_\eps] - \Ric[L_\eps] \rightarrow 0$ locally uniformly, and thus the first piece tends to zero.
We shall show that the second piece also converges to zero.

For this, by homogeneity, it suffices to establish convergence on $\widetilde{\mathcal{K}}:= \{v/\|v\|_{\tilde g} : v \in \mathcal{K}\}$.
On a neighborhood of this set, for sufficiently small $\eps>0$, we have $L_\eps = \mathcal{H}_\eps L$.
By compactness, we can restrict to a fixed chart $(U_a, \psi_a)$.
We now estimate the terms arising from homogenization.
Recalling \eqref{eq:HeaL}, for $(x,v) \in \diff\psi_a(\widetilde{\mathcal{K}})$ we have
\begin{align*}
&\mathcal{H}_{\eps,a} L(x,v) - (\ell_a \star \rho_\eps)(x,v) \\
&= \int_{\R^n \times \R^n} \bigl\{ \ell_a\bigl( x-\eps y,v-\eps r_a(x,v)w \bigr) -\ell_a(x-\eps y,v-\eps w)\bigr\} \rho(y,w) \diff y \diff w.
\end{align*}
Hence, since $r_a(x,v)=1$,
\begin{align*}
&\partial_{x^i} \bigl( \mathcal{H}_{\eps,a} L(x,v) - (\ell_a \star \rho_\eps)(x,v) \bigr) \\
&= -\eps \partial_{x^i}r_a(x,v) \int_{\R^n \times \R^n} w^j \partial_{v^j} \ell_a(x-\eps y, v - \eps w) \rho(y,w) \diff y \diff w.
\end{align*}
As $\ell_a \in C^{(1,3)}$, it follows that $\partial_{x^i} \mathcal{H}_{\eps,a} L - \partial_{x^i}(\ell_a \star \rho_\eps)$ is $O(\eps)$. 
We can argue analogously for mixed derivatives $\partial^{\beta}_x \partial^{\sigma}_v$, provided $|\beta| \leq 1$ and $|\beta| + |\sigma| \leq 4$.
The case of two $x$-derivatives ($|\beta| = 2$, $|\sigma| \leq 2$) is handled as in the proof of Lemma \ref{lem:RicciCheckConvergence} using \eqref{eq:ConvergenceC0derivative}.
Hence, we may replace $\mathcal H_{\eps,a}L$ by the pure convolution term $\ell_a\star\rho_\eps$, up to an error converging to zero locally uniformly.

Now, recall that the coordinate expression for $\Ric[L]$ is given by
\begin{equation*}
    \frac{\partial G^i}{\partial x^i} - \frac{\partial N^i_i}{\partial x^k}v^k + \frac{\partial N^i_i}{\partial v^k}G^k - N^i_kN^k_i.
\end{equation*}
After expanding this, we see that the only terms containing second-order horizontal derivatives of $L$ are
\begin{equation*}
    g^{il} \partial_{x^i} A_l - \frac{1}{2}v^k \bigl( (\partial_{v^i}g^{il})\partial_{x^k}A_l + g^{il} \partial_{x^k} \partial_{v^i}A_l \bigr),
\end{equation*}
where we have used that $G^{i} = g^{il} A_l$ with $g^{il} \in C^{(1,1)}$ and $A_l := v^j \partial_{x^j} \partial_{v^l} L - \partial_{x^l} L \in C^{(0,2)}$, as well as $N_i^i = \frac{1}{2}((\partial_{v^i}g^{il})A_l + g^{il} \partial_{v^i}A_l) \in C^{(0,1)}$.
All the other terms contain at most first-order $x$-derivatives with total order at most $4$, hence the difference of these terms computed from $L_\eps$ and $L$ converges locally uniformly to zero by Proposition \ref{prop:PropertiesApproximation}\ref{PropApproximation1}.

For the terms containing second order $x$-derivatives, the key ingredient is the following \emph{Friedrichs-type commutator estimate}; cf.\ \cite[Lemma 4.9]{Graf-SingularityC1}:

Let $a_\eps \in C^1(\R^{2n})$, for $\eps \geq 0$, and suppose that $a_\eps \rightarrow a_0$ in $C^1_{\loc}$ as $\eps \rightarrow 0$   with $\| a_\eps - a_0\|_{\infty,K} \leq c_K \eps $ on any compact set $K \subseteq \R^{2n}$ with some $c_K>0$.
Then, for any $f \in C(\R^{2n})$, we have 
\begin{equation}\label{eq:Friedrichs}
    \partial_{z}^\sigma\bigl(a_\eps (f \star \rho_\eps)\bigr) - \partial_{z}^\sigma\bigl((a_0f) \star \rho_\eps\bigr) \rightarrow 0
\end{equation}
locally uniformly, where $z=(x,v)$, $|\sigma| \le 1$.
By the product rule and using $\partial_z^{\sigma} a_\eps \rightarrow \partial_z^{\sigma} a_0$, this takes the more suggestive form
\begin{equation*}
    a_\eps \partial_z^{\sigma} (f\star \rho_\eps) - (a_0 \partial_z^{\sigma} f)\star \rho_\eps\rightarrow 0,
\end{equation*}
where $\partial_z^{\sigma} f$ is interpreted distributionally.

We can apply this to our setting by noticing that $g^{il}_\eps \rightarrow g^{il}$ and $\partial_{v^i}g_\eps^{il} \rightarrow \partial_{v^i}g^{il}$ in $C^1_\loc$ (for the latter we need $L \in C^{(1,3)}$) with the norm of their difference being $O(\eps)$ on any compact set, again by Proposition \ref{prop:PropertiesApproximation} (and the formula for matrix inversion; cf.\ \eqref{eq:matrixInversion}).

Hence,
\begin{equation*}
    \Ric[L_\eps] - \sum_{a \in A} (\chi_a \circ \pi) (\psi_a)^{*} \bigl( (\psi_a)_*\bigl( (\zeta_a \circ \pi) \Ric[L] \bigr) \star \rho_\eps \bigr) \rightarrow 0
\end{equation*}
locally uniformly on $\widetilde{\mathcal{K}}$.
On this set, the latter term agrees with $\mathcal{H}_\eps \Ric[L]$ and by homogeneity, the convergence holds on all of $\mathcal K$, completing the proof.
\end{proof}

\begin{Remark}[Case of $L \in C^{(1,2),\alpha}$ with $\alpha>1/2$]\label{rm:12alpha}
\begin{enumerate}
\item[(a)]
One can alternatively prove Proposition \ref{Prop: ApproxtimelikeRiccibound} for $L \in C^{(1,2),\alpha}$ with $\alpha > 1/2$. 
In this case $g^{ij}\in C^{(1,0),\alpha}, A_l\in C^{(0,1),\alpha}$.
Hence, the first term in
\[
2N^i_i =(\partial_{v^i} g^{il}) A_l +g^{il} \partial_{v^i} A_l
\]
is merely the product of two locally $\alpha$-H\"older continuous functions.
Locally on $\R^{2n}$, the corresponding terms in the difference $\Ric[L_\eps] - \Ric[L]\star\rho_\eps$ hence reduce to a differentiated commutator of the form
\begin{equation*}
    \partial_z^{\sigma} \bigl(a_\eps(f \star \rho_\eps) - (a_0f) \star \rho_\eps \bigr)
\end{equation*}
where $z=(x,v)$, $| \sigma| \leq 1$ and $a_\eps,f\in C^{0,\alpha}_{\loc}(\R^{2n})$.
Here, the commutator estimate gives
    \begin{equation*}
        \bigl\|\partial_z^{\sigma} \bigl(a_\eps(f \star \rho_\eps) - (a_0f) \star \rho_\eps \bigr)\bigr\|_{\infty, K} \leq c_K \eps^{2 \alpha - 1}
    \end{equation*}
for any compact set $K \subseteq \R^{2n}$.
In particular, this applies to the commutator
\begin{equation*}
        \partial_{x^k} \Bigl( \bigl( (\partial_{v^i} g^{il}) \star \rho_\eps \bigr) (A_l \star \rho_\eps) -\bigl( (\partial_{v^i} g^{il}) A_l \bigr)\star \rho_\eps \Bigr)
    \end{equation*}
    in the second term of the Ricci curvature.
    The commutators arising from the third term are treated in the same way after rewriting
    \begin{equation*}
        (\partial_{v^k} N_i^{i}) G^k = \partial_{v^k}(N_i^{i}G^k) - N_i^{i} \partial_{v^k} G^k,
    \end{equation*}
    where $N^i_i$ is $\alpha$-H\"older continuous.
    As $2 \alpha -1 >0$, all of these terms converge to zero as $\eps \rightarrow 0$.
    The remaining terms are handled exactly as in the proof of Proposition \ref{Prop: ApproxtimelikeRiccibound}.

    \item[(b)]
    The condition $L \in C^{(1,2),\alpha}$ implies $g_{ij} \in C^{(1,0),\alpha}=C^{1,\alpha}$.
    Hence, for purely Lorentzian metrics $g$, this condition is stronger than $g \in C^{1}$; even though, for general $L$, it reduces the demands on vertical regularity.
    In this sense, these two conditions are independent.
\end{enumerate}
\end{Remark}

Similarly to Proposition \ref{Prop: ApproxtimelikeRiccibound}, one obtains the following approximation result for timelike lower Ricci curvature bounds with general $K \in \R$. 

\begin{Corollary}[Approximation of timelike Ricci curvature bounds]
\label{cor:GeneralRicciBound}
    Let $(M,L)$ be a $C^{(1,3)}$ Finsler spacetime such that $\mathrm{Ric} \geq K$ timelike distributionally, where $K \in \R$. 
    Then for any compact set $\mathcal{K} \subseteq \Omega^T$ and $\delta > 0$, there exists $\varepsilon_0 = \varepsilon_0(\mathcal{K},\delta) > 0$ such that for all $\varepsilon \in (0,\varepsilon_0)$ and $v \in \mathcal{K}$ with $2 \widecheck L_\varepsilon(v) = -1$, we have
    \begin{equation*}
        \Ric[\widecheck L_\varepsilon](v) \geq K - \delta.
    \end{equation*}
\end{Corollary}

\begin{proof}
By assumption $\Ric[L]+2KL\geq0$ timelike distributionally, thus
\[
    \mathcal H_\varepsilon\Ric[L]
    +2K \mathcal H_\varepsilon L \geq 0
\]
pointwise on $\mathcal K$ for $\eps >0$
small enough, since the kernels defining $\mathcal H_\varepsilon$ are
nonnegative and supported in $\Omega^T$.
In the proof of
Proposition \ref{Prop: ApproxtimelikeRiccibound}, we have already established that
\[
    \operatorname{Ric}[\widecheck L_\varepsilon]
    -\mathcal H_\varepsilon\Ric[L]\rightarrow0
\]
locally uniformly.
Since also
$\widecheck L_\varepsilon-\mathcal H_\varepsilon L\to0$ locally uniformly by Proposition \ref{prop:PropertiesApproximation}, the claim follows.
\end{proof}

Next, we consider the null case.

\begin{Proposition}[Approximation of null energy condition]
\label{Prop: Null approximate Ricci}
    Let $(M,L)$ be a $C^{(1,3)}$ Finsler spacetime satisfying the distributional null energy condition. 
    Then for any compact set $\mathcal K \subseteq TM \setminus \mathbf 0$ and $\delta > 0$, there exists $\varepsilon_0 = \varepsilon_0(\mathcal K,\delta)>0$ such that for all $\varepsilon \in (0,\varepsilon_0)$ and $v \in \mathcal{K} \cap \Omega^C[\widecheck{L}_\eps]$ with $\widecheck{L}_\varepsilon(v) = 0$, we have
    \begin{equation*}
        \Ric[\widecheck{L}_\eps](v) \ge - \delta.
    \end{equation*}
\end{Proposition}

\begin{proof}
By positive $2$-homogeneity, it suffices to consider the case where $\mathcal K$ is contained in the $\tilde g$-unit tangent bundle over one of the relatively compact base neighborhoods, denoted by $U$, used in the construction of the approximation. 
For each small $\varepsilon>0$ and $v\in \mathcal K$, denote by $\mu_{\varepsilon,v}$ the nonnegative smooth test density on $\Omega^C$ for which
\[
    \mathcal H_\eps\Ric[L](v)
    =
    \langle\Ric[L],\mu_{\varepsilon,v}\rangle.
\]
Its total mass
\[
    a_{\varepsilon,v}:=\int_{TM}\mu_{\varepsilon,v}
\]
converges uniformly to $1$ on $\mathcal K$ due to the partition of unity construction.
Let $\mathcal K'\subseteq TM\setminus \mathbf 0$ be a compact set containing all these supports for sufficiently small $\varepsilon$.
We claim that there exists $c>0$ such that
\begin{equation*}
    \frac{\mu_{\varepsilon,v}}{a_{\varepsilon,v}}
    \in M_{c\varepsilon}(\mathcal K')
    \qquad
    \text{for all }v\in \mathcal{K} \cap\Omega^C[\widecheck L_\varepsilon] \cap \widecheck L_\eps^{-1}(0).
\end{equation*}
Indeed, Proposition \ref{prop:PropertiesApproximation} and the local Lipschitz continuity of $L$ give constants $A,B>0$ such that
\[
    |L_\varepsilon(v)-L(v)|\leq A\varepsilon,
    \qquad
    |L(w)-L(v)|\leq B\varepsilon
    \quad\text{for }w\in\operatorname{supp}\mu_{\varepsilon,v}.
\]
The functions $\nu_a$ used in the proof of Proposition \ref{prop:ChruscielGrantApprox} are strictly positive on $\tilde g$-unit future $L$-null vectors in their respective charts.
Possibly choosing larger constants $C_a$ in the construction of $Q$ (cf.\ again the proof of Proposition \ref{prop:ChruscielGrantApprox}), we can ensure that
\[
    Q(v)\geq A+B+1
\]
on a neighborhood of future $\tilde g$-unit null vectors with basepoints in $U$.
Since $\widecheck L_\varepsilon\to L$, the $\tilde g$-unit future $\widecheck L_\varepsilon$-null vectors are contained in that neighborhood for sufficiently small $\varepsilon > 0$. 
For such a vector $v$, using $\widecheck L_\varepsilon=L_\varepsilon+\varepsilon Q$, we obtain
\[
    L(v)
    \leq L_\varepsilon(v)+A\varepsilon
    =-\varepsilon\bigl(Q(v)-A\bigr).
\]
Consequently, for every $w\in\operatorname{supp}\mu_{\varepsilon,v}$ with $\widecheck{L}_\eps$-null $v$ as above,
\[
    -\varepsilon\bigl(Q(v)+A+B\bigr)
    =L_\eps(v) -\eps(A+B)
    \leq L(w)
    \leq-\varepsilon\bigl(Q(v)-A-B\bigr)
    \leq-\varepsilon.
\]
Since $Q$ is bounded on $\mathcal K$, it follows that
\[
    \operatorname{supp}\mu_{\varepsilon,v}
    \subseteq
    \{|L|\leq c\varepsilon\}\cap\Omega^C\cap \mathcal K'
\]
for some $c>0$, proving the claim.

By the distributional null energy condition, the uniform boundedness of $a_{\varepsilon,v}$ together with the first part of the proof implies
\[
    \liminf_{\varepsilon\to0}
    \inf_{v\in \mathcal K\cap\Omega^C[\widecheck L_\eps] \cap \widecheck L_\eps^{-1}(0)}
    \mathcal H_\varepsilon\Ric[L](v)\geq0.
\]
Together with the locally uniform convergence $\Ric[\widecheck L_\varepsilon] -\mathcal H_\varepsilon\Ric[L]\rightarrow0$ established in the proof of Proposition \ref{Prop: ApproxtimelikeRiccibound}, we obtain the claimed estimate.
\end{proof}

\subsection{Convergence of weighted Ricci curvature}

In the weighted case, a regularization of the measure is needed, for which we proceed as in the construction of $\mathcal H_\varepsilon$; see Subsection \ref{ssc:L_e} for the notation.
Consider a convolution kernel as in \eqref{eq:HeaL} of the form
\[
    \rho(y,w)=\rho^\mathsf x(y)\rho^\mathsf v (w),
\]
where $\rho^\mathsf x,\rho^\mathsf v$ are nonnegative smooth kernels on $\R^n$ of unit integral, supported in $B_1(0)$, and $\rho^\mathsf x$ is positive on a neighborhood of $0$.
Given a $C^1$ measure $\mathfrak m$ on $M$, write
\[
    (\psi_a)_*(\mathfrak m|_{U_a})
    =\sigma_a(x) \diff x^1 \cdots \diff x^n,
    \qquad
    p_a:=(\zeta_a\circ\psi_a^{-1})\sigma_a,
\]
where $p_a$ is extended by $0$ outside $\psi_a(U_a)$.
Define
\begin{equation*}
    \mathfrak m_\varepsilon
    :=
    \sum_{a\in A}
    \chi_a\,\psi_a^*
    \left[
        (p_a \star\rho^\mathsf x_\varepsilon)\,\diff x^1 \cdots \diff x^n
    \right],
    \qquad
    \rho^{\mathsf{x}}_\varepsilon
    :=\varepsilon^{-n}\rho^\mathsf x(\,\cdot\,/\varepsilon).
\end{equation*}
This is a smooth positive measure on $M$. The usual convolution estimates and local finiteness imply
\[
    \mathfrak m_\varepsilon\rightarrow\mathfrak m
    \quad\text{in }C^1_{\mathrm{loc}}.
\]
More precisely, if $\sigma_{\varepsilon,a}$ denotes the coordinate density of the full measure $\mathfrak m_\varepsilon$ in the chart $\psi_a$, then
\[
    \sigma_{\varepsilon,a}\rightarrow\sigma_a
    \quad\text{in }C^1_{\mathrm{loc}},
    \qquad
    \|\sigma_{\varepsilon,a}-\sigma_a\|_{\infty,B}
    =O(\varepsilon)
\]
on every compact $B\subseteq\psi_a(U_a)$.

Note that in the following approximation we need only the $C^{(1,2)}$ regularity, while the unweighted case required $C^{(1,3)}$ (recall Propositions \ref{Prop: ApproxtimelikeRiccibound}, \ref{Prop: Null approximate Ricci} and Corollary \ref{cor:GeneralRicciBound}).

\begin{Proposition}[Approximation of weighted energy conditions]
\label{Prop: approx weighted energy conditions}
Let $(M,L,\mathfrak m)$ be a $C^{(1,2)}$ Finsler spacetime $(M,L)$ equipped with a $C^1$ measure $\mathfrak m$, and let $N \in (-\infty,1] \cup (n,\infty]$ in \textit{(i)} and $N \in (-\infty,2] \cup (n,\infty]$ with $n \ge 3$ in \textit{(ii)}.
\begin{enumerate}
    \item\label{appRN1}
    If $(M,L,\mathfrak m)$ satisfies $\Ric_{\mathfrak m}^N \ge K$ timelike distributionally for $K \in \R$, then for every compact set $\mathcal K \subseteq \Omega^T$ and $\delta > 0$, there exists $\varepsilon_0 = \varepsilon_0(\mathcal K,\delta) > 0$ such that for all $\varepsilon \in (0,\varepsilon_0)$ and $v \in \mathcal K$ with $2\widecheck L_\varepsilon(v) = -1$, we have
    \begin{equation}
    \label{eq:WeightedGeneralRicciApprox}
        \Ric^N_{\mathfrak m_\varepsilon}[\widecheck L_\varepsilon](v) \geq K - \delta.
    \end{equation}
    \item\label{appRN2}
    If $(M,L,\mathfrak m)$ satisfies the distributional null $N$-energy condition, then for every compact set $\mathcal K \subseteq TM \setminus \mathbf 0$ and $\delta > 0$, there exists $\varepsilon_0 = \varepsilon_0(\mathcal K,\delta) > 0$ such that for all $\varepsilon \in (0,\varepsilon_0)$ and $v \in \mathcal K \cap \Omega^C[\widecheck{L}_\eps]$ with $\widecheck L_\varepsilon(v) = 0$, we have
    \begin{equation*}
        \Ric^N_{\mathfrak m_\varepsilon}[\widecheck L_\varepsilon](v) \geq -\delta.
    \end{equation*}
\end{enumerate}
\end{Proposition}

\begin{proof}
For simplicity of the argument, we assume $M = \mathbb R^n$ and $\m_\varepsilon = \m \star \rho^\mathsf x_\varepsilon$. The general case follows easily from the arguments we give here.
\begin{enumerate}
\item As in the proof of Proposition \ref{Prop: ApproxtimelikeRiccibound}, we show that
\begin{equation}
\label{eq:ProofWeightedRicci}
    \| \Ric_{\m_\eps}^N[\widecheck L_\eps] - \mathcal H_\varepsilon(\Ric_{\m}^N[L]) \|_{\infty, \mathcal{K}} \rightarrow 0
\end{equation}
on any compact $\mathcal K\subseteq TM\setminus\mathbf 0$; cf.\ also the proof of Corollary \ref{cor:GeneralRicciBound}. By $2$-homogeneity, it suffices to restrict our attention to $v \in \mathcal K$ having unit $\tilde g$-norm. 

To this end, we first compare $\Ric_{\m_\eps}^N[L_\eps] - (\Ric_{\m}^N[L])\star \rho_\eps$ term by term using the formula 
\begin{equation*}
    \Ric_{\m}^N[L] =\frac{\partial G^i}{\partial x^i} - N^i_kN^k_i- v^iv^j\frac{\partial^2 u}{\partial x^i\partial x^j} + G^i\frac{\partial u}{\partial x^i} -\frac{1}{N-n}\left(N^i_i-v^j\frac{\partial u}{\partial x^j}\right)^2,
\end{equation*}
where $u=\log \sigma$, established in Remark \ref{rmk:WeightedRicciSimplified}.

The first term can be handled exactly as in Proposition \ref{Prop: ApproxtimelikeRiccibound}, using $G^i = g^{il} A_l$ and the fact that $g^{il} \in C^{(1,0)}$ and $A_l \in C^{(0,1)}$. 
Moreover, $C^{(1,2)}_\loc$-convergence of $L_\eps$ to $L$ gives
\begin{equation*}
    G^i_\eps \rightarrow G^i, \qquad (N_\eps)^i_j \rightarrow N^i_j
\end{equation*}
locally uniformly.
Thus, using $u_\eps := \log \sigma_\eps \rightarrow u = \log \sigma$ in $C^1_\loc$, we see that the second, fourth, and fifth terms in the difference converge to $0$ locally uniformly.

Finally, setting
\begin{equation*}
    a_\eps:= \sigma_\eps^{-1}, \qquad a:= \sigma^{-1}, \qquad f_j:= \frac{\partial \sigma}{\partial x^j},
\end{equation*}
we have 
\begin{equation*}
    a_\eps \rightarrow a \quad \textnormal{in}\ C^{1}_\loc, \qquad \|a_\eps - a\|_{\infty, \mathcal K} = O(\eps)
\end{equation*}
and we can write
\begin{equation*}
    \frac{\partial u_\eps}{\partial x^j} = a_\eps (f_j \star \rho_\eps^{\mathsf x}), \qquad \frac{\partial u}{\partial x^j} = a f_j.
\end{equation*}
Hence, applying the Friedrichs estimate \eqref{eq:Friedrichs}, we get
\begin{equation*}
    \frac{\partial^2 u_\eps}{\partial x^i \partial x^j} - \frac{\partial}{\partial x^i}\left( \frac{\partial u}{\partial x^j} \star \rho_\eps^{\mathsf x}\right) = \frac{\partial}{\partial x^i}\bigl(a_\eps (f_j \star \rho_\eps^{\mathsf x})\bigr) - \frac{\partial}{\partial x^i}\bigl((af_j) \star \rho_\eps^{\mathsf x}\bigr) \rightarrow 0
\end{equation*}
locally uniformly.
Now, convolving with respect to $\rho_\eps^{\mathsf v}$ and reintroducing the coefficients $v^iv^j$ locally produces an error term of the form
\begin{equation*}
    \left(v^iv^j-\int(v^i-w^i)(v^j-w^j) \rho_\eps^{\mathsf v}(w) \diff w \right)\partial_{x^i}(f_j \star \rho_\eps^{\mathsf x}).
\end{equation*}
As $|w|$ is $O(\eps)$ on the support of $\rho_\eps^{\mathsf v}$ and $f_j$ is continuous, we can appeal to \eqref{eq:ConvergenceC0derivative} to conclude $v^iv^j \partial_{x^i}\partial_{x^j} u_\eps - (v^iv^j \partial_{x^i}\partial_{x^j} u)\star \rho_\eps \rightarrow 0$ locally uniformly, proving that $\Ric_{\m_\eps}^N[L_\eps] - (\Ric_{\m}^N[L])\star \rho_\eps$ converges to $0$.

To finally get \eqref{eq:ProofWeightedRicci}, we use \eqref{eq:WeightedRicciSimplified} to see $\Ric_{\m_\eps}^N[\widecheck{L}_\eps] -\Ric_{\m_\eps}^N[L_\eps] \to 0$.
The Hessian terms cancel (since $u_\eps$ is common), while $\partial_{x^i}(\widecheck G_\varepsilon^i-G_\varepsilon^i)$
tends to $0$ locally uniformly by \eqref{eq:GNN}.
The difference of the terms $N^i_kN^k_i$ tends to $0$ locally uniformly by \eqref{eq:EstimatesCheckSpray} and the local uniform boundedness of the connection coefficients.
It therefore suffices to argue that both
\begin{equation*}
    (\widecheck G_\eps^i - G_\eps^i) \frac{\partial u_\eps}{\partial x^i}  \quad \textnormal{and} \quad \left((\widecheck N_\eps)^i_i - v^j\frac{\partial u_\eps}{\partial x^j}\right)^2- \left((N_\eps)^i_i -v^j\frac{\partial u_\eps}{\partial x^j}\right)^2
\end{equation*}
converge to $0$ locally uniformly, where the last term is omitted in the case $N = \infty$. 
This, in turn, follows easily from the local uniform boundedness of $\partial_{x^i} u_\eps$ and \eqref{eq:EstimatesCheckSpray}. 
From \eqref{eq:ProofWeightedRicci}, we simply repeat the argument given in the proof of Corollary \ref{cor:GeneralRicciBound} to get \eqref{eq:WeightedGeneralRicciApprox}.

\item By \eqref{eq:ProofWeightedRicci}, this is entirely analogous to the proof of Proposition \ref{Prop: Null approximate Ricci}. \qedhere
\end{enumerate}
\end{proof}

\subsection{Global hyperbolicity}

In this subsection, we introduce and discuss the concept of global hyperbolicity for low regularity Finsler spacetimes, following the Lorentzian work of S\"amann \cite{Saemann} (see also \cite{FathiSiconolfi} concerning continuous cone structures), to which we refer for more details, as many of the proofs are analogous to the Lorentzian case presented there. 
Our main goal is to prove the existence of $L$-maximizers that can be obtained as $C^2$-limits of $L_\varepsilon$-maximizers, where $L_\varepsilon$ is any suitable smooth Lorentz--Finsler approximation of $L$.

Let $(M,L)$ be a $C^{(0,2)}$ Finsler spacetime. 
We fix a complete Riemannian metric $\tilde g$ on $M$ and denote its distance function by $d_{\tilde g}$. 

\begin{Definition}[Non-total imprisonment]
    We say that $(M,L)$ is \emph{non-totally imprisoning} if no inextendible future causal curve is contained in a compact set.
\end{Definition}

The following is a consequence of a simple limit curve argument.

\begin{Lemma}[Boundedness of $\tilde g$-lengths]
    If $(M,L)$ is non-totally imprisoning, then for every compact set $K \subseteq M$, there is a constant $C > 0$ that bounds the $\tilde g$-lengths of all future causal curves contained in $K$.
\end{Lemma}

\begin{Definition}[Global hyperbolicity]
We say that $(M,L)$ is \emph{globally hyperbolic} if it is non-totally imprisoning and the causal diamond $J(x,y) := J^+(x) \cap J^-(y)$ is compact for every $x,y \in M$.
\end{Definition}

For the following two results only, we need a stronger notion of order between two $C^{(0,2)}$ Lorentz--Finsler structures $L_1,L_2$:
We write $L_1 \triangleleft L_2$ (resp.\ $L_1 \trianglelefteq L_2$) if, for every $v \in TM \setminus \mathbf 0$, the difference of the fundamental tensors $g^2_v - g^1_v$ is negative definite (resp.\ negative semidefinite).
From $g_v(v,v)=2L(v)$, we infer that $L_1 \triangleleft L_2$ implies $L_1(v)>L_2(v)$ for $v \neq 0$, thereby $L_1 \prec L_2$ (recall Definition \ref{df:order}).

\begin{Proposition}[Existence of smooth wider non-totally imprisoning structures]
    Let $(M,L)$ be a globally hyperbolic $C^{(0,2)}$ Finsler spacetime. 
    Then there exists a smooth time-oriented Lorentz--Finsler structure $L'$ on $M$ such that $L \triangleleft L'$ and $(M,L')$ is non-totally imprisoning.
\end{Proposition}

\begin{proof}
    We can follow the lines of the proof of \cite[Thm.\ 4.3]{Saemann}. We sketch the most important points. 
    Let $\pi:TM \to M$ be the canonical projection.
    
    First, we show that for every relatively compact open set $B \subseteq M$, there exists a $C^{(0,2)}$ Lorentz--Finsler structure $\widetilde{L}$ such that $\widetilde{L} \triangleright L$ on $B$, $\widetilde{L} = L$ on $M \setminus B$, and that $(M,\widetilde{L})$ is non-totally imprisoning.
    To this end, fix an arbitrary smooth function $\beta_B$ on $M$ such that $\beta_B >0$ on $B$ and $\beta_B=0$ on $M \setminus B$, and set for $s>0$ and $v \in TM$
    \[
    L_s(v) :=L(v) -\frac{s}{2}\beta_B\bigl( \pi(v) \bigr) \tilde{g}(v,v).
    \]
    Since $L$ is $C^{(0,2)}$ and $B$ is relatively compact, for some $s'>0$ and any $s \in (0,s')$, $L_s$ is a Lorentz--Finsler structure time-oriented by the same vector field as $L$.
    By construction, clearly $L \triangleleft L_s$ on $B$ and $L=L_s$ on $M \setminus B$.
    Now, we can show that there exists $s'' \in (0,s')$ such that $L_s$ is non-totally imprisoning for every $s \in (0,s'')$.
    This is argued by contradiction. 
    Suppose that there is a sequence $(s_m)_{m \in \N}$ such that $s_m \to 0$ and $(M,L_{s_m})$ is totally imprisoning. 
    As in \cite[Thm.\ 4.3]{Saemann}, this leads to a contradiction to the global hyperbolicity of $(M,L)$. 
    Thus, choosing $\widetilde{L}=L_s$ for small $s>0$ shows the claim.
    Moreover, $\widetilde{L}$ itself is in fact globally hyperbolic again as in the proof of \cite[Thm.\ 4.3]{Saemann}. 

    Using this auxiliary result, one may proceed as follows. 
    Fix $x \in M$ and let $B_m:=B^{d_{\tilde g}}_m(x)$ be the open ball of center $x$ and radius $m$ with respect to $\tilde{g}$. 
    Then there is a globally hyperbolic Lorentz--Finsler structure $L_2$ such that $L_2 \trianglerighteq L$, $L_2 \triangleright L$ on $\overline{B}_2$, and $L_2 = L$ on $M \setminus \overline{B}_3$. 
    Similarly, we find a globally hyperbolic Lorentz--Finsler structure $L_3$ such that $L_3 \trianglerighteq L_2$, $L_3 \triangleright L_2$ on $\overline{B}_3 \setminus B_2$, and $L_3 = L_2$ on $M \setminus (\overline{B}_4 \setminus B_1)$. 
    Continuing this construction inductively, one obtains a sequence $(L_m)_{m \ge 3}$ such that $L_m \trianglerighteq L_{m-1}$, $L_m \triangleright L_{m-1}$ on $\overline{B}_m \setminus B_{m-1}$, and $L_m = L_{m-1}$ on $M \setminus (\overline{B}_{m+1} \setminus B_{m-2})$. 
    By construction, all $L_m$'s with $m \ge k+1$ ($\ge 2$) agree on $B_k$. 
    This allows one to define a global structure $L_\infty$ by setting it to be $L_m$ on $B_{m-1}$. 
    Then we have $L_\infty \trianglerighteq L_m \trianglerighteq L$ for all $m$. 
    Moreover, $L_\infty \triangleright L$ by construction, and $L_\infty$ is non-totally imprisoning since so is $L_m$ and $L_\infty=L_m$ on $B_{m-1}$.  
    
    Finally, let $L'':=(L+L_\infty)/2$ and note that $L \triangleleft L'' \triangleleft L_\infty$.
    Using the regularization and base-dependent parameter construction of Proposition \ref{prop:PropertiesApproximation}\ref{PropApproximation2} with the parameter chosen to preserve these strict inequalities between the vertical Hessians, we obtain a smooth Lorentz–Finsler structure $L'$ satisfying $L \triangleleft L' \triangleleft L_\infty$. Since $L_\infty$ is non-totally imprisoning, so is $L'$.
\end{proof}

\begin{Theorem}[Existence of smooth wider globally hyperbolic structures]\label{th:gh-stab}
    Let $(M,L)$ be a globally hyperbolic $C^{(0,2)}$ Finsler spacetime. 
    Then there exists a smooth Lorentz--Finsler structure $L'$, time-oriented by the same vector field as $L$, such that $L \prec L'$ and $(M,L')$ is globally hyperbolic.
\end{Theorem}

\begin{proof}
    Using the non-total imprisonment result above, this is entirely analogous to \cite[Thm.\ 4.5]{Saemann}.
    We remark that $\prec$ in the proof of \cite[Thm.\ 4.5]{Saemann} is necessarily replaced with $\triangleleft$; then for $L_1 \triangleleft L_2$ and any continuous function $f:M \to [0,1]$, the continuous convex combination $(1-f)L_1 +fL_2$ is again a Lorentz--Finsler metric, time-orientable by the same vector field as $L_1$ (cf.\ \cite[Lem.\ 4.1]{Saemann}).
\end{proof}

We note that, as a consequence, every globally hyperbolic $C^{(0,2)}$ Finsler spacetime has a smooth Cauchy hypersurface, since any Cauchy hypersurface for $L'$ is also one for $L$. 
Here, by a \emph{Cauchy hypersurface} we mean a subset that is met exactly once by every inextendible future causal curve.

\begin{Theorem}[$L$-maximizers as limits of $\widecheck L_\varepsilon$-maximizers]
Let $(M,L)$ be a globally hyperbolic $C^{(1,2)}$ Finsler spacetime. 
Then for any $x < y$, there exists a maximizing $L$-geodesic from $x$ to $y$. 
Moreover, such a geodesic $\eta$ can be obtained as the $C^2$-limit of a sequence of $\widecheck L_{\varepsilon_m}$-geodesics $\eta_m$ as $\eps_m \to 0$.
\end{Theorem}

\begin{proof}
    We first consider the case of $x \ll y$. 
    Fix a smooth globally hyperbolic Lorentz--Finsler structure $L'$ such that $L \prec L'$ given by Theorem \ref{th:gh-stab}.
    We can find a sequence $L_m(v):=\widehat L_{\mu(\eps_m, \pi(v))}$ satisfying $L \prec L_m\prec L'$, where $\mu$ is a smooth parameter chosen to globalize the lightcone relation using \cite[Lem.\ 4.3]{HoermannKunzingerSteinbauer-WaveEq} (compare with the proofs of Propositions \ref{prop:PropertiesApproximation} and \ref{prop:ChruscielGrantApprox}).
    Since $L_m$ is also globally hyperbolic and $x \ll y$ with respect to $L_m$, we find a timelike maximizing $L_m$-geodesic $\eta_m:[0,1] \to M$ from $x$ to $y$. 
    Note that there is a compact set $\mathcal K \subseteq TM \setminus \mathbf{0}$ such that $\dot \eta_m(t) \in \mathcal{K}$ for all $t \in [0,1]$ and $m \in \N$; otherwise one would find an inextendible $L$-causal curve (as a subsequential limit of $\eta_m$) in the compact causal diamond $J_{L'}(x,y)$ with respect to $L'$. 
    Therefore, up to going to a subsequence, $\eta_m$ converges to a causal $L$-geodesic $\eta:[0,1] \to M$ from $x$ to $y$ in $C^2$ (since $\ddot{\eta}^i_m =-\widehat{G}^i_m(\dot{\eta}_m) \to -G^i(\dot{\eta})=\ddot{\eta}^i$). 

    It remains to show that $\eta$ is $L$-maximizing. 
    To this end, let $c:[0,1] \to M$ be any Lipschitz $L$-causal curve from $x$ to $y$ parametrized proportionally to $\tilde g$-arclength; then it is contained in the compact diamond $K:=J_{L'}(x,y)$. 
    Set $C:=\sup_{t \in [0,1]} |\dot{c}(t)|^2_{\tilde g}$ (which is finite; cf.\ \cite[Lemma 2.7]{Saemann}),
    \begin{equation*}
        \delta_m:=\sup_{v \in TK \setminus \mathbf 0} \frac{|L(v) -  L_m(v)|}{|v|^2_{\tilde g}},
    \end{equation*}
    and note that $\delta_m \to 0$. 
    It follows that
    \begin{align*}
        \mathbf L(c) &= \int_0^1 \sqrt{-2L(\dot c)} \diff t
        \leq \int_0^1 \sqrt{-2 L_m(\dot c) + 2 C \delta_m} \diff t
        \leq \mathbf L_m(c) +      \sqrt{2C\delta_m} \\ 
        &\leq d_m(x,y) + \sqrt{2C\delta_m},
    \end{align*}
    where $\mathbf L_m$ and $d_m$ are the Lorentz--Finsler arclength and distance function for $L_m$, respectively. 
    Letting $m \to \infty$ and using the arbitrariness of $c$, we get $d(x,y) \leq \liminf_{m \to \infty} d_m(x,y)$. 
    On the other hand, the $C^2$-convergence $\eta_m \to \eta$ implies $d_m(x,y) = \mathbf L_m(\eta_m) \to \mathbf L(\eta) \leq d(x,y)$, so in fact $\lim_{m \to \infty} d_m(x,y) = \mathbf{L}(\eta) =d(x,y)$, thus $\eta$ is $L$-maximizing and timelike.

    Now, if we consider $\widecheck L_\eps$, then $x \ll y$ for $L$ implies $x \ll y$ for $\widecheck L_\eps$ provided $\eps$ is sufficiently small. 
    Thus, one may proceed as above, since every $\widecheck L_\eps$ is globally hyperbolic (by $\widecheck{L}_\eps \prec L$ and the global hyperbolicity of $L$), and find a sequence of $\widecheck{L}_{\eps_m}$-geodesics converging to a maximizing $L$-geodesic. 

    The case of $x < y$ is argued analogously. 
    We refer to \cite[Prop.\ 2.13]{Graf-SingularityC1} for more details. 
\end{proof}

Simple analogs of the above result (corresponding to the situations in Hawking's and Penrose's theorems, respectively) are the following.

\begin{Corollary}[$L$-maximizers as limits of $\widecheck L_\varepsilon$-maximizers for Cauchy hypersurfaces]
\label{Corollary: hypersurfacemaxaslimit}
    Let $(M,L)$ be a globally hyperbolic $C^{(1,2)}$ Finsler spacetime. 
    Suppose that $\Sigma \subseteq M$ is a $C^2$ spacelike Cauchy hypersurface. 
    Then for each $y \in I^+(\Sigma)$, there exists a timelike maximizing geodesic from $\Sigma$ to $y$, that can be obtained as the $C^2$-limit of $\widecheck{L}_{\eps_m}$-geodesics from $\Sigma$ to $y$ as $\eps_m \to 0$.
\end{Corollary}

\begin{Corollary}[$L$-maximizers as limits of $\widecheck L_\varepsilon$-maxmizers for codimension-$2$ surfaces]
\label{Corollary: Trappedsurfacemaxaslimit}
    Let $(M,L)$ be a globally hyperbolic $C^{(1,2)}$ Finsler spacetime. 
    Suppose that $S \subseteq M$ is a $C^2$ compact spacelike submanifold of codimension $2$. 
    Then for any $y \in J^+(S) \setminus (I^+(S) \cup S)$, there exists a maximizing null geodesic from $S$ to $y$, which can be obtained as the $C^2$-limit of $\widecheck{L}_{\eps_m}$-geodesics from $S$ to $y_m$ as $\eps_m \to 0$, where $y_m \to y$.
\end{Corollary}

Here we say that $\Sigma$ (or $S$) is \emph{spacelike} if $T_x \Sigma \cap \overline{\Omega_x^T} =\{0\}$ for every $x \in \Sigma$. 
For a future null geodesic, being maximizing means that no two points in it are connected by a timelike curve (i.e., its image is achronal).

\section{Singularity theorems}
\label{section: singularity theorems}

With the approximation and convergence results from the previous section at our disposal, we can now prove the Hawking and Penrose singularity theorems.
We begin with the unweighted case in regularity $C^{(1,3)}$ and then proceed to the weighted case, where our results hold for $L \in C^{(1,2)}$ and $\mathfrak m \in C^1$.
Although we focus on future incompleteness, the standard reversing argument gives past incompleteness for all the corresponding singularity theorems (see Corollary \ref{cr:past-sing}, Remark \ref{remark: past versions of singularity theorems}).

\subsection{The Hawking singularity theorems}\label{ssc:Hawking}

To begin with, we suppose that the Lorentz--Finsler structure $L$ is smooth and introduce some necessary notions (see \cite{Minguzzi-LightconesFinsler,MinguzziRaychaudhuri} for more details).
Let $\Sigma \subseteq M$ be a $C^2$ spacelike hypersurface (that is, $T_x \Sigma \cap \overline{\Omega_x^T} =\{0\}$ for all $x \in \Sigma$). 
Then there exists a $C^1$ future timelike vector field $V$ on a neighborhood of $\Sigma$ satisfying
\[
D^V_V V \equiv 0, \qquad 
L(V) \equiv -\frac{1}{2}, \qquad 
\ker g_V(V,\cdot)|_{T_xM} =T_x\Sigma \,\ \text{for all}\ x \in \Sigma.
\]
By the latter two conditions, $V|_\Sigma$ can be regarded as the future unit normal vector field on $\Sigma$, and we extend it with the geodesic condition $D^V_V V \equiv 0$ (i.e., integral curves of $V$ are geodesics).
We also remark that, for every $w \in T_x\Sigma$,
\[
g_V(V,w) =\frac{\partial L}{\partial v^i}(V)w^i.
\]
Thus, $V(x)+T_x\Sigma$ is tangent to $L^{-1}(-1/2)$ in $T_xM$; this property determines $V(x)$.

Given $x \in \Sigma$, let $\eta:I \to M$ be the future timelike geodesic given as the integral curve of $V$ with $\eta(0)=x$.
For $t \in I$, we set $N_\eta(t):=\ker g_{\dot{\eta}}(\dot{\eta}(t),\cdot)$, which is an $(n-1)$-dimensional subspace of $T_{\eta(t)}M$.
Then we consider the endomorphism $B_t:N_\eta(t) \to N_\eta(t)$ defined by $B_t(w):=D_w^{\dot{\eta}} V$.
It follows from $g_V(V,V) \equiv -1$ and $V(\eta(t))=\dot{\eta}(t)$ that, for all $w \in N_\eta(t)$,
\[
w[g_V(V,V)] =2g_{\dot{\eta}}\bigl( B_t(w),\dot{\eta}(t) \bigr) =0,
\]
which indeed ensures $B_t(w)\in N_{\eta}(t)$.
We define the \emph{expansion} and \emph{shear} (of $V$) at $\eta(t)$ by
\begin{equation}\label{eq:expan}
\theta(t) :=\mathrm{trace}(B_t), \qquad 
\sigma(t) :=\mathrm{trace} \Biggl( \biggl( B_t -\frac{\theta(t)}{n-1}I_{N_{\eta}(t)} \biggr)^2 \Biggr),
\end{equation}
where $I_{N_{\eta}(t)}$ is the identity of $N_{\eta}(t)$. If $L$ is merely $C^{(1,3)}$, then the expansion is directly defined as a function on $\Sigma$ via $\theta:=\mathrm{trace}(w \mapsto D^V_wV)$, $w \in T\Sigma$.

With these notations, we introduce the following property, which will play a key role.

\begin{Definition}[$\beta$-contracting]\label{df:beta}
We say that $\Sigma$ is \emph{$\beta$-contracting} for $\beta>0$ if $\theta(0)<-\beta$ for every $x \in \Sigma$. 
In the case of $\beta=0$ (i.e., $\theta(0)<0$), we simply say that $\Sigma$ is \emph{contracting}.
\end{Definition}

When $L$ is smooth, we have the Finsler analog to the fundamental \emph{Raychaudhuri equation} (see \cite[Eq.\ (25)]{MinguzziRaychaudhuri} where $\dim M=n+1$): 
\begin{equation}\label{eq:Ray}
\theta'(t) =-\Ric\bigl( \dot{\eta}(t) \bigr) -\sigma(t) -\frac{\theta(t)^2}{n-1}.
\end{equation}

\begin{Proposition}[Timelike quantitative focusing]
\label{Prop: Hawking quantitative focusing}
Let $(M,L)$ be a smooth Finsler spacetime, $\Sigma \subseteq M$ be a $C^2$ spacelike hypersurface and $\eta : [0, b] \to M$ with
$x := \eta(0) \in \Sigma$  be a future unit speed timelike geodesic that maximizes the
distance from $\Sigma$, i.e., $\mathbf{L}(\eta\vert_{[0,t]}) = \sup_{y\in \Sigma} d(y,\eta(t))$ for all $t \in [0, b]$. 
Then $\dot{\eta}(0) \perp T_x\Sigma$ in the sense that $T_x\Sigma=\ker g_{\dot \eta(0)}(\dot \eta(0), \cdot)$. 
Furthermore, if $\theta(0)<-\beta<0$
and $\Ric(\dot{\eta}) \ge -\delta$ with $0 < \delta \le \frac{3\beta}{b} (1-r)$ for some
$0 < r < 1$, then $b \le\frac{n-1}{r\beta}$ necessarily holds.
\end{Proposition}

The proof given below does not use the notion of mean curvature vector field, nor a second variation formula (compare with \cite[Lem.\ 4.10]{Graf-SingularityC1}), as these notions are less investigated in the Finslerian framework (see, e.g., \cite{Shlec}). 
It only uses the Raychaudhuri equation.

\begin{proof}
First, we prove the orthogonality. 
Suppose not and let $z=\eta(\eps)$, $\eps>0$, with $\eps$ so small that $z$ is included in a convex open neighborhood $U$ of $x=\eta(0)$.
In $U$ we consider the locus $d(\cdot,z)=\eps$ including $x$.
This is a local regular hypersurface tangent to $\ker g_{\dot{\eta}}(\dot{\eta}(0),\cdot)$; see \cite[Eq.\ (13)]{MinguzziConvexneighborhoods}.
If $T_x\Sigma=\ker g_V(V(x),\cdot)$ does not coincide with $\ker g_{\dot{\eta}}(\dot{\eta}(0),\cdot)$, then $\Sigma$ enters the region $d(\cdot,z)>\eps$, thus it is possible to connect $\Sigma$ to $z$ with a segment of longer length violating the maximality of $\eta$.
Hence, $\dot{\eta}(0) \perp T_x\Sigma$ holds.
In particular, $\dot{\eta}(0)=V(x)$.

Next we prove the inequality $b \le\frac{n-1}{r\beta}$. 
Put $c:=1/(n-1)$ for brevity and later convenience. 
From the Raychaudhuri equation \eqref{eq:Ray} and the hypothesis $\Ric(\dot{\eta}) \ge -\delta$, we get the inequality 
\begin{equation}\label{eq:Ray-ineq}
\theta'(t) \leq \delta - c\theta(t)^2.
\end{equation}
The solution $\phi(t)$ to the equality case with initial condition $\phi(0)=-\beta$, $\theta(0)<\phi(0)$, is obtained by integrating and using partial fraction decomposition:
\[
\frac{1}{2c\sqrt{\delta/c}}\log \left|\frac{\sqrt{\delta/c}+\phi(t)}{\sqrt{\delta/c}-\phi(t)}\right|=t+C .
\]
Solving  the logarithmic form for \(\phi\) gives the explicit solution
\[
\phi (t)=\sqrt{\delta/c} \cdot \frac{Ae^{2t\sqrt{c\delta}}-1}{Ae^{2t\sqrt{c\delta}}+1},
\]
where $A$ is a constant determined by the initial conditions. 
Note that 
\[
0 > -\beta\sqrt{c/\delta} =\phi(0)\sqrt{c/\delta} =\frac{A-1}{A+1}, 
\]
thus $\vert A\vert<1$ (it will be useful to invert it as $A=\frac{1-\beta\sqrt{c/\delta}}{1+\beta\sqrt{c/\delta}}$). 
Note that we have blow up if $-1<A<0$ which occurs if and only if $ \beta\sqrt{c/\delta}>1$ (blow up is at time $\hat t=- \frac{1}{2\sqrt{c\delta}} \log \vert A\vert$). Since the numerator of the expression for $\phi(t)$ is negative, the blow up is to $-\infty$.
This behavior can be qualitatively understood from the inequality $\theta' \leq c((\sqrt{\delta/c})^2 - \theta^2)$ above, because $\theta<-\sqrt{\delta/c}$ makes $\theta$ keep getting smaller.

Suppose the claim is false, so $b>\frac{1}{cr\beta}$, then, using the hypothesis $\delta \leq \frac{3\beta}{b}(1-r)$ and $r\in (0,1)$,
\[
\delta \le \frac{3\beta}{b} (1-r) < 3c\beta^2 r(1-r) < c\beta^2 \,\Rightarrow\, \sqrt{\delta/c} < \beta ,
\]
so we are in the blow-up case.
A comparison argument with the solution \(\phi(t)\) shows  that \(\theta(t)\leq \phi(t)\), and so \(\theta\) blows up to \(-\infty\) as well at a finite time $T>0$ satisfying
\begin{align*}
T &\le \hat t 
=\frac{1}{2\sqrt{c\delta}} \log \left\vert \frac{1+\beta\sqrt{c/\delta}}{1-\beta\sqrt{c/\delta}}\right\vert 
< \frac{b r}{2} \frac{\beta}{\sqrt{\delta/c}} \log \frac{1+\sqrt{\delta/c}/\beta}{1-\sqrt{\delta/c}/\beta}\\
&=\frac{br}{2} f\left(\frac{\sqrt{\delta/c}}{\beta}\right),
\end{align*}
where $f(s):=\frac{1}{s} \log \frac{1+s}{1-s}$ for $s \in (0,1)$ and we used $b>\frac{1}{cr\beta}$ again.

Set $s = \sqrt{\delta/c}/\beta$. 
Then the hypotheses \(\delta \leq \frac{3\beta}{b}(1-r)\) and \(b > \frac{1}{cr\beta}\) imply 
\[
s \le \sqrt{\frac{3(1-r)}{c\beta b}} \le \sqrt{3r(1-r)}
\]
and since $f$ is increasing on $[0,1)$ (see Remark \ref{Remark: monotonicities}\ref{mono1} below)
\[
T<b \biggl[\frac{r}{2} f\bigl(\sqrt{3r(1-r)}\bigr)\biggr].
\]
We claim that the term in square brackets provides an increasing function in $r$ and is bounded by 1 over $(0,1]$, and argue this in Remark \ref{Remark: monotonicities}\ref{mono2} below. 
We conclude that $T<b$ which implies that a focal point exists before $\eta(b)$, contradicting the maximality of $\eta$ in $[0,b]$ (see \cite[Prop.\ 5.1]{MinguzziRaychaudhuri}).
Hence, $b\le \frac{1}{cr\beta}= \frac{n-1}{r\beta}$.
\end{proof}

The proof of the previous result can be complemented by the following observation.

\begin{Remark}
\label{Remark: monotonicities}
Let us address some of the claims concerning monotonicity that appear in the proof of Proposition \ref{Prop: Hawking quantitative focusing} in more detail.
\begin{enumerate}
\item\label{mono1}
We want to show that \( f'(s) > 0 \) for \( s \in (0,1) \), where
\begin{equation}\label{eq:f'}
f'(s) = \frac{2s - (1-s^2) \log \left( \frac{1+s}{1-s} \right)}{s^2(1-s^2)}.
\end{equation}
The denominator \( s^2(1-s^2) > 0 \) on \( (0,1) \), so the sign of \( f'(s) \) is determined by the numerator:
\[
N(s) = 2s - (1-s^2) \log \left( \frac{1+s}{1-s} \right).
\]
Using the series expansion for \( \log \left( \frac{1+s}{1-s} \right) = 2 \sum_{n=0}^\infty \frac{s^{2n+1}}{2n+1} \), we compute:
\[
\begin{aligned}
N(s) &= 2s - (1-s^2) \cdot 2 \sum_{n=0}^\infty \frac{s^{2n+1}}{2n+1} \\
&= 2s - 2 \left( s - 2 \sum_{n=1}^\infty \frac{s^{2n+1}}{(2n+1)(2n-1)} \right) \\
&= 4 \sum_{n=1}^\infty \frac{s^{2n+1}}{(2n+1)(2n-1)}.
\end{aligned}
\]
For \( s \in (0,1) \), every term in the series is positive. Thus, \( N(s) > 0 \) for \( 0 < s < 1 \), and \( N(0) = 0 \). Therefore, \( f'(s) > 0 \) on \( (0,1) \), meaning that \( f(s) \) is  increasing.
\item\label{mono2}
We want to argue that $r \mapsto (r/2)f(\sqrt{3r(1-r)})$ is monotonically increasing and bounded above on $(0,1]$.
To this end, we set $s = s(r):=\sqrt{3r(1-r)}$ and $g(r):=(r/2)f(s(r))$. 
Direct differentiation gives, with the help of \eqref{eq:f'},
\[
g'(r)
=
\frac{(1-s^2)f(s)-2(2r-1)}
     {4(1-r)(1-s^2)}.
\]
For \(0<r\leq 1/2\), this is positive. 
For
\(1/2<r<1\), since
\(2r-1=\sqrt{1-\frac43s^2}\), it suffices to prove
\[
(1-s^2)f(s)>2\sqrt{1-\frac43s^2},
\qquad 0<s<\frac{\sqrt{3}}{2}.
\]
Define
\[
H(s):=
\frac{s^2f(s)^2}{4}
-\frac{s^2(3-4s^2)}{3(1-s^2)^2}.
\]
Then
\[
H(0)=H'(0)=0,
\qquad
\frac{\diff}{\diff s}\bigl((1-s^2)H'(s)\bigr)
=
\frac{16s^4}{3(1-s^2)^3}>0.
\]
Consequently, \(H(s)>0\), which yields the required inequality.
Thus \(g'(r)>0\) on \((0,1)\). Since
\(\lim_{r\uparrow1}g(r)=1\), we conclude that
\[
\frac{r}{2}f\!\left(\sqrt{3r(1-r)}\right)<1
\qquad\text{for all }0<r<1.
\]
\end{enumerate}
\end{Remark}

We are in a position to prove the first version of the Hawking singularity theorem.
Given a spacelike hypersurface $\Sigma \subseteq M$, we set $d_\Sigma(y):=\sup_{x \in \Sigma}d(x,y)$ for $y \in M$ (note that $d_\Sigma(y)>0$ if and only if $y \in I^+(\Sigma) =\bigcup_{x \in \Sigma} I^+(x)$). 

\begin{Theorem}[$C^{(1,3)}$ Hawking singularity theorem, version 1]
\label{Theorem: unweighted Hawking 1}
Let $(M,L)$ be a globally hyperbolic $C^{(1,3)}$ Finsler spacetime satisfying the distributional strong energy condition with a $C^2$ spacelike Cauchy hypersurface $\Sigma \subseteq M$ that is $\beta$-contracting with $\beta > 0$. 
Then we have $\sup_M d_\Sigma \leq (n-1)/\beta$. 
In particular, $(M,L)$ is future timelike geodesically incomplete.
\end{Theorem}

\begin{proof}
    Suppose by contradiction that $d_\Sigma(y) > (n-1)/\beta$ for some $y \in I^+(\Sigma)$. 
    It follows from Corollary \ref{Corollary: hypersurfacemaxaslimit} that we can find a maximizing $L$-geodesic $\eta:[0,l] \to M$ from $\Sigma$ to $y$ which is the $C^2$ limit of $\widecheck{L}_{\eps_m}$-maximizers $\eta_m:[0,l_m] \to M$ from $\Sigma$ to $y$, where $l:=\mathbf L(\eta)=d_\Sigma(y)$ and we put $l_m := \mathbf L_{\eps_m}(\eta_m)$ for their $\widecheck L_{\varepsilon_m}$-Lorentz--Finsler arclength. 
    The set $\mathcal K:=\bigcup_m \mathrm{im}(\dot{\eta}_m) \cup \mathrm{im}(\dot{\eta})$ is compact in $TM$. 

    Denote by $\theta_m$ the expansion of $\Sigma$ with respect to $\widecheck L_{\eps_m}$ (evaluated at $t=0$ along the geodesic orthogonal to $\Sigma$, to be precise). 
    Then $\theta_m$ converges to the expansion $\theta$ of $L$ locally uniformly on $\Sigma$. 
    Hence, on a neighborhood of $\eta(0)$ in $\Sigma$, we have $\theta_m < -\beta < 0$ for large $m$. 
    Take $0 < r < 1$ such that $d_\Sigma(y) > (n-1)/r\beta$, and set $\delta_m:=\frac{3\beta}{l_m}(1-r)$. 
    By $\delta_m \to \delta :=\frac{3\beta}{l}(1-r)$ and Proposition \ref{Prop: ApproxtimelikeRiccibound}, for sufficiently large $m$ (depending only on $\mathcal K,r$), all $\widecheck L_{\eps_m}$-unit future timelike vectors $v$ in $\mathcal K$ satisfy $\Ric[\widecheck{L}_{\eps_m}](v) \geq -\delta_m$, so that $\Ric[\widecheck{L}_{\eps_m}](\dot{\eta}_m) \geq -\delta_m$.
    Hence we deduce from Proposition \ref{Prop: Hawking quantitative focusing} that $l_m \leq \frac{n-1}{r\beta}$. 
    Letting $m \to \infty$ and then $r \to 1$ shows $l \leq (n-1)/\beta$, contradicting the assumption.
\end{proof}

The $C^{(1,3)}$ Hawking theorem given above admits a natural version for past incompleteness, which we discuss now.
In complete analogy, the remaining unweighted and weighted theorems all admit such versions as well, we describe this briefly in Remark \ref{remark: past versions of singularity theorems} further below.

\begin{Corollary}[Past version of the $C^{(1,3)}$ Hawking singularity theorem, version 1]
\label{cr:past-sing}
Let $(M,L)$ be a globally hyperbolic $C^{(1,3)}$ Finsler spacetime satisfying the distributional strong energy condition.
Suppose that there exists a $C^2$ spacelike Cauchy hypersurface $\Sigma \subseteq M$ satisfying $\theta > \beta > 0$. 
Then 
\begin{equation*}
    \sup_{y \in \Sigma} d(\cdot,y) \leq \frac{n-1}{\beta}.
\end{equation*}
In particular, $(M,L)$ is past timelike geodesically incomplete.
\end{Corollary}

\begin{proof}
    The proof follows by applying Theorem \ref{Theorem: unweighted Hawking 1} to the reverse Lorentz--Finsler structure $\overleftarrow{L}(v):=L(-v)$.
    Precisely, we choose the time orientation such that $v$ is $\overleftarrow{L}$-future timelike if and only if $-v$ is $L$-future timelike, and then the $\overleftarrow{L}$-Ricci curvature $\overleftarrow{\Ric}$ satisfies $\overleftarrow{\Ric} =\iota^*(\Ric)$ in $\mathcal{D}^{'(1)}(TM \setminus \mathbf 0)$, where $\iota:TM \setminus \mathbf 0 \to TM \setminus \mathbf 0$ is the fiber reflection $\iota(v):=-v$ and $\iota^*$ denotes the pullback under $\iota$. Due to this, it is easily concluded that $\overleftarrow{L}$ continues to satisfy the distributional strong energy condition and $\Sigma$ is $\beta$-contracting with respect to $\overleftarrow{L}$, yielding
    \[
    \sup_{x \in \Sigma} \overleftarrow{d}(x,\cdot)
    =\sup_{y \in \Sigma}d(\cdot,y)
    \le \frac{n-1}{\beta},
    \]
    where $\overleftarrow{d}(x,y)$ $(=d(y,x))$ denotes the Lorentz--Finsler distance with respect to $\overleftarrow{L}$.
\end{proof}

The second version deals with the case of $\beta=0$, upon replacing the Cauchy hypersurface with a compact spacelike hypersurface.

\begin{Theorem}[$C^{(1,3)}$ Hawking singularity theorem, version 2]
\label{Theorem: Hawking version 2}
Let $(M,L)$ be a $C^{(1,3)}$ Finsler spacetime satisfying the distributional strong energy condition, $\Sigma \subseteq M$ a $C^2$ compact spacelike hypersurface that is contracting. 
Then $(M,L)$ is future timelike geodesically incomplete.
\end{Theorem}

\begin{proof}
    The proof is similar in spirit to the first version; we give a sketch for completeness. 
    Assume by contradiction that $(M,L)$ is future timelike geodesically complete.
    As in the proof of the first version, we denote by $\theta$ and $\theta_\eps$ the expansions of $\Sigma$ with respect to $L$ and $\widecheck{L}_\eps$, respectively, as functions on $\Sigma$. 
    We remark that $\Sigma$ is also $\widecheck{L}_\eps$-spacelike since $\widecheck{L}_\eps$ has narrower future timelike cones than $L$. 
    By the compactness of $\Sigma$, we find $\beta:=-\max_\Sigma \theta/2 >0$, so that $\theta_\varepsilon < -\beta < 0$ on $\Sigma$ for sufficiently small $\eps>0$. 
    
    Let
    \begin{equation*}
        \mathcal K:= \bigcup_{0 \leq \varepsilon \leq 1} \{ v \in TM \,:\, \pi(v) \in \Sigma,\, v \in \Omega^T[\widecheck{L}_\eps],\, v \perp T\Sigma\text{ for } \widecheck{L}_\eps,\, 2\widecheck{L}_\varepsilon(v) = -1\}.
    \end{equation*}
    This is compact in $TM$. 
    Hence there exists $\varepsilon_0 > 0$ such that the set $\mathcal{F}_{\mathcal K, b, \varepsilon_0} \subseteq TM$ (as in Corollary \ref{cr:Fcpt}) is (well-defined and) compact, where we set $b:=4(n-1)/\beta$, thanks to the assumed geodesic completeness and Proposition \ref{Prop: Existence time for Leps geodesics}. 
    Let $0 < \delta < 3\beta/2b$, then there is $\eps_1 \in (0,\eps_0)$ such that for every $\eps \in (0,\eps_1)$ and $v \in \mathcal{F}_{\mathcal K, b, \varepsilon_0}$ with $2\widecheck{L}_\varepsilon(v) = -1$, we have $\Ric[\widecheck{L}_\eps](v) \geq -\delta$. 
    Therefore, it follows from Proposition \ref{Prop: Hawking quantitative focusing} (with $r=1/2$) that every $\widecheck{L}_\eps$-geodesic starting $\widecheck L_\varepsilon$-orthogonally to $\Sigma$ will stop maximizing the $\widecheck L_\varepsilon$-distance before proper time $2(n-1)/\beta$.

    At this point, we may proceed as in the smooth proof \cite[Thm.\ 8.13]{LuMinguzziOhtaSingularity} (cf.\ \cite[Thm.\ 6.49]{Minguzzicausality}). 
    Up to going to a covering, we may assume that $\Sigma$ is acausal and a partial Cauchy hypersurface for $\widecheck{L}_\eps$ with $\eps \in (0,\eps_1)$.
    The above argument shows that $\Sigma$ is future causally araying with respect to $\widecheck L_\varepsilon$, a contradiction to \cite[Thm.\ 8.11]{LuMinguzziOhtaSingularity} (see \cite[Thm.\ 6.48]{Minguzzicausality} for a proof).
\end{proof}

In order to obtain sharp versions of the Hawking theorems for general timelike Ricci bounds $\Ric \geq K$ for $K \in \R$, we require the following refined version of the estimate given in Proposition \ref{Prop: Hawking quantitative focusing}.

\begin{Proposition}[Timelike quantitative focusing for general bounds]
\label{Prop: quantitative focusing for general bounds}
Let $(M,L)$ be a smooth Finsler spacetime, $\Sigma\subseteq M$ a $C^2$ spacelike hypersurface, and $\eta:[0,b]\to M$ a future unit speed timelike
geodesic which maximizes the distance from $\Sigma$.
Let $K\in\mathbb R$ and $\beta>0$, with $\beta>\sqrt{(n-1)|K|}$ when $K<0$.
If $\theta(0)<-\beta$ and $\Ric(\dot\eta)\geq K-\delta$, where
\[
    0<\delta
    \leq(1-r)\left(K+\frac{\beta^2}{n-1}\right)
    \quad\text{for some }0<r<1,
\]
then we have
\begin{equation*}
    b\leq\frac{T_{K,\beta}}{r},
    \quad \text{where }
    T_{K,\beta}:=
    \begin{cases}
        \dfrac{n-1}{\beta},
        & K=0,\\[1.2ex]
        \sqrt{\dfrac{n-1}{K}}\,
        \operatorname{arctan}\!
        \left(\dfrac{\sqrt{(n-1)K}}{\beta}\right),
        & K>0,\\[1.6ex]
        \sqrt{\dfrac{n-1}{|K|}}\,
        \operatorname{arctanh}\!
        \left(\dfrac{\sqrt{(n-1)|K|}}{\beta}\right),
        & K<0.
    \end{cases}
\end{equation*}
\end{Proposition}

\begin{proof}
Orthogonality follows as in Proposition \ref{Prop: Hawking quantitative focusing}. 
Let $c:=1/(n-1)$.
Since $\eta$ maximizes the distance from $\Sigma$, there is no focal point to $\Sigma$ along $\eta$ in $(0,b)$. 
Thus the expansion is defined on $[0,b)$, where the Raychaudhuri equation \eqref{eq:Ray} and the assumed Ricci curvature bound give
\begin{equation*}
\theta'\leq -K+\delta-c\theta^2.
\end{equation*}
Setting \(z:=-\theta\), we obtain
\begin{equation}
\label{eq: diff ineq for z}
z'\geq K-\delta+cz^2,
\qquad z(0)>\beta.
\end{equation}
Our assumptions on $K$ and $\beta$ ensure that \(K+c\beta^2>0\). 
Moreover, for every $u\geq\beta$, the hypothesis $\delta\leq(1-r)(K+c\beta^2)$ implies
\begin{align*}
K-\delta+cu^2
&\geq K+cu^2-(1-r)(K+c\beta^2)\\
&\geq K+cu^2-(1-r)(K+cu^2)\\
&= r(K+cu^2)
\geq r(K+c\beta^2)>0.
\end{align*}
In particular, $z$ cannot decrease to \(\beta\), thereby $z'(t)>0$ and $z(t)>\beta$ throughout $[0,b)$.

Dividing the differential inequality \eqref{eq: diff ineq for z} by its positive right-hand side and integrating, we find, for every $0<t<b$,
\begin{align*}
t
&\leq \int_0^t
   \frac{z'(s)}{K-\delta+cz(s)^2}\,\diff s 
=\int_{z(0)}^{z(t)}
   \frac{\diff u}{K-\delta+cu^2}
\leq\int_\beta^\infty
   \frac{\diff u}{K-\delta+cu^2}\\
&\leq\frac{1}{r}\int_\beta^\infty
   \frac{\diff u}{K+cu^2}.
\end{align*}
The last inequality follows from
\(K-\delta+cu^2\geq r(K+cu^2)>0\) for \(u\geq\beta\), which was established before. The remaining integral is finite and equals
\begin{equation*}
\int_\beta^\infty\frac{\diff u}{K+cu^2}
=
\begin{cases}
\displaystyle
\frac{1}{c\beta},
& K=0,\\[1.2ex]
\displaystyle
\frac{1}{\sqrt{cK}}
\left(\frac{\pi}{2}
-\arctan\!\left(\beta\sqrt{\frac{c}{K}}\right)\right),
& K>0,\\[1.2ex]
\displaystyle
\frac{1}{2\sqrt{c|K|}}
\log\!\left(
\frac{\beta\sqrt{c}+\sqrt{|K|}}
     {\beta\sqrt{c}-\sqrt{|K|}}
\right),
& K<0.
\end{cases}
\end{equation*}
For \(c=1/(n-1)\), these are precisely the displayed expressions for \(T_{K,\beta}\). 
Letting \(t\uparrow b\) therefore yields $b \leq T_{K,\beta}/r$ as claimed.
\end{proof}

\begin{Remark}[On the discrepancy between Propositions \ref{Prop: Hawking quantitative focusing} and \ref{Prop: quantitative focusing for general bounds} for $K = 0$]
For \(K=0\), the hypothesis on \(\delta\) in Proposition \ref{Prop: quantitative focusing for general bounds} becomes
\begin{equation*}
0<\delta\leq(1-r)\frac{\beta^2}{n-1},
\end{equation*}
whereas Proposition \ref{Prop: Hawking quantitative focusing} assumes
\begin{equation*}
0<\delta\leq(1-r)\frac{3\beta}{b}.
\end{equation*}
These are different sufficient conditions for the same conclusion \(b\leq(n-1)/(r\beta)\), thus Proposition \ref{Prop: Hawking quantitative focusing} does not follow from Proposition \ref{Prop: quantitative focusing for general bounds} with $K = 0$. 
The difference arises because the proof of Proposition \ref{Prop: quantitative focusing for general bounds} uses the comparison
\begin{equation*}
K-\delta+cu^2\geq r(K+cu^2)
\qquad (u\geq\beta)
\end{equation*}
with subsequent integral estimates, whereas Proposition \ref{Prop: Hawking quantitative focusing} uses an estimate involving the interval length $b$. 
We expect that a version of Proposition \ref{Prop: Hawking quantitative focusing} for general bounds $K$ can also be established. 
This discrepancy highlights that both results are independently of interest and it has no bearing on their applications in the Hawking singularity theorems (Theorems \ref{Theorem: unweighted Hawking 1}, \ref{Theorem: Hawking version 2}, \ref{Theorem: unweighted general Hawking 1}): 
The curvature error in the regularization can be made sufficiently small to satisfy either condition in Propositions \ref{Prop: Hawking quantitative focusing} and \ref{Prop: quantitative focusing for general bounds} for $K = 0$.
\end{Remark}

\begin{Theorem}[$C^{(1,3)}$ Hawking singularity theorem,
version 1, general $K$]
\label{Theorem: unweighted general Hawking 1}
Let $(M,L)$ be a globally hyperbolic $C^{(1,3)}$ Finsler spacetime satisfying $\Ric \geq K$ timelike distributionally for some $K\in\mathbb R$.
Suppose that $\Sigma\subseteq M$ is a $C^2$ spacelike Cauchy hypersurface that is $\beta$-contracting, where $\beta>0$ and additionally $\beta>\sqrt{(n-1)|K|}$ if $K<0$. 
Then $\sup_M d_\Sigma\leq T_{K,\beta}$, where $T_{K,\beta}$ is as in Proposition \ref{Prop: quantitative focusing for general bounds}.
In particular, $(M,L)$ is future timelike geodesically incomplete.
\end{Theorem}

\begin{proof}
We follow the proof of Theorem \ref{Theorem: unweighted Hawking 1}.
Suppose that $d_\Sigma(y)>T_{K,\beta}$ for some $y\in I^+(\Sigma)$, and choose $r\in(0,1)$ sufficiently close to $1$ such that $d_\Sigma(y)>T_{K,\beta}/r$.
By Corollary \ref{Corollary: hypersurfacemaxaslimit}, we can find $\widecheck L_{\varepsilon_m}$-unit speed maximizing geodesics from $\Sigma$ to $y$ converging to an $L$-maximizer, with their proper-time lengths $l_m\to d_\Sigma(y)$.
As in the proof of Theorem \ref{Theorem: unweighted Hawking 1}, their tangents lie in a compact subset of $\Omega^T$, and their initial expansions satisfy $\theta_m(0)<-\beta$ for sufficiently large $m$.
Fix
\[
    0<\delta<
    (1-r)\left(K+\frac{\beta^2}{n-1}\right).
\]
By Corollary \ref{cor:GeneralRicciBound}, we get $\Ric[\widecheck L_{\varepsilon_m}](\dot\eta_m)\geq K-\delta$ for large $m$.
Proposition \ref{Prop: quantitative focusing for general bounds} therefore implies $l_m\leq T_{K,\beta}/r$. 
Letting $m \to \infty$ (equivalently, $\varepsilon_m \to 0)$ shows $d_\Sigma(y) \leq T_{K,\beta}/r$, contradicting our choice of $r$.
\end{proof}

\subsection{The Penrose singularity theorem}\label{sc:Penrose}

We first recall the notion of trapped surfaces (see \cite[\S 4]{MinguzziRaychaudhuri}, \cite[\S 8.1]{LuMinguzziOhtaSingularity}). Throughout this subsection, let $n=\dim M \geq 3$.
Let $S$ be a codimension $2$, orientable, $C^2$ spacelike submanifold (without boundary) in a $C^{(1,3)}$ Finsler spacetime $(M,L)$. 
Then, for each $x \in S$, since $T_x S \cap \overline{\Omega_x^T}=\{0\}$, there are exactly two hypersurfaces $H_x^{\pm} \subseteq T_xM$ containing $T_x S$ and tangent to $\partial \Omega_x^T$. 
By the orientability of $S$, we can take $C^1$ vector fields $V^\pm$ on $S$ such that $V^\pm(x) \in (H_x^\pm \cap \partial \Omega_x^T) \setminus \{0\}$ (unique up to a rescaling by a positive function on $S$), and $V^\pm$ satisfy $\ker g_{V^\pm}(V^\pm(x),\cdot)=H_x^\pm$.
We define the associated \emph{null expansions} $\theta^\pm(x)$ as the trace of the endomorphism $w \mapsto D^{V^\pm}_wV^\pm$ on $T_xS$. 

\begin{Definition}[Trapped surfaces]
    We say that $S$ is a (\emph{future}) \emph{trapped surface} if it is compact and $\theta^\pm <0$ on $S$.
\end{Definition}

\begin{Proposition}[Null quantitative focusing]
\label{Prop: Null quantitative focusing}
Let $(M,L)$ be a smooth Finsler spacetime, $S \subseteq M$ be a codimension $2$, orientable, $C^2$ spacelike submanifold and $\eta:[0,b] \to M$ with $x:=\eta(0) \in S$ be a future null geodesic that maximizes the distance from $S$. 
Then $\dot{\eta}(0) \perp T_xS$ in the sense that $T_xS + \R \dot{\eta}(0)= \ker g_{\dot{\eta}(0)}(\dot{\eta}(0),\cdot)$.
Moreover, under the normalization $\dot{\eta}(0)=V^+(x)$ or $V^-(x)$, if $\theta^\pm(x) < - \beta<0$ and $\Ric(\dot \eta) \geq - \delta$ with $0 < \delta \leq \frac{3 \beta}{b}(1-r)$ for some $0 < r < 1$, then we have $b \leq \frac{n-2}{r\beta}$.
\end{Proposition}

\begin{proof}
This is shown in the same way as Proposition \ref{Prop: Hawking quantitative focusing}.
Precisely, we have
\[
\ker g_{\dot{\eta}(0)}\bigl(\dot{\eta}(0),\cdot\bigr)
=T_xS + \R \dot{\eta}(0)
=H^+_x \text{ or } H^-_x,
\]
which implies that $\dot{\eta}(0)=cV^+(x)$ or $cV^-(x)$ for some $c>0$.
We set $c=1$ by reparametrizing $\eta$.
Then, using the null Raychaudhuri equation (\cite[Eq.\ (31)]{MinguzziRaychaudhuri} with $\kappa=0$, where $\dim M=n+1$), we obtain the claimed estimate of $b$.
\end{proof}

\begin{Theorem}[$C^{(1,3)}$ Penrose singularity theorem]
\label{theorem: unweighted penrose}
Let $(M,L)$ be a globally hyperbolic $C^{(1,3)}$ Finsler spacetime satisfying the distributional null energy condition and containing a $C^2$ trapped surface $S$ and a noncompact Cauchy hypersurface. 
Then there exists a future inextendible null geodesic $\eta$ emanating orthogonally from $S$ that is future incomplete.
\end{Theorem}

\begin{proof}
By way of contradiction, suppose that every future null geodesic starting orthogonally to $S$ is future complete. 
Fix future $L$-null normals $V^\pm$ of $S$ as above, which are of regularity $C^1$.
They are the uniform limits of future $\widecheck L_\varepsilon$-null normals $V^\pm_\varepsilon$ of $S$ (note that $S$ is also $\widecheck{L}_\eps$-spacelike). 
We write $\theta^\pm_\varepsilon$ for the null expansions of $V^\pm_\eps$ with respect to $\widecheck L_\varepsilon$. 
Consider the set
\[
\mathcal{K} :=\biggl\{ \lambda v : v=V^\pm(x),\, x \in S,\, 0 \le \lambda \le \frac{2(n-2)}{-\theta^\pm(\pi(v))} \biggr\}.
\]
This is a compact subset of $TM$ (since $S$ is compact by definition). 
We claim that
\begin{equation*}
     J^+(S) \setminus I^+(S) \subseteq F_{\mathcal K,1} :=\{ \eta(t) : \eta\ \text{is an $L$-geodesic},\, \dot \eta(0) \in \mathcal K,\, t \in [0,1]\}.
\end{equation*}
Note that $F_{\mathcal{K},1}$ is compact (thus so is $J^+(S) \setminus I^+(S)$ by global hyperbolicity if we show the claim). 

Suppose to the contrary that there is some $y \in J^+(S) \setminus I^+(S)$ not in $F_{\mathcal K , 1}$.
Since clearly $S \subseteq F_{\mathcal{K},1}$, necessarily $y \not\in S$.
According to Corollary \ref{Corollary: Trappedsurfacemaxaslimit}, any such $y$ is reached by a maximizing $L$-null geodesic $\eta:[0,b] \to M$ from $S$ that is the $C^2$-limit of maximizing $\widecheck{L}_{\eps_m}$-null geodesics $\eta_m:[0,b_m] \to M$ from $S$ with $\lim_{m \to \infty}\eps_m =0$, where $\eta$ is parametrized as $\dot{\eta}(0)=V^+(\eta(0))$ or $V^-(\eta(0))$, and we assume $\dot{\eta}(0)=V^+(\eta(0))$ without loss of generality, as well as $\dot{\eta}_m(0)=V^+_{\eps_m}(\eta_m(0))$.

By $y \not\in F_{\mathcal{K},1}$, we have $b>-2(n-2)/\theta^+(\eta(0))$, so that
\[
\theta^+\bigl( \eta(0) \bigr) <-\beta :=-\frac{2(n-2)}{b}.
\]
Then it follows from the $C^2$-convergence that, for sufficiently large $m$, $\theta^+_{\eps_m}(\eta_m(0)) <-\beta$ and $b_m > 3b/4$.
Moreover, we can assume $\Ric[\widecheck L_{\eps_m}](\dot \eta_m) \geq - \beta/b_m$ by virtue of Proposition \ref{Prop: Null approximate Ricci}.
Hence, from Proposition \ref{Prop: Null quantitative focusing} with $b=b_m$, $\delta=\beta/b_m$ and $r=2/3$, we infer that
\[
b_m \leq \frac{1}{r} \frac{n-2}{\beta} =\frac{3}{2} \frac{b}{2} <b_m,
\]
a contradiction.

From here, one may follow the smooth argument given for \cite[Thm.\ 8.4]{LuMinguzziOhtaSingularity} or \cite[Thm.\ 6.25]{Minguzzicausality} to conclude (keeping in mind Lemma \ref{Lemma: lowregFinsler cone structure}).
\end{proof}

\subsection{Weighted versions of the singularity theorems}
\label{Subsection: weighted singularity theorems}

In this final subsection, we give weighted versions of the Hawking and Penrose singularity theorems.
Our results are new even for $C^1$ Lorentzian metrics.

In the weighted case, we need to modify the notion of geodesic (in)completeness.
For this purpose, let us recall the following notion introduced in \cite{LuMinguzziOhtaSingularity} (called $\epsilon$-range; note that $\dim M=n+1$ there), inspired by Wylie and Woolgar's works in Riemannian and Lorentzian settings \cite{WW16,WW18,Wylie17}.

\begin{Definition}[$\kappa$-ranges]
For $N \in (-\infty,1] \cup [n,\infty]$, where $n=\dim M$, we define the \emph{timelike $\kappa$-range of $N$} to be those $\kappa \in \R$ which satisfy
\begin{equation*}
\begin{cases}
    \kappa = 0, & N = 1,\\
    |\kappa| < \sqrt{\frac{N-1}{N - n}}, & N \neq 1,n, \\
    \kappa \in \R, & N=n.
\end{cases}
\end{equation*}
Similarly, for $N \in (-\infty,2] \cup [n,\infty]$ (provided $n \ge 3$), we define the \emph{null $\kappa$-range of $N$} to be those $\kappa \in \R$ which satisfy
\begin{equation*}
\begin{cases}
    \kappa = 0, & N = 2,\\
    |\kappa| < \sqrt{\frac{N -2}{N - n}}, & N \neq 2,n, \\
    \kappa \in \R, & N=n.
\end{cases}
\end{equation*}
In both cases, we also say that $\kappa \in \R$ is \emph{timelike/null admissible for $N$} if it belongs to the corresponding $\kappa$-range of $N$.
\end{Definition}

Let $N \in (-\infty,1] \cup [n,\infty]$ and $\kappa$ be timelike admissible for $N$.
Given a $C^2$ spacelike hypersurface $\Sigma \subseteq M$, let $V$ be the future unit normal vector field on a neighborhood of $\Sigma$ as in Subsection \ref{ssc:Hawking}, and $\Psi$ be the distortion function associated with $\mathfrak m$ as in Subsection \ref{Subsection: weights}. 
Then, for an integral curve $\eta:[0,b] \to M$ of $V$ with $\eta(0) \in \Sigma$, we define the weighted counterparts to the expansion and shear of $V$ by
\begin{align*}
    \theta_\kappa(t) &:= \exp\biggl( \frac{2(1-\kappa)}{n-1} \Psi\bigl(\dot{\eta}(t)\bigr) \biggr) \Bigl(\theta(t) - S\Psi\bigl(\dot{\eta}(t)\bigr)\Bigr), \\
    \sigma_\kappa(t) &:= \exp\biggl( \frac{4(1-\kappa)}{n-1} \Psi\bigl(\dot{\eta}(t)\bigr) \biggr) \sigma(t),
\end{align*}
where $\theta$ and $\sigma$ are the unweighted expansion and shear defined in \eqref{eq:expan}.
We say that $\Sigma$ is \emph{$(\kappa,\beta)$-contracting} if $\theta_\kappa < -\beta < 0$ for all $\eta$ as above, where $\beta > 0$. 
Along $\eta$, it will be convenient to consider the parametrization
\begin{equation*}
\varphi_\eta(t) :=\int_0^t \exp\biggl( \frac{2(\kappa -1)}{n-1} \Psi\bigl(\dot{\eta}(s)\bigr) \biggr) \,\diff s.
\end{equation*}

When $L$ is smooth, the \emph{weighted Raychaudhuri inequality} asserts the following:
\begin{equation}\label{eq:wRay}
(\theta_\kappa \circ \varphi_\eta^{-1})' \le -\Ric^N_{\mathfrak m} \bigl((\eta \circ \varphi_{\eta}^{-1})' \big) - \sigma_{\kappa}(\varphi_{\eta}^{-1}) -c_{N,\kappa}\theta_{\kappa}(\varphi_{\eta}^{-1})^2
\end{equation}
on $\varphi_\eta([0,b])$ (see \cite[Prop.\ 5.7]{LuMinguzziOhtaSingularity}), where we set
\begin{equation*}
c_{N,\kappa} :=\frac{1}{n-1}\biggl( 1-\kappa^2 \frac{N-n}{N-1} \biggr) \,\ (N \neq 1),
\quad c_{\infty,\kappa}:=\frac{1-\kappa^2}{n-1},
\quad c_{1,0}:=\frac{1}{n-1}.
\end{equation*}
With this, we can generalize Proposition \ref{Prop: Hawking quantitative focusing} as follows.

\begin{Proposition}[Weighted timelike quantitative focusing]
\label{Prop: weighted timelike quant focusing}
Let $(M,L,\mathfrak m)$ be a smooth weighted Finsler spacetime, $\Sigma$ be a $C^2$ spacelike hypersurface and $\eta:[0,b]\to M$ with
$x := \eta(0) \in \Sigma$ be a future unit speed timelike geodesic that maximizes the distance from $\Sigma$. 
Let $N\in(-\infty,1]\cup[n,\infty]$, $\kappa$ be timelike admissible for $N$, and set $B:=\varphi_\eta(b)$. 
Assume that $\theta_\kappa(0)<-\beta<0$ and
\begin{equation}
\label{eq:Ric_N>delta}
    \Ric^N_{\mathfrak m}(\dot\eta)
    \geq
    -\delta\exp\left(
        \frac{4(\kappa-1)}{n-1}\Psi(\dot\eta)
    \right),
\end{equation}
where
\[
    0<\delta\leq\frac{3\beta}{B}(1-r)
    \qquad\text{for some }0<r<1.
\]
Then we have 
\[
    B\leq\frac1{c_{N,\kappa}r\beta}.
\]
When $N\in[n,\infty)$, we can take $\kappa=1$ and obtain
$b\leq(N-1)/(r\beta)$ under
$\operatorname{Ric}^N_{\mathfrak m}(\dot\eta)\geq-\delta$.
\end{Proposition}

\begin{proof}
By the positive $2$-homogeneity of $\Ric^N_\m$ and \eqref{eq:Ric_N>delta}, we have 
\[
\begin{aligned}
    &\Ric^N_{\mathfrak m}
    \left(
        (\eta\circ\varphi_\eta^{-1})'
        (\varphi_\eta(t))
    \right)=
    \exp\left(
        \frac{4(1-\kappa)}{n-1}\Psi\bigl(\dot\eta(t)\bigr)
    \right)
    \Ric^N_{\mathfrak m}\bigl(\dot\eta(t)\bigr) \ge -\delta.
\end{aligned}
\]
Thus upon setting
$\Theta:=\theta_\kappa\circ\varphi_\eta^{-1}$, we get from \eqref{eq:wRay} that
\[
    \Theta'\leq\delta-c_{N,\kappa}\Theta^2
    \quad\text{on }[0,B],
    \qquad
    \Theta(0)<-\beta.
\]
The same scalar comparison as in the proof of Proposition \ref{Prop: Hawking quantitative focusing}, with interval length $B$ and coefficient
$c_{N,\kappa}$ gives the claim. 
\end{proof}

For a future-inextendible unit speed timelike geodesic $\eta:[0,T)\to M$, write $\varphi_\eta(T) :=\lim_{t\uparrow T}\varphi_\eta(t)$.
We say that $\eta$ is \emph{future $\kappa$-incomplete} if $\varphi_\eta(T)<\infty$. 
Note that the $1$-completeness boils down to the usual geodesic completeness, while $\kappa=1$ is admissible only for $N \in [n,\infty)$. 
We refer to \cite{WW16,WW18} for the case of $\kappa=0$.

In the arguments below, $\Psi_\varepsilon$ denotes the distortion associated with $(\widecheck L_\varepsilon,\mathfrak m_\varepsilon)$. 
By our approximation results (recall $\Psi \in C^{(1,0)}$), we have
\begin{equation*}
    \Psi_\varepsilon\to\Psi,
    \qquad
    \widecheck S_\varepsilon\Psi_\varepsilon\to S\Psi \qquad \text{in } C^0_{\mathrm{loc}}(TM\setminus \mathbf 0).
\end{equation*}

\begin{Theorem}[Weighted $C^{(1,2)}$ Hawking singularity theorem, version 1]
\label{theorem: weighted hawking 1}
Let $(M,L)$ be a globally hyperbolic $C^{(1,2)}$ Finsler spacetime, equipped with a $C^1$ measure $\mathfrak m$, satisfying 
\[
\Ric^N_\m \geq K\exp\biggl( \frac{4(\kappa-1)}{n-1}\Psi \biggr)
\]
timelike distributionally for some $K\in\R$, $N \in (-\infty,1] \cup [n,\infty]$ and $\kappa$ timelike admissible for $N$.
Suppose that $\Sigma\subseteq M$ is a $C^2$ spacelike Cauchy hypersurface with $\theta_\kappa <-\beta$,
where $\beta>0$ and, additionally, $\beta>\sqrt{|K|/c_{N,\kappa}}$
when $K<0$.
Then we have
\[
\sup_\eta \lim_{t \uparrow T} \varphi_\eta(t) \leq T^{(N,\kappa)}_{K,\beta},
\]
where $\eta:[0,T) \to M$ runs over all future unit speed timelike geodesics maximizing the distance from $\Sigma$, and we set 
\[
    T^{(N,\kappa)}_{K,\beta}
    :=
    \int_\beta^\infty
        \frac{\diff z}{K+c_{N,\kappa}z^2}
    =
    \begin{cases}
        \dfrac{1}{c_{N,\kappa}\beta},
        & K=0,\\[1.2ex]
        \dfrac{1}{\sqrt{c_{N,\kappa}K}}\,
        \arctan\!\left(
            \dfrac{1}{\beta}\sqrt{\dfrac{K}{c_{N,\kappa}}}
        \right),
        & K>0,\\[1.6ex]
        \sqrt{\dfrac{1}{c_{N,\kappa}|K|}}\,
        \operatorname{arctanh}\!\left(
            \dfrac{1}{\beta}\sqrt{\dfrac{|K|}{c_{N,\kappa}}}
        \right),
        & K<0.
    \end{cases}
\]
In particular, every such $\eta$ is future $\kappa$-incomplete and, when $N \in [n,\infty)$ and $\kappa=1$, we have $\sup_M d_\Sigma \le T^{(N,1)}_{K,\beta}$ and $(M,L)$ is future timelike geodesically incomplete.
\end{Theorem}

\begin{proof}
Given $\eta$ as in the statement, we first generalize Proposition \ref{Prop: quantitative focusing for general bounds} assuming that $L$ and $\m$ are smooth:
Under the hypotheses $\theta_\kappa <-\beta$,
\[
\Ric^N_\m (\dot{\eta}) \ge (K-\delta) \exp\left(
        \frac{4(\kappa-1)}{n-1}\Psi(\dot\eta)
    \right)
\]
and $0<\delta \le (1-r)(K+c_{N,\kappa}\beta^2)$, we obtain
$\varphi_\eta(T) \leq T^{(N,\kappa)}_{K,\beta}/r$.
Note that, as in the proof of Proposition \ref{Prop: weighted timelike quant focusing}, the above curvature bound implies
\[
\Ric^N_\m \left(
        (\eta\circ\varphi_\eta^{-1})'
        (\varphi_\eta(t))
    \right) \geq K-\delta.
\]
Thus, the weighted Raychaudhuri inequality \eqref{eq:wRay} gives
\[
    \Theta' \leq -K+\delta-c_{N,\kappa}\Theta^2
    \quad\text{on }[0,\varphi_\eta(T)),
    \quad \text{where}\ \Theta:=\theta_\kappa \circ \varphi_{\eta}^{-1},
\]
which enables us to follow the proof of Proposition \ref{Prop: quantitative focusing for general bounds} by replacing $T_{K,\beta}$ with $T^{(N,\kappa)}_{K,\beta}$.

Having established the above, the proof is now analogous to that of Theorem \ref{Theorem: unweighted general Hawking 1}, using Proposition \ref{Prop: approx weighted energy conditions}\ref{appRN1}.
\end{proof}

\begin{Theorem}[Weighted $C^{(1,2)}$ Hawking singularity theorem, version 2]
\label{Theorem: Weighted Hawking 2}
Let $(M,L,\mathfrak m)$ be a $C^{(1,2)}$ Finsler spacetime equipped with a $C^1$ measure $\mathfrak m$ satisfying $\Ric^N_{\mathfrak m}\geq0$ timelike distributionally for some $N\in(-\infty,1]\cup[n,\infty]$.
Suppose that $\Sigma\subseteq M$ is a compact $C^2$ spacelike hypersurface such that $\theta_\kappa<0$ for some $\kappa$ which is timelike admissible for $N$.

Then there exists a future inextendible unit speed timelike geodesic $\eta:[0,T)\to M$ that starts orthogonally to $\Sigma$ and is future $\kappa$-incomplete.
In particular, when $N\in[n,\infty)$ and $\kappa=1$, this geodesic is future incomplete in the usual sense.
\end{Theorem}

\begin{proof}
The proof is similar to that of Theorem \ref{Theorem: Hawking version 2}. 
Assume first that $N\ne n$ and suppose, by contradiction, that every future inextendible timelike geodesic starting orthogonally to $\Sigma$ is future $\kappa$-complete.
As in the proof of Theorem \ref{Theorem: Hawking version 2}, up to passing to a covering, we may assume that $\Sigma$ is a compact acausal partial Cauchy hypersurface. 
Since $\widecheck L_{\eps_m} \prec L$, $\Sigma$ is also a partial Cauchy hypersurface for $\widecheck L_{\eps_m}$.
Thus, by \cite[Thm.\ 8.11]{LuMinguzziOhtaSingularity}, there exist future timelike $\Sigma$-rays $\eta_m$ for $\widecheck L_{\varepsilon_m}$, where $\varepsilon_m \to 0$. 
These are timelike geodesics starting orthogonally to $\Sigma$, which we parametrize by $\widecheck L_{\varepsilon_m}$-unit speed.

As in Theorem \ref{Theorem: Hawking version 2}, their initial velocities lie in a common compact set. 
Applying the compactness argument for the geodesic equation from the proof of Proposition \ref{Prop: Existence time for Leps geodesics} along a compact exhaustion of $TM\setminus \mathbf 0$, we obtain a timelike $\Sigma$-ray $\eta:[0,T)\to M$ with respect to $L$, starting orthogonally to $\Sigma$.
For every $b<T$, the geodesics $\eta_m$ are defined on $[0,b]$ for all sufficiently large $m$ and converge to $\eta$ there in $C^2$.

By our contradiction assumption, $\eta$ is future $\kappa$-complete. 
Denote by $\theta_\kappa$ and $\theta_{\kappa,\varepsilon}$ the initial weighted expansions of $\Sigma$ for $(L,\mathfrak m)$ and $(\widecheck L_\varepsilon,\mathfrak m_\varepsilon)$, respectively.
Since $\theta_\kappa<0$, by the compactness of $\Sigma$ and the uniform convergence of the initial weighted expansions, we have
\[
    \beta:=-\frac12\max_\Sigma\theta_\kappa>0,
\]
so that $\theta_{\kappa,\eps}<-\beta$ on $\Sigma$ for sufficiently small $\eps$.

Put $c:=c_{N,\kappa}>0$ and $B:=4/(c\beta)$.
Since $\eta$ is future $\kappa$-complete (i.e., $\varphi_\eta(T)=\infty$), there exists $b<T$ such that $\varphi_\eta(b)=B$.
The $C^2$-convergence on $[0,b]$ implies that the velocities of
these segments lie in a common compact subset of $\Omega^T$ and $B_m:=\varphi_{\eta_m}(b)\to B$,
where $\varphi_{\eta_m}$ is computed using $\Psi_{\varepsilon_m}$.

Fix $0<\delta<3\beta/(4B)$. The functions
$\Psi_{\varepsilon_m}(\dot\eta_m)$ are uniformly bounded on $[0,b]$.
Consequently, Proposition \ref{Prop: approx weighted energy conditions}\ref{appRN1} gives
\[
    \operatorname{Ric}^{N}_{\mathfrak m_{\varepsilon_m}}
    [\widecheck L_{\varepsilon_m}](\dot\eta_m)
    \ge
    -\delta\exp\left(
        \frac{4(\kappa-1)}{n-1}
        \Psi_{\varepsilon_m}(\dot\eta_m)
    \right)
\]
on $[0,b]$ for all sufficiently large $m$.
Moreover,
$\delta<3\beta/(2B_m)$ for large $m$ as well.
Since every $\eta_m|_{[0,b]}$ maximizes the $\widecheck L_{\varepsilon_m}$-distance from $\Sigma$,
Proposition \ref{Prop: weighted timelike quant focusing} with $r=1/2$ therefore yields
\[
    B_m\le\frac{2}{c\beta}=\frac{B}{2}.
\]
Letting $m\to\infty$ gives $B\le B/2$, a contradiction.

When $N=n$ with $\theta_\kappa<0$, we have $S\Psi \equiv 0$ by convention, and hence $\Ric_\m^{N'} =\Ric_\m^n =\Ric$ for every $N'>n$.
Thus, for sufficiently small $N'>n$, $\Ric_\m^{N'} \ge 0$ holds and $\kappa$ is timelike admissible for $N'$.
Then the above argument yields the desired future timelike $\kappa$-incompleteness.

Finally, for $N\in[ n,\infty)$ and $\kappa=1$, since $\varphi_\eta(t)=t$, the $1$-completeness coincides with the ordinary geodesic completeness. 
\end{proof}

For a codimension $2$, orientable, $C^2$ spacelike submanifold $S\subseteq M$, with $n\geq3$, define the weighted null expansions on $S$ by
\[
    \theta_\kappa^\pm
    :=
    \exp\left(
        \frac{2(1-\kappa)}{n-2}\Psi(V^\pm)
    \right)
    \bigl(\theta^\pm-S\Psi(V^\pm)\bigr),
\]
where $\theta^\pm$ are the unweighted null expansions introduced in Subsection \ref{sc:Penrose}.
For a future inextendible null geodesic $\eta:[0,T)\to M$, set
\[
    \varphi_\eta^0(t)
    :=
    \int_0^t
    \exp\left(
        \frac{2(\kappa-1)}{n-2}\Psi\bigl(\dot\eta(s)\bigr)
    \right)\,\diff s,
    \qquad
    \varphi_\eta^0(T)
    :=
    \lim_{t\uparrow T}\varphi_\eta^0(t).
\]
We say that $\eta$ is \emph{future $\kappa$-incomplete} if $\varphi_\eta^0(T)<\infty$.

For $\kappa$ null admissible for $N$, we put
\[
    \bar c_{N,\kappa}
    :=
    \frac1{n-2}
    \left(1-\kappa^2\frac{N-n}{N-2}\right) \,\ (N\ne2),
    \quad \bar c_{\infty,\kappa}:=\frac{1-\kappa^2}{n-2},
    \quad \bar c_{2,0}:=\frac1{n-2}.
\]
The null weighted Raychaudhuri inequality \cite[Prop.\ 6.2]{LuMinguzziOhtaSingularity} gives the analog of Proposition \ref{Prop: weighted timelike quant focusing} for null maximizers from $S$, with $\varphi_\eta^0$ and $\bar c_{N,\kappa}$ in place of $\varphi_\eta$ and $c_{N,\kappa}$.

\begin{Theorem}[Weighted $C^{(1,2)}$ Penrose singularity theorem]
\label{theorem: weighted penrose}
Let $(M,L,\mathfrak m)$ be a globally hyperbolic $C^{(1,2)}$ Finsler spacetime of dimension $n\geq3$, equipped with a $C^1$ measure  $\mathfrak m$, satisfying the distributional null $N$-energy condition for some $N\in(-\infty,2]\cup[n,\infty]$.
Suppose that $M$ contains a noncompact Cauchy hypersurface and a compact orientable $C^2$ spacelike submanifold $S$ of codimension $2$ satisfying $\theta_\kappa^\pm<0$ for some $\kappa$ which is null admissible for $N$. 

Then there exists a future inextendible null geodesic $\eta:[0,T)\to M$, starting orthogonally to $S$, which is future $\kappa$-incomplete.
In particular, if $N\in[n,\infty)$ and $\kappa=1$, $\eta$ is future incomplete in the usual sense.
\end{Theorem}

\begin{proof}
Assume first that $N\ne n$. 
We can follow the same lines as Theorem \ref{theorem: unweighted penrose} by modifying $\mathcal{K}$ in its proof to
\[
\mathcal{K} :=\biggl\{ \lambda v : v=V^\pm(x),\, \dot \eta(0) = v, \, x \in S,\, 0 \le \varphi^0_\eta(\lambda) \le \frac{2}{-c\theta_\kappa^\pm(\pi(v))} \biggr\},
\]
where $\eta$ runs among geodesics and we set $c:=\bar{c}_{N,\kappa}$.
Indeed, to show $J^+(S) \setminus I^+(S) \subseteq F_{\mathcal{K},1}$, assuming $y \not\in F_{\mathcal{K},1}$ we have
\[
\theta^+_\kappa \bigl(\eta(0)\bigr)
<-\beta :=-\frac{2}{c\varphi^0_\eta(b)}.
\]
Then, under $\theta^+_{\kappa,\eps_m}(\eta_m(0))<-\beta$, $\varphi^0_{\eta_m}(b_m) >3\varphi^0_\eta(b)/4$ and $\Ric[\widecheck L_{\eps_m}](\dot{\eta}_m) \ge -\beta/\varphi^0_{\eta_m}(b_m)$, applying the weighted version of Proposition \ref{Prop: Null quantitative focusing} with $r=2/3$, we have
\[
\varphi^0_{\eta_m}(b_m) \le \frac{1}{r} \frac{1}{c\beta} =\frac{3}{2} \frac{\varphi^0_\eta(b)}{2} <\varphi^0_{\eta_m}(b_m), \]
a contradiction.

The case $N = n$ holds as argued in Theorem \ref{Theorem: Weighted Hawking 2}.
\end{proof}

We remark that our weighted singularity theorems, i.e., Theorems \ref{theorem: weighted hawking 1}, \ref{Theorem: Weighted Hawking 2} and \ref{theorem: weighted penrose}, are new even in the case $L(v) = g(v,v)/2$, where $g$ is a $C^1$ Lorentzian metric, and provide an extension of Graf's results \cite{Graf-SingularityC1}.

\begin{Remark}[Past versions of singularity theorems]
\label{remark: past versions of singularity theorems}
All of our unweighted (Theorems \ref{Theorem: unweighted Hawking 1}, \ref{Theorem: Hawking version 2}, \ref{Theorem: unweighted general Hawking 1} and \ref{theorem: unweighted penrose}) and weighted (Theorems \ref{theorem: weighted hawking 1}, \ref{Theorem: Weighted Hawking 2} and \ref{theorem: weighted penrose}) singularity theorems can be formulated to yield past timelike/null ($\kappa$-)incompleteness resp.\ a bound on the past Lorentz--Finsler distance.
The only difference is that the assumption on the timelike/null expansion needs to be reversed as in Corollary \ref{cr:past-sing}. 
The proofs all follow from the corresponding future versions applied to the reverse structure $\overleftarrow{L}(v):=L(-v)$, since the weighted and unweighted distributional timelike and null Ricci bounds are satisfied by $L$ if and only if they are satisfied by $\overleftarrow{L}$.
See \cite[Rem.\ 8.14]{LuMinguzziOhtaSingularity} for more details.
\end{Remark}

\section*{Data availability statement}

There is no data associated with this manuscript.

\section*{Conflict of interest statement}

The authors confirm that no conflict of interest exists for this work.

\section*{Use of AI declaration}

During the final stages of this work, parts of the manuscript were refined through discussions with large language models (LLMs).
The main arguments and techniques presented herein were developed independently by the authors, and LLMs were used solely to verify these results.

\section*{Acknowledgments}

The authors would like to thank the Bernoulli Centre at EPFL for the hospitality during the workshop ``Optimal Transport and Metric Geometry across Structures" in October 2025, and the Faculty of Mathematics of the University of Vienna during the workshop ``Curvature, Optimal Transport and Gravity" in September 2026.

SO is supported by the JSPS Grant-in-Aid for Scientific Research (KAKENHI) 22H04942, 24K00523, 24K21511, 26H01996, and by the JST CREST JPMJCR25Q2. He is also grateful to Universit\"at Bonn for its hospitality during his visit in Summer 2026.
This research was funded in part by the Austrian Science Fund (FWF) [Grants DOI \href{https://doi.org/10.55776/EFP6}{10.55776/EFP6}, \href{https://doi.org/10.55776/STA32}{10.55776/STA32}, \href{https://doi.org/10.55776/PAT1996423}{10.55776/PAT1996423} and \href{https://doi.org/10.55776/J4913}{10.55776/J4913}]. For open access purposes, the authors have applied a CC BY public copyright license to any author accepted manuscript version arising from this submission.

\addcontentsline{toc}{section}{References}
\printbibliography

 \end{document}